\documentclass[,reqno]{amsart}
\usepackage{amssymb}
\usepackage{amsmath}
\usepackage{latexsym,color}
\usepackage{a4wide}

\usepackage{amsmath}
\usepackage{amsthm}
\usepackage{amssymb}
\usepackage{cancel}
\usepackage{color}
\usepackage{calrsfs}
\usepackage[pdfborder={0 0 0}, colorlinks=true, pdfborderstyle={}, linkcolor=magenta, citecolor=dblue, urlcolor=red]{hyperref}
\usepackage{txfonts}
\usepackage{pxfonts}
\usepackage{epsfig}
\usepackage{psfrag}
\usepackage{subfigure}
\usepackage{graphicx}
\usepackage[mathscr,mathcal]{eucal}
\usepackage{enumerate}
\usepackage{t1enc}
\usepackage[latin2]{inputenc}
\usepackage[normalem]{ulem}
\definecolor{dblue}{rgb}{0.09,0.32,0.44} 

\DeclareMathOperator{\essup}{essup}

\newtheorem{defi}{Definition}[section]

\newtheorem{lemm}[defi]{Lemma}
\newtheorem{thm}[defi]{Theorem}
\newtheorem{remark}[defi]{Remark}
\newtheorem{cor}[defi]{Corollary}
\newtheorem{assum}[defi]{Assumption}
\newtheorem{prop}[defi]{Proposition}

\numberwithin{equation}{section}
\numberwithin{figure}{section}

\renewcommand{\rho}{\varrho}
\renewcommand{\phi}{\varphi}

\newcommand{\bmo}{\operatorname{BMO}}

\newcommand{\be}{\begin{equation}}

	\def\({\left(}
	\def\){\right)}
\begin{document}
			
			\title[Time discretization of QFBSDEs with singular drift]{Time discretization of Quadratic Forward-Backward SDEs with singular drift}
			
			\author{Rhoss Likibi Pellat}
			\address{African Institute for Mathematical Sciences, Ghana}
			\email{rhoss@aims.edu.gh}
			\author{Emmanuel Che Fonka}
			\address{Department of Mathematics and Computer Science, University of Bamenda}
			\email{emmafonka@gmail.com}
			\author{Olivier Menoukeu pamen}
			\address{Institute for Financial and Actuarial Mathematics, Department of Mathematical Sciences, University of Liverpool, L69 7ZL, United Kingdom}
			\address{African Institute for Mathematical Sciences, Ghana}
			\email{menoukeu@liverpool.ac.uk}
			
			\thanks{The project on which this publication is based has been carried out with funding provided by the Alexander von Humboldt Foundation, under the programme financed by the German Federal Ministry of Education and Research entitled German Research Chair No 01DG15010. \\Rhoss Likibi Pellat also acknowledges the financial support for the LMS Grant 52321.}

			\date{}
			
			\keywords{quadratic forward-backward SDEs; singular drift; Zvonkin's transformation; quasi-diffeomorphism flow; time-discretization; Euler-Maruyama scheme.}
			
			\subjclass[2020]{Primary: 65C30,65C99,60H10, 60H07,60H50, Secondary: 35K10,35B65} 
			
			\maketitle
			
			\begin{abstract}
				We investigate the convergence rate for the time discretization of a class of quadratic backward SDEs --potentially involving path-dependent terminal values-- when coupled with non-standard Lipschitz-type forward SDEs. In our review of the explicit time-discretization schemes in the spirit of Pag\`es \& Sagna (see \cite{PaSa18}), we achieve an error control close to $\frac{1}{2}$, even under the modest assumptions considered in this work (see \cite{ChaRichou16}, for comparison).
				
				A central element of our approach is a thorough re-examination of Zhang's $L^2\text{-time regularity}$ of the martingale integrand $Z$ which follows from an extension of the first-order variational regularity for this class of singular forward-backward SDEs with non-uniform Cauchy-Lipschitz drivers. This is complemented by the recently introduced caracterisation of stochastic processes of {\it bounded mean oscillation} (abbreviated as $\bmo$) by K. L\^e (see \cite{Le22}) which we used to derive an $L^p\text{-version}$ of the strong approximation of SDEs with singular drifts from Dareiotis \& Gerencs\'er (see \cite{DaGe20}).
				
				As such, this study addresses a crucial gap in the numerical analysis of forward-backward SDEs (FBSDEs). To our knowledge, for the first time, the impact of regularization by noise on Euler-Maruyama numerical schemes for singular forward SDEs has been successfully transferred to enhance the convergence rate of the discrete time approximations for solutions to backward SDEs.
			\end{abstract} 
			\tableofcontents
			\section{Introduction}
			\subsection{General overview}
			This paper continues the work begun in previous articles \cite{ImkRhOliv24,RhOliv24}, focusing on path regularity and numerical approximation of solutions to quadratic backward SDEs coupled with forward SDEs driven by additive Brownian motion and non Lipschitz-continuous drift coefficients (see also \cite{Rhoss23}).
			
			
			Given a fixed real number  $T> 0$, let $(\Omega,\mathfrak{F},\mathbb{P})$ be a complete probability space on which a $d\text{-dimensional}$ Brownian motion $(W_t)_{t\in [0,T]}$ is defined and let $\{\mathfrak{F}_t\}_{t\in [0,T]}$ be its natural filtration augmented by all $\mathbb{P}\text{-null}$ sets of the sigma-algebra $\mathfrak{F}$. The system of interest in this paper is given by the following (decoupled) forward-backward SDEs 
			\begin{align}
				X_t &= x+  W_t + \int_0^t b(s, X_s)\mathrm{d}s, \label{2.1} \\
				Y_t &= \xi + \int_t^T g(s,X_s,Y_s,Z_s)\mathrm{d}s - \int_t^T Z_s\mathrm{d}W_s \label{2.2}.
			\end{align}
			The existence and uniqueness of solutions to \eqref{2.2} were initiated by Pardoux and Peng (see \cite{PP90}) for nonlinear Lipschitz-continuous generators
			$g$ and square-integrable terminal values $\xi \in L^2(\mathfrak{F}_T)$.
			The solution to such equations is interpreted as an $\{\mathfrak{F}_t\}_{t\geq 0}\text{-adapted}$ pair of processes $(Y_t,Z_t)_{0\leq t\leq T}$ where the value at the maturity time is given by  $Y_T =\xi$. This unique feature of BSDEs has proven to be an essential mathematical tool for understanding a wide range of complex systems where present dynamics are influenced by both past and future information. Consequently, the theory of BSDEs has garnered significant interest over the last three decades and remains an active area of research due to its various applications in optimal control, finance, insurance, economics, engineering and social sciences (see, for instance, \cite{el97, CIKHN04,Tian23}, and references therein).
			
			
			Of particular interest is the regularity of solutions to equation \eqref{2.2}. More specifically, the focus is on investigating whether the processes $(Y_t)_{t\in[0,T]}$ and $(Z_t)_{t\in[0,T]}$ possess smoothness properties, such as regularity of sample paths and classical or Malliavin differentiability, for all $t \in [0,T]$. For instance, when approximating the solutions to BSDE \eqref{2.2} numerically, these properties are critical and heavily rely on the smoothness of the coefficients $(g,\xi)$ involved in \eqref{2.2}. However, a deeper analysis reveals that the noise $(W_t)_{t\geq 0}$ and the regularity of the coefficients of the forward equation \eqref{2.1} cannot be overlooked and rather play an essential role in this investigation. 
			
			To our knowledge, the first result addressing the regularity of BSDEs was obtained in \cite{ParPe92} under quite restrictive (strong) assumptions on the coefficients. Additionally, in the context of fully coupled FBSDEs, similar regularity results were achieved in several relevant papers \cite{MPY94, Delarue, Delarue2,DelarueMenozzi06,RhOlivOuk} via the so-called {\it four-step scheme}, which provides a more general Feynman-Kac representation of the control process $Z_t$ as the "derivative" in some sense of the continuous solution to the associated quasi-linear parabolic PDE. Numerous other papers have emerged in recent years on the smoothness of the control process $Z$ under much weaker assumptions on the coefficients.  
			We refer for instance to the groundbreaking works \cite{MaZhang01,MaZhang02,Zhang041}, in which the authors respectively introduced the notion of $L^2\text{-time}$ regularity and a representation theorem for the control $Z$ under Lipschitz-continuous conditions.
			
			
			As mentioned earlier, our primary focus is on the numerical aspects of forward-backward stochastic differential equations (FBSDEs) with quadratic drivers and singular drifts. This particular class of FBSDEs holds significant interest due to its diverse applications in the field of financial mathematics. For instance, this class of FBSDEs occurs in problems involving the valuation and risk management of contingent claims on non-tradable assets, when the dynamics of the underlying asset is governed by an SDE for which the drift coefficient does not satisfy the standard Lipschitz continuous condition(see \cite{RhOliv24}). Due to their complexity and nonlinearity, finding explicit, closed-form solutions to this class of FBSDEs is virtually impossible. Thus, developing effective numerical methods to efficiently approximate their solutions becomes essential. For a comprehensive review of the numerical methods for BSDEs in general, we encourage interested readers to consult the survey by Chessari et al. \cite{Chessarietal23}.
			
			Up to now, the literature has presented only a limited number of comprehensive results regarding the time discretization and numerical approximation of quadratic FBSDEs. Among them, we can refer to \cite{ImkDos}, which, as far as we know, is the first study to address the convergence rate of the error between solutions to quadratic FBSDEs (in the framework of \cite{Kobylansky}) and their approximations obtained through a truncation procedure. Building on the results concerning path regularity and the truncated approximation established in \cite{ImkDos}, the authors in \cite{ChaRichou16} extended the investigation of the BTZ\footnote{Named after the authors Bouchard-Touzi-Zhang, who initially proposed the scheme in the Lipschitz framework; see \cite{BT04}, \cite{Zhang041}} scheme to quadratic drivers. Their work demonstrated a convergence rate for the global error of order less than  ${1/2}$. Specifically, by adopting a uniform time-step over the interval $[0, T]$ with $h:=\sup_i|t_{i+1}-t_i|$, the authors derived the following error estimate, valid for all $\eta > 0$
			\begin{align}
				\mathbb{E}\Big[\sup_{0\leq i\leq N}|Y_{t_i}-\mathcal{Y}_{i}^{\pi}|^{2}\Big] + \mathbb{E}\Big[\sum_{i=1}^{N-1}\int_{t_i}^{t_{i+1}}|Z_{s}-\mathcal{Z}_{i}^{\pi}|^2\mathrm{d}s\Big]\leq Ch^{1-\eta}.
			\end{align}
			More recently, in his PhD thesis, the author of \cite{Zhou} enhanced the convergence rate to $1/2$ in the context of path-dependent terminal conditions and quadratic drivers. This improvement was achieved by introducing an inner modification of the explicit BTZ scheme originally developed by the authors in \cite{PaSa18} for Lipschitz drivers.

			However,  the aforementioned approaches depend strongly on the smoothness of the coefficients in the forward SDEs, creating a notable gap in the development of robust techniques that specifically address the challenges posed by BSDEs when coupled with non smooth-type forward SDEs. In fact, the path-regularity result of the pair solutions $(Y,Z)$ to the BSDE \eqref{2.2} --which is crucial to their analysis-- significantly depends on the existence of a $C^k\text{-diffeomorphic}$ solution to the associated forward equation \eqref{2.1} with $k\geq 1.$  
			These limitations highlight the need for novel methods capable of handling the intricate interactions between the nonlinearities introduced by the quadratic structure and the singularities of the drift term. Bridging this gap would significantly advance the field, enabling more accurate and reliable numerical approximations in complex scenarios where traditional assumptions no longer hold.
			

			
			On the other hand, the field of numerical analysis for stochastic differential equations (SDEs) has seen significant advancements since the pioneering works of \cite{MR1214374,MR71666,MR1335454}. The need to approximate solutions of SDEs with non-regular coefficients has become increasingly essential, driven by the complexity of real-world models in various domains. This development is not merely a mathematical pursuit but a necessity to address critical questions arising for instance in stochastic optimal control and social sciences(see for instance \cite{BMBPD17} and references therein). As many of these models involve irregularities or singularities in their coefficients, accurate and efficient approximation methods are crucial for understanding the behavior of such systems and making reliable predictions.
			
			Among the many contributions in this direction, we can notably highlight the following works: in \cite{PaTa17} the authors prove for the first time the strong convergence of the Euler-Maruyama scheme of an $\alpha\text{-stable}$ L\'evy process for which the drift coefficient is only H\"older continuous; in \cite{LeoSzo18}, the authors establish a strong convergence result of the Euler-Maruyama method for multidimensional SDE with discontinuous drift and degenerate diffusion coefficient; in \cite{DaGe20} the cases of bounded Dini-continuous and measurable drifts were successively circumvent; more recently, in \cite{JourMeno24}, the authors derived a rate of convergence for the case where the drift satisfies the Ladyzhenskaya-Serrin-Prodi condition with $\frac{d}{\rho}+\frac{2}{q}<1$. See also \cite{ButDaGe21} for the convergence rate of SDEs driven by fractional Brownian motions with non-regular drift and \cite{Babi23} for strong rate convergence with unbounded drift coefficient.
			

			Here is the novelty considered in this work: we assume the drift $b$ to be only bounded and belongs to the class of slowly varying function at zero (the drift is at least Lipschitz continuous in \cite{ImkDos,Richou,ChaRichou16,Zhou} and at most H\"older continuous in \cite{ImkRhOliv24}), the generator $g$ has a quadratic growth  in the control variable $Z$ and does not satisfy the common uniform Lipschitz continuous condition in its backward and forward components (the driver is of quadratic type but uniformly Lipschitz in $x$ and $y$ in \cite{ImkDos,Richou,ChaRichou16,Zhou}); as in \cite{Zhang041} and \cite{Zhou}, the terminal value here is of the form $\xi = \Phi(X)$, where $\Phi$ is a bounded functional of the forward process $X$ that satisfy the so-called $L^{\infty}$ or $L^1$ Lipschitz conditions (see Definition \ref{defi 2.8} herein).
			
			It is worth emphasizing that, although the diffusion coefficient $\sigma$ is represented by the identity matrix, one could not retrieve the established results regarding the uniform boundedness of the control process $Z$ nor the nonuniform bound of the type $|Z_t| \leq C(1+|X_t|),$ for all $t\in[0,T]$, where $(X_t)_{t\geq 0}$ stands for the solution to \eqref{2.1}. This arises because the drivers $g$ examined here are not uniformly Lipschitz in their forward and backward components, respectively. Most importantly, this fact prevents us from encountering cases previously underlined or addressed in the literature (\cite{Richou,ChaRichou16,ImkRhOliv24,Zhang041,Zhou}). However, due to the integrability property of BMO martingales (see \cite{Kazamaki}), we rather establish the crucial integrability of the supremum norm of the control $Z$;
			\begin{align*}
				\mathbb{E}[\sup_{0\leq t\leq T}|Z_t|^p] \leq C, \text{ for all } p>p_0',\quad  \frac{1}{p_0}+ \frac{1}{p_0'} =1,
			\end{align*}
			and $p_0$ stands for the integrability power of the  Dol\'eans-Dade exponential of the martingale $\int_0^{\cdot}Z_s\mathrm{d}W_s.$ 
			
			Consequently, by replacing the the drivers $g(\cdot,x.y.z)$ by their Lipschitz approximation $g_n(\cdot,x,\omega_n(y),\bar\omega_n(z))$ as defined for instance by the equation \eqref{truncdriver} below, we observe that the corresponding solution $(Y^n,Z^n)$ deviates from the original one $(Y,Z)$. However, the two remains sufficiently close to one another when measured under the appropriate topology (see Theorem \ref{thmconve1}, herein). This approximation preserves essential properties of the system while mitigating the complexities introduced by the original drivers, making it possible to further considering the machinery of the numerics for Lipschitz-type BSDEs, through the well established BTZ scheme in our general setting.
			
			Under the above assumptions, we first establish the regularity (in both the Malliavin and Sobolev sense) of the solution to the FBSDE system \eqref{2.1}-\eqref{2.2}. These regularity results allow us to revisit the celebrated representation theorem of Ma and Zhang (see \cite{MaZhang02, MaZhang01}) to our setting. This representation of the control $Z_t$ could prove pivotal when considering the application of the Malliavin weights dynamics programming method as an alternative way to efficiently approximate solution to quadratic FBSDEs, particularity in cases of rough drift coefficient. This consideration holds significant promise for extending the work of \cite{GobetTurk16}.
			
			Beyond the singularity of the drift coefficient, it is important to note that the representation theorem in question still depends on the regularity of the functional terminal value(see Theorem \ref{threpr1}). However, when dealing with singular type functional terminal condition, it becomes essential to explore more advanced  concepts such as the generalized fractional smoothness framework introduced in \cite{3G12} and further extended in several other papers as in \cite{GeissYlinen} and references therein. Applying this notion to our broader setting could provide valuable insights into the behavior of solutions under less restrictive conditions. Specifically, it would be highly worthwhile to investigate how singularities propagate over time in the context of BSDEs with quadratic drivers.
			
			We partially address this issue by providing a representation of adapted solutions to quadratic BSDEs as functionals of a diffusion process when the terminal value is a discontinuous discrete functional of the same diffusion process, the drift is Lipschitz continuous and the driver is Lipschitz in $Y$ (see Theorem \ref{prop48}, herein).

			
			We also generalise the path regularity of quadratic BSDEs with path-dependent terminal values. Specifically, we prove that the path regularity bound holds when the drift is bounded and Dini-continuous (see Definition \ref{Dini}) and the driver is quadratic in $z$ and stochastically Lipschitz in $x$ and $y$. Additionally, we establish an error bound for the BTZ scheme under this framework. Despite the poor regularity assumptions of the coefficients, the error bound remains consistent with the results in \cite{Zhou}, indicating robustness in the approximation method under these conditions (see Theorem \ref{Main63} herein). In particular, we obtain the following error bound
			\begin{align}
				\mathbb{E}\Big[\sup_{0\leq i\leq N}|Y_{t_i}-\mathcal{Y}_{i}^{\pi}|^{2}\Big] + \mathbb{E}\Big[\sum_{i=1}^{N-1}\int_{t_i}^{t_{i+1}}|Z_{s}-\mathcal{Z}_{i}^{\pi}|^2\mathrm{d}s\Big]\leq C h^{1-\gamma},
			\end{align}
			provided that
			\begin{align}\label{1.5}
				\mathbb{E}\Big[|\Phi(X)-\Phi(\mathcal{X}^{\pi})|^{8q^{*}}\Big]^{\frac{1}{4q^{*}}}+\mathbb{E}[|X_t- \mathcal{X}^{\pi}_i|^{16q^{*}}]^{\frac{1}{8q^{*}}} \leq Ch^{1-\gamma},
			\end{align}
			holds for any $\eta \in (0,1)$ and $\mathcal{X}^{\pi}_i$ stands for a discrete approximation of the forward process $X$ given by . 
			
			Assuming that the terminal value $\Phi$ is $L^{1}\text{-Lipschitz}$ condition (see Definition \ref{defi}, herein), we highlight the fact that the convergence rate given by \eqref{1.5} is not addressed in the work of Dareiotis and Gerencsér (\cite{DaGe20}), specifically in cases where the drift coefficient is bounded and Dini continuous.  In fact, the authors only provide the following convergence rate in the $L^2(\Omega)\text{-norm}$ (see \cite[Theorem 2.1]{DaGe20})
			\begin{align}\label{1.6}
				\sup_{t\in[0,T]}\mathbb{E}|\mathcal{X}^{\pi}_i-X_t|^2 \leq C h^{1-\gamma}.
			\end{align}
			The key limitation in their approach stems from the pivotal quadrature estimate provided in \cite[Lemma 2.1]{DaGe20}, which relies on explicit moment computations, thereby restricting the estimate to second order accuracy. Building on this observation, the author in \cite{Le22} identifies the error in the quadrature rule for approximating functionals of type $x\mapsto\int_0^t b(s,x+W_s)\mathrm{d}s$ as a {\it bounded mean oscillation} process (see Definition \ref{defi28}, herein), where $b:[0,T]\times\mathbb{R}^d\rightarrow \mathbb{R}$ is  any bounded and measurable function and $\{W_t\}_{t\geq 0}$ stands for a $d\text{-dimensional}$ Brownian motion (see \cite[Example 1.4]{Le22}). Leveraging this characterization, the author applies the John-Nirenberg inequality (see \cite[Theorem 2,3]{Le22}) to propose an enhanced version of the quadrature estimate that can achieve higher orders of accuracy. More precisely, the author obtains:
			\begin{align}\label{1.7}
				\essup_{\omega\in \Omega}\left(\mathbb{E}\sup_{s\leq t \leq 1}\Big|\int_s^t b(r,W_r)-b(r,W_{k_N(r)})\mathrm{d}r\Big|^{2p}\big/\mathfrak{F}_s\right) \leq C(p)(N^{-1}\log(N+1))^p	,
			\end{align}
			for all $p\geq 1$. 
			
			Assuming further that the terminal value $\Phi$ satisfies the $L^{\infty}\text{-Lipschitz}$ condition (see Definition \ref{defi}, herein), we note that a convergence rate of type \eqref{1.5} remains outside the scope of Dareiotis and Gerencsér (\cite{DaGe20}). This is because, in this scenario, a significantly stronger convergence rate-- incorporating the supremum within the expectation-- is essential. More precisely, one needs to establish the following bound
			\begin{align*}
				\mathbb{E}\Big[	\sup_{0\leq t\leq 1}|\mathcal{X}_i^{\pi}-X_t|^{2p}\Big] \leq C h^{(1-\gamma)p},
			\end{align*}
			for all $p\geq 1$. The latter is achieved in this paper as an application of the estimate \eqref{1.7}.
			
			\subsection{Standing Assumption and the BTZ scheme}
			In this section, we present the assumptions under which the main results of this paper are derived. Additionally, we recall the definition of the renowned BTZ scheme. The following set of assumptions will be used in proving the path regularity result. 
			
			\begin{assum}\leavevmode\label{assum2.1}
				\begin{itemize}
					\item[(A1)]
					The drift $b\in L^{\infty}([0,T];C_b(\mathbb{R}^d;\mathbb{R}^d))$. Suppose in addition, for $r_0 \in (0,1)$ there is a Dini function $h$, which is also a slowly varying function at zero, such that for every $x\in \mathbb{R}^d$
					\begin{align}
						|b(t,x)-b(t,y)| \leq h(|x-y|), \text{ for all } y\in B_{r_0}(x), t\in [0,T].
					\end{align}
					Moreover, for all $p\geq 1$ there is a small enough positive real number $\delta=\delta(p)<r_0$ such that the function $F_{\delta}$ defined below, is increasing and concave on $[0,\delta]$
					\[ F_{\delta}(r) = \int_{0<\tau\leq r}\frac{h(\tau)}{\tau}\mathrm{d}\tau +h(r)+2r +\int_{0<\tau\leq \delta}\frac{h(\tau)}{\tau^2}\mathrm{d}\tau,\,\,\, r\in[0,\delta]. \]
					\item[(A2)]  The function $g: [0,T]\times \mathbb{R}^d\times \mathbb{R}\times \mathbb{R}^d \rightarrow \mathbb{R}  $ is measurable and satisfy: $\|g(t,0,0,0)\|_{\infty} \leq \Lambda_0 $ and there exist non negative constants $\Lambda_x,\Lambda_y,\Lambda_z$ and a locally bounded function $\ell\in L_{loc}^{1}(\mathbb{R}_{+})$ such that for all $(t,x,y,z)\in [0,T]\times \mathbb{R}^d\times \mathbb{R}\times \mathbb{R}^d,$ $(t,x',y',z')\in [0,T]\times \mathbb{R}^d\times \mathbb{R}\times \mathbb{R}^d,$ $\alpha_0 \in (0,1)$
					\begin{align*}
						(i) |g(t,x,y,z)-g(t,x',y,z)| &\leq \Lambda_x (1 + |y|+\ell(|y|)|z|^{\alpha_0})|x-x'|, \\
						(ii) |g(t,x,y,z)-g(t,x,y',z')| &\leq \Lambda_y (1 +(|z|^{\alpha_0}+|z'|^{\alpha_0}))|y-y'| \\
						&\quad + \Lambda_z(1+ (\ell(|y|)+\ell(|y'|))(|z|+|z'|)) |z-z'|.
					\end{align*}
					\item[(A3)][Path-dependent functional] $\xi = \Phi(X_{\cdot})$ is either $ L^{\infty}\text{-Lipschitz}$ or $ L^1\text{-Lipschitz}$  i.e., $\Phi$ satisfies \eqref{Lips1} or \eqref{Lips2} and $X$ solves the SDE \eqref{2.1}. 
					
					We emphasise writing $X_{\cdot}$ rather than $X_T$ because the terminal value may depend on the entire path of the process $X$ up to time $T$.
					
				\end{itemize}
			\end{assum}
			To establish the nonlinear Feynman-Kac formula for discrete functional quadratic FBSDEs with non-Lipschitz terminal values, we introduce the following slightly stronger assumptions.
			\begin{assum}\leavevmode\label{assum 4.8}
				\begin{itemize}
					\item[(B1)]
					The drift $b$ is Lipschitz continuous with Lipschitz constant $\Lambda_b$.
					\item[(B2)]  The function $g: [0,T]\times \mathbb{R}^d\times \mathbb{R}\times \mathbb{R}^d \rightarrow \mathbb{R}  $ is measurable and satisfy: $\|g(t,0,0,0)\|_{\infty} \leq \Lambda_0 $ and there exist non negative constants $\Lambda_x,\Lambda_y,\Lambda_z$ and a locally bounded function $\ell\in L_{loc}^{1}(\mathbb{R}_{+})$ such that for all $(t,x,y,z)\in [0,T]\times \mathbb{R}^d\times \mathbb{R}\times \mathbb{R}^d,$ $(t,x',y',z')\in [0,T]\times \mathbb{R}^d\times \mathbb{R}\times \mathbb{R}^d,$ $\alpha_0 \in (0,1)$
					\begin{align*}
						(i) |g(t,x,y,z)-g(t,x',y,z)| &\leq \Lambda_x (1 + |y|+\ell(|y|)|z|^{\alpha_0})|x-x'|, \\
						(ii) |g(t,x,y,z)-g(t,x,y',z')| &\leq \Lambda_y |y-y'|
						+ \Lambda_z(1+ (\ell(|y|)+\ell(|y'|))(|z|+|z'|)) |z-z'|.
					\end{align*}
					\item[(B3)]There are $\mathcal{R}=\{r_0,\cdots,r_L\}$ with $0=r_0 < r_1<\cdots<r_L=T$ and a bounded and measurable function $\phi:(\mathbb{R}^d)^L\rightarrow \mathbb{R}$ such that  $$\xi = \phi(X_{r_1},\cdots,X_{r_L})$$ 
					and $X$ solves the SDE \eqref{2.1}. 
				\end{itemize}
			\end{assum}
			
			To study the numerical approximation of our system of interest, we assume that the driver of the BSDE is H\"older continuous in time. This is summarised below:
			\begin{assum}\leavevmode\label{assum5.1} The coefficients satisfy Assumption \ref{assum2.1} and in particular the generator $g$ satisfies the following growths:
					there exist non negative constants $\Lambda_t,\Lambda_x,\Lambda_y,\Lambda_z$ and a locally bounded and non-decreasing function $\ell$ such that for all $(t,x,y,z)\in [0,T]\times \mathbb{R}^d\times \mathbb{R}\times \mathbb{R}^d,$ $(t,x',y',z')\in [0,T]\times \mathbb{R}^d\times \mathbb{R}\times \mathbb{R}^d,$ $\alpha_0 \in (0,1)$
					\begin{align*}
						(i) |g(t,x,y,z)-g(t',x',y,z)|&\leq \Lambda_t |t-t'|^{\frac{1}{2}} +\Lambda_x(1+|y|+\ell(|y|)|z|^{\alpha_0})|x-x'|,\\
						(ii) |g(t,x,y,z)-g(t,x',y',z')|&\leq  \Lambda_y(1+|z-z'|^{\alpha_0})|y-y'|\\
						&\quad + \Lambda_z(1+ (\ell(|y-y'|))(|z|+|z'|)) |z-z'|,\\
						(iii)|g(t,x,y,z)| \leq &\Lambda_0 +\Lambda_y|y|+ \Lambda_z(|z|+2\ell(|y|)|z|^2).
					\end{align*}
					
			\end{assum}
			A more general type of quadratic drivers introduces an additional layer of complexity addressed in this paper. To efficiently approximate the solution of a quadratic BSDEs, a common approach is to truncate the generator and the analysis then focuses on quantifying the error between the solution obtained through this truncated approach and the solution from the numerical scheme that will be made precise below.	
			
			We consider an explicit time discretization scheme in the spirit of \cite{PaSa18} (see also \cite{Zhang041}), where the conditioning is performed  inside the driver $g$ and given recursively as follows
			\begin{align}\label{BTZ}
				\mathcal{Y}^{\pi}_i &= \mathbb{E}[\mathcal{Y}^{\pi}_{i+1}/ \mathfrak{F}_{t_i} ] + g_n\Big(t_i,\mathcal{X}_i^{\pi},\mathbb{E}[\mathcal{Y}^{\pi}_{i+1}/\mathfrak{F}_{t_i}],\mathcal{Z}^{\pi}_i\Big)\delta t_i ,\notag\\
				\mathcal{Z}^{\pi}_i &= \mathbb{E}[\mathcal{Y}^{\pi}_{i+1}H^{R}_i/\mathfrak{F}_{t_i}]\\
				\mathcal{Y}_N^{\pi} &= \Phi(\mathcal{X}^{\pi}), \quad \mathcal{Z}_N^{\pi} = 0,\notag
			\end{align}
			where the coefficients $(H_i^R)_{0\leq i < N}$ are $\mathbb{R}^{1\times d}$ independent random variables vectors defined by
			\begin{equation}\label{eq79}
				H^R_i = \Big(\frac{R}{\sqrt{\delta t_i}}\wedge \frac{\Delta W_{t_i}^1}{\delta t_i}\vee\frac{-R}{\sqrt{\delta t_i}},\cdots,\frac{R}{\sqrt{\delta t_i}}\wedge \frac{\Delta W_{t_i}^d}{\delta t_i}\vee\frac{-R}{\sqrt{\delta t_i}} \Big)
			\end{equation}
			for any given $R > 0$, the function $g_n: [0,T]\times\mathbb{R}^d\times\mathbb{R}\times\mathbb{R}^d \rightarrow \mathbb{R}$ is defined by 
			\begin{align}\label{truncdriver}
				g_n(t,x,y,z) := g(t,x,\tilde \varpi_n(y),\varpi_n(z)),
			\end{align}
			where $\varpi_n : \mathbb{R}^d\rightarrow \mathbb{R}^d$ is given by $ z\mapsto \varpi_n(z) = (\tilde \varpi_n(z_1),\cdots,\tilde \varpi_n(z_d))$, for $n \in \mathbb{N}$ and $(\tilde \varpi_n)_{n\in \mathbb{N}}$ is a sequence of smooth real valued functions that truncate the identity on the real line.  By construction, it is standard to check that the family of functions $(g_n)_{n\in \mathbb{N}}$ is Lipschitz continuous in $y$ and $z$ with Lipschitz constants that may depend on the level of the truncation $n \in \mathbb{N}$. Moreover, assuming $g$ satisfies Assumption \ref{assum5.1} (H3) then for all $t,x,y,y',z,z'$ we have 
			\begin{align*}
				|g_n(t,x,y,z)-g_n(t,x,y',z')|
				&\leq \Lambda_y(1+|z-z'|^{\alpha_0})|y-y'| + \Lambda_z(n)\ell(|y-y'|))|z-z'|,\\
				|g_n(t,x,y,z)|
				&\leq \Lambda_0 +\Lambda_y|y|+ \Lambda_z(|z|+2\ell(|y|)|z|^2)
			\end{align*}
			%
			Hence, the resulting BSDE given below
			\begin{align}\label{TruncY}
				Y^n_t = \Phi(X) +\int_t^T g_n(s,X_s,Y_s^n,Z_s^n)\mathrm{d}s -\int_t^T Z_s^n\mathrm{d}W_s
			\end{align}
			admits a unique solution $(Y^n,Z^n)\in \mathcal{S}^2\times\mathcal{H}^2$ (see \cite[Lemma 4.3]{ImkRhOliv24}) such that
			\begin{align}\label{77}
				\sup_{n\in\mathbb{N}}\left(\|Y^n\|_{\infty} + \|Z^n*W\|_{\bmo}\right) < \infty.
			\end{align}
			
			

			
			The discrete-time process $\mathcal{X}_i^{\pi}$ may stand for "sampling" of the forward diffusion process $X$ or any  numerical scheme approximating $(X_t)_{t\geq 0}$ with some partition $\pi$. In this paper, we consider a discretization in the spirit of the Euler-Maruyama method, which for instance is provided by
			\begin{align}\label{Euler}
				\mathcal{X}_{i+1}^{\pi} = \mathcal{X}_i^{\pi} + b(t_i,\mathcal{X}_i^{\pi})\delta t_i + \Delta W_{t_i}, \quad \mathcal{X}_0^{\pi} = x,
			\end{align}
			where $\delta t_i = t_{i+1}-t_i$ and $\Delta W_{t_i}= W_{t_{i+1}}- W_{t_i} $ for all $0\leq i\leq N-1$.
			We will define the path 
			\begin{equation}
				\mathcal{X}^{\pi}_t =\mathcal{X}_i^{\pi} \text{ if } t\in [t_i,t_{i+1}),\quad 1\leq i\leq N,
			\end{equation}
			with the convention $[t_N,t_N) = {T}$, so that $\mathcal{X}^{\pi}$ belongs to $\mathbb{D}$.
			

			\subsection{Main results}
			\subsubsection{On path-regularity of solution to  QFBSDEs with path dependent terminal condition} The first main result shows that under weaker conditions on the generator and the drift of the forward equation, the following path regularity result holds bound of the quadratic BSDE with path dependent terminal value.
			\begin{thm}[Path-regularity of QFBSDEs]\label{thm1.1}
				Consider FBSDE \eqref{2.1}--\eqref{2.2} such that $b$,$g$ and $\xi$ satisfy Assumptions \ref{assum2.1}. 
				Let $\pi_1: 0=t_0 <t_1<\cdots < t_N=T$ be a partition and $h= \sup_j|t_{j+1}-t_j|$. Then, for any $p >p_0'> 1$ there is a constant $C$ only depending on $p,d,T$, $\|b\|_{L^{\infty}([0,T];C_b(\mathbb{R}^d))}$ and $\Lambda_{\Phi}$ (independent of the partition) such that  
				\begin{align*}
					\mathbb{E} \big[\sup_{v\in [s,t]}|Y_v-Y_t|^{2p}\big] \leq C |t-s|^p,\\
					\sum_{j=0}^{N-1}\mathbb{E}\Big(\int_{t_j}^{t_{j+1}}|Z_t - Z_{t_j}|^2 \mathrm{d}t \Big)^p \leq C h^p.
				\end{align*}
				where $p_0'$ is the conjugate of $p_0$ and $p_0$ is such that $\mathcal{E}\left(\nabla_z g*W\right) \in L^{p_0}$.
			\end{thm}
			The notations used in the above theorem and the subsequent one are introduced in Section \ref{secpremres}, herein. The strategy to address the problem follows the approach developed in \cite{MaZhang01}. Specifically, we will first prove a version of Theorem \ref{thm1.1} where the terminal value is a functional depending on a discrete set of variables. Then, using density arguments (see Lemma \ref{5lem22}), we will derive the desired result. More precisely, we have
			
			\begin{thm}\label{thm1.2}
				Consider FBSDE \eqref{2.1}--\eqref{2.2} such that $b$ and $g$ satisfy Assumptions \ref{assum2.1} and the terminal value $\xi$ has the following structure
				\begin{itemize}
					\item Let $\pi_0: 0=s_0 <\cdots < s_{r-1}<s_r=T$ be a partition, we set $\xi= \phi(X_{s_0},\cdots,X_{s_r})$, where $\phi: \mathbb{R}^{m\times r} \rightarrow \mathbb{R}$ is bounded and satisfies \eqref{2.6}--\eqref{2.7}. 
					\item Let $\pi_1: 0=t_0 <t_1<\cdots < t_N=T$ be another partition and $h= \sup_j|t_{j+1}-t_j|$.
				\end{itemize}
				Then, for any $p >p_0'> 1$ there is a constant $C$ only depending on $p,d,T$, $\|b\|_{L^{\infty}([0,T];C_b(\mathbb{R}^d))}$ and $S(\phi)$ and not depending on the partitions $\pi_0$ and $\pi_1$ such that  
				\begin{align*}
					\mathbb{E} \big[\sup_{v\in [s,t]}|Y_v-Y_t|^{2p}\big] \leq C |t-s|^p\\
					\sum_{j=0}^{N-1}\mathbb{E}\Big(\int_{t_j}^{t_{j+1}}|Z_t - Z_{t_j}|^2 \mathrm{d}t \Big)^p \leq C h^p.
				\end{align*}
				where $p_0'$ satisfied $1/p_0 + 1/p_0' =1$ and $p_0$ is such that $\mathcal{E}\left(\nabla_z g*W\right) \in L^{p_0}$.
			\end{thm}
			\begin{proof}
				See Subsection \ref{sectprofmain1}.
			\end{proof}	
			We will close this discussion with another important result. It provides an alternative way to control the difference between two $\mathcal{H}_{\bmo}$ processes under an equivalent probability measure, denoted $\mathbb{Q}^{N}$. This result also pertains to the path regularity of the difference between the derivatives (with respect to the starting point of the diffusion process $X$) of the backward component $Y$, and can be viewed as a mere extension of the path regularity result of the process $Z$ stated earlier.
			\begin{prop}\label{cor 1.3}
				Let the assumptions of Theorem \ref{thm1.1} be in force. Let $\{\chi_i\}_{i=0}^{N-1}$ be a sequence of random variables such that $\chi_i \in L^{\infty}(\Omega,\mathfrak{F}_{t+i},\mathbb{P})$, $\mathbb{E}[\chi_i/\mathfrak{F}_{t_i}] =1$ and $\chi_i \in [\epsilon,2-\epsilon]$ for all $0\leq i\leq N-1$ with some $\epsilon \in (0,1)$ independent of $N$. Define the process $\hat{\chi}(t)$ and the probability measure $\mathbb{Q}^N$, respectively by 
				\begin{align*}
					\hat{\chi}(t) =\prod_{t_i\leq t}\chi_i, \quad \frac{\mathrm{d}\mathbb{Q^N}}{\mathrm{d}\mathbb{P}} =  \hat{\chi}(t).
				\end{align*}
				Provided the equivalent probability measure $\mathbb{Q}^N$ satisfies the reverse H\"older inequality, then for any $p>q^{*} $ and $\eta >1$ the following holds:
				\begin{align*}
					\mathbb{E}\Big[\sup_{0\leq i\leq N-1}\mathbb{E}^{\mathbb{Q}^{N}}\Big\{\sum_{j=i}^{N-1}\Big(\int_{t_j}^{t_{j+1}}|Z_s-\hat{Z}_{t_j}|^2\mathrm{d}s\Big)^{\eta}\big/\mathfrak{F}_{t_i}\Big\}^p\Big]\leq Ch^{p\eta},
				\end{align*}
				with \[\hat{Z}_{t_j}:=\frac{1}{\delta t_j} \mathbb{E}\Big[\int_{t_j}^{t_{j+1}} Z_s \mathrm{d}s/\mathfrak{F}_{t_j} \Big].\]
			\end{prop}
			\begin{proof}
				See Subsection \ref{sectprofmain1}.
			\end{proof}	
			\subsubsection{On nonlinear Feynman-Kac formula for discrete functional  QFBSDEs with non-Lipschitz terminal value}
			
			Here we assume that the solution $(Y,Z)$ to the BSDE \eqref{2.2} is realised such that, on $[s_{l-1},s_l),$ we have
			\[ Y_t = u_l(\bar{X}_{l-1};t,X_t) \text{ and } Z_t = \omega_l(\bar{X}_{l-1};t,X_t),  \]
			where $\bar{X}_{l-1}:=(X_{s_1},\cdots,X_{s_{l-1}})$. Let $F_l(\bar{x}_{l-1};\cdot,\cdot):[r_{l-1},r_l]\times\mathbb{R}^d\rightarrow\mathbb{R}$ be defined by
			\begin{align*}
				F_l(x_1,\cdots,x_{l-1};t,x) = F_l(\bar{x}_{l-1};t,x):= \mathbb{E}u_l(x_1,\cdots,x_{r_{l-1}};r_l,X_{r_l}^{t,x}).
			\end{align*}
			Then, the function $F_l$ is the solution to the PDE defined below on the interval $[r_{l-1},r_l)$ for fixed $x_1,\cdots,x_{l-1} \in \mathbb{R}^d.$
			\begin{align}
				\frac{\partial}{\partial t}F_l(\bar{x}_{l-1};t,x) +\frac{1}{2}\nabla_{xx}^2F_l(\bar{x}_{l-1};t,x) +\langle b,\nabla_xF_l(\bar{x}_{l-1};t,x) \rangle =0.
			\end{align}
			The main result of this part is the following, which gives the wellposedness of the functions $u_l$ and $\omega_l$ defined above.
			\begin{thm}[Nonlinear Feynman-Kac formula for discrete functional  QFBSDEs]\label{prop48}
				Let Assumption \ref{assum 4.8} be in force. Then for $l=1,\cdots, L$ there exist measurable $u_l:(\mathbb{R}^d)^{l-1}\times [s_{l-1},s_l)\times \mathbb{R}^d \rightarrow \mathbb{R}$ and $\omega_l:(\mathbb{R}^d)^{l-1}\times [s_{l-1},s_l)\times \mathbb{R}^d \rightarrow \mathbb{R}^{1\times d}$ and Borel sets $D_l \subseteq (\mathbb{R}^{d(l-1)}, l=2,\cdots,L$, such that $D_l^{c}$ is of Lebesgue measure zero, and such that
				\begin{itemize}
					\item[(i)] $u_l(\bar{x}_{l-1};\cdot,\cdot): [s_{l-1},s_l)\times \mathbb{R}^d \rightarrow \mathbb{R}$ is continuously differentiable w.r.t. the space variable with $\nabla_x u_l(\bar{x}_{l-1};t,x)= \omega_l(\bar{x}_{l-1};t,x)$, where $\bar{x}_{l-1} = (x_1,\cdots,x_{l-1})$
					\item[(ii)] there is a constant $C>0$ independent of $l$ such that
					\begin{align*}
						\sup_{t\in[s_{l-1},s_l)}|u_l(\bar{x}_{l-1};t,x)| + \sup_{t\in[s_{l-1},s_l)}\sqrt{s_l-t}|\omega_l(\bar{x}_{l-1};t,x)| \leq C.
					\end{align*}
					\item[(ii)] for all $l=1,\cdots,L,$ $x_1,\cdots,x_{l-1},x \in \mathbb{R}^d$ and $s_{l-1}\leq s< s_l$ the triplet 
					\[  \left( X_t^{s,x}, u_l(\bar{x}_{l-1};t,X_t^{s,x}), \omega_l(\bar{x}_{l-1};t,X_t^{s,x}) \right)_{t\in [s,s_l)}.\]
					solves the FBSDE with generator $g$ and terminal condition 
					\[ u_l(\bar{x}_{l-1};s_l,X_{s_l}^{s,x}),   \]
					where 
					\begin{equation*}
						u_l(\bar{x}_{l-1};s_l,x):=
						\begin{cases}
							u_{l+1}(\bar{x}_{l-1},x;s_l,x)\mathcal{X}_{D_l}(\bar{x}_{l-1}): 2\leq l < L,\\
							\varphi(\bar{x}_{l-1},x)\mathcal{X}_{D_l}(\bar{x}_{l-1}): l = L,
						\end{cases}
					\end{equation*}
					and $u_1(r_1,x):= u_2(x;r_1,x)$.
				\end{itemize} 
			\end{thm}
			
			As well explained in \cite{3G12}, we modify the functional for the backward process 
			$Y$ at each level $l=2,\cdots,L$ on a null set. Since  $\mathbb{P}(X_{r_1}\in D_2\cdots,(X_{r_1},\cdots,X_{r_{l-1}})\in D_l)=1$, this modification does not affect the bounded solution of the BSDE of interest, making Proposition \ref{prop48} well suited to our setting. The proof of Theorem \ref{prop48} follows from the next Lemma, which extends \cite[Lemma A.2]{3Garxiv}.

			\begin{lemm}[Representation of a QBSDE parameterized by a parameter $y\in \mathbb{R}^K$]\label{lemm 4.9}
				Suppose Assumption \ref{assum 4.8} holds and let $H :\mathbb{R}^K\times \mathbb{R}^d\rightarrow \mathbb{R}$ be a bounded and measurable function. Then there is a Borel set $F\subseteq \mathbb{R}^K$ such that $F^{c}$ is of Lebesgue measure zero and such that for 
				\[ G(\bold{y};x):=\mathcal{X}_F(\bold{y})H(\bold{y};x)  \]
				and 
				\begin{equation*}
					U(\bold{y};t,x)=
					\begin{cases}
						Y_t^{\bold{y}:t,x} \text{ a.s.}&:\,\, r\leq t<R\\
						G(\bold{y};x)&:t=R,
					\end{cases}
				\end{equation*}	
				where $(Y_s^{\bold{y}:t,x})_{s\in [t,R]}$ is the $Y\text{-component}$ of the BSDE with respect to the forward $(X_s^{t,x})_{s\in [t,R]}$, the terminal condition $G(\bold{y};X_R^{t,x})$ with terminal time $R\in (0,T]$  and the generator $g$, the following assertions are satisfied
				\begin{itemize}
					\item[(i)] For fixed $\bold{y}\in \mathbb{R}^K$ we have that $U(\bold{y};\cdot,\cdot)\in C^{0,1}([r,R)\times \mathbb{R}^d)$.
					\item[(ii)] The functions $U:\mathbb{R}^K\times[r,R]\times\mathbb{R}^d\rightarrow\mathbb{R}$ and $\nabla_xU:\mathbb{R}^K\times[r,R]\times\mathbb{R}^d\rightarrow\mathbb{R}^{1\times d}$ are measurable.
					\item[(iii)] There is a constant $C$ such that for all $(\bold{y},t,x)\in \mathbb{R}^K \times[r,R]\times\mathbb{R}^d$ the following holds
					\begin{align*}
						\sup_{(\bold{y},t,x)\in \mathbb{R}^K \times[r,R]\times\mathbb{R}^d} \left(	|U(\bold{y};t,x)| + (R-t)^{1/2}|\nabla_x U(\bold{y};t,x)| \right) \leq C.
					\end{align*}
					\item[(iv)] For any $y\in \mathbb{R}^K$, the solution of the BSDE with the terminal value $G(\bold{y};X_R^{t,x})$, generator $g$ and forward diffusion $(X_s^{t,x})_{s\in[r,R]}$ can be represented as 
					\begin{align*}
						Y_t^{\bold{y};t,x} = U(\bold{y};t,X_t^{r,x}),\quad Z_t^{\bold{y};r,x} = \nabla_x U(\bold{y};t,X_t^{r,x}).
					\end{align*}
				\end{itemize}
			\end{lemm}
			\begin{proof}
				See Subsection \ref{profmaindisc}.
			\end{proof}	
			
			\subsubsection*{On the BTZ scheme}
			
			Considering the above BTZ scheme, the main result pertains with the following error bound. 
			\begin{thm}[BTZ scheme error bound]\label{Main63}
				Let $(X,Y,Z)$ be the solution to FBSDE \eqref{2.1}-\eqref{2.2} associated to the parameters $(b,\xi,g)$ satisfying Assumption \ref{assum5.1}. Suppose the partition $\pi : 0=t_0 <\cdots < t_N = T$ with $h:=\sup_{j}(t_{j+1}-t_j) \leq K N^{-1}$ for some $K> 0$. Then,
				\begin{align*}
					&\mathbb{E}\Big[\sup_{0\leq i\leq N}|Y_{t_i}-\mathcal{Y}_{i}^{\pi}|^{2}\Big] + \mathbb{E}\Big[\sum_{i=1}^{N-1}\int_{t_i}^{t_{i+1}}|Z_{s}-\mathcal{Z}_{i}^{\pi}|^2\mathrm{d}s\Big]\\
					&\leq C\Big(\mathbb{E}\Big[|\Phi(X)-\Phi(\mathcal{X}^{\pi})|^{8q^{*}}\Big]^{\frac{1}{4q^{*}}}+ h + \mathbb{E}[|X_t- \mathcal{X}^{\pi}_i|^{16q^{*}}]^{\frac{1}{8q^{*}}}\Big).
				\end{align*}
			\end{thm}
			
			The proof of the above theorem is quit lengthy and will be divided into several steps, condensed into two distinct propositions that can be found in Subsection \ref{sub4.3}. 
			\begin{remark}\leavevmode
				\begin{itemize}
					\item By retaining all the assumptions of Theorem \ref{Main63} and further assuming that the coefficients of the forward SDE satisfy the standard conditions\footnote{the drift and the diffusion are Lipschitz continuous and satisfy the linear growth bounds as in \cite{ChaRichou16,Zhou}}, a convergence rate of order $\frac{1}{2}$ is achieved in this setting.  This result can still be viewed as an extension of the findings in \cite{Zhou}, broadening their applicability to more general class of quadratic BSDEs with stochastic Lipschitz drivers in all their components.
					\item Moreover, given a sequence of partitions $\{\pi(N)\}_{N\geq 1}$ with maximum time-step $h(N)$ such that $h(N)\leq KN^{-1}$, with $K>0$. Denote by $Z^{\pi(N)}(t)$ the $Z\text{-component}$ obtained in the numerical approximation \eqref{BTZ} under the same framework. Suppose that there exist constants  $C,\eta, q^{*} >0$ such that 
					\begin{align*}
						\mathbb{E}\Big[|\Phi(X)-\Phi(\mathcal{X}^{\pi(N)})|^{8q^{*}}\Big]^{\frac{1}{4q^{*}}}+ \mathbb{E}[|X_t- \mathcal{X}^{\pi(N)}_i|^{16q^{*}}]^{\frac{1}{8q^{*}}} < CN^{-\gamma}.
					\end{align*}
					Then we obtain that
					\begin{align*}
						\lim_{N\rightarrow \infty}	\mathbb{E}|Z^{\pi(N)}(t) -Z_t|^2 =0 \text{ a.s. for } t\in [0,T].
					\end{align*}
				\end{itemize}
			\end{remark}
			The following corollary offers valuable insight into the result mentioned above. Specifically, it provides a precise error estimate for the numerical approximation of the BSDEs of interest within our framework, further clarifying the impact of the noise regularization on the Euler scheme for singular-type SDEs.
			\begin{cor}\label{cor111}
				Let the assumptions of Theorem \ref{Main63} be in force with in particular $T=1$. Then
				\begin{align*}
					\mathbb{E}\Big[\sup_{0\leq i\leq N}|Y_{t_i}-\mathcal{Y}_{i}^{\pi}|^{2}\Big] + \mathbb{E}\Big[\sum_{i=1}^{N-1}\int_{t_i}^{t_{i+1}}|Z_{s}-\mathcal{Z}_{i}^{\pi}|^2\mathrm{d}s\Big] \leq Ch^{(1-\gamma)}.
				\end{align*}
			\end{cor}

			\subsection{Layout of the paper}
			The rest of the paper is organised as follows: In Section \ref{section1}, we provide some important insight on smoothness and numerical aspect of SDEs with singular drift coefficients.
			In Section \ref{secpremres}, we first give some basic concepts and definition that will be used throughout the paper. We also present some preliminary results on smoothness of FBSDE that are key to derive our main results. Section \ref{secproofmain} is devoted to the proofs of the main results whereas Section \ref{secconcl} provides the main findings of the paper and potential future direction. The Appendix contains the proofs of some auxiliary results. 
			
			\section{Smoothness and numerical approximation for singular SDEs}\label{section1}
			Here, we provide some results on smoothness of solution to SDE with Dini-continuous drifts.
			Before proceeding further, we recall the definition of Dini continuity.
			\begin{defi}\label{Dini}
				An increasing continuous function $\omega_f: \mathbb{R}_{+}\rightarrow \mathbb{R}_{+}$ is called a Dini function if 
				\begin{equation}
					\int_{0+} \frac{\omega_f(s)}{s}\mathrm{d}s < +\infty.
				\end{equation}
				A measurable function $\omega_f: \mathbb{R}_{+}\rightarrow \mathbb{R}_{+}$ is called a slowly varying function at zero if for every $\lambda >0$
				\begin{equation*}
					\lim_{s\rightarrow 0} \frac{\omega_f(\lambda s)}{\omega_f(s)} = 1.
				\end{equation*}
				A measurable function $f: \mathbb{R}^d \rightarrow \mathbb{R}$ is said to be Dini continuous if there is a Dini function $h$ such that 
				\begin{align}
					|f(x)-f(y)| \leq \omega_f(|x-y|).
				\end{align}
			\end{defi}
			When $\omega_f(r) = r^{\alpha}$ for some $\alpha \in (0,1)$ then the above regularity coincides with H\"{o}lder continuity.
			Every Dini continuous function is continuous.
			
			Let us now assume that the initial condition for the equation \eqref{2.1} is given by $(s,x) \in [0,T]\times\mathbb{R}^d$, for any $0\leq s \leq t \leq T$ and recall the following definition from \cite{Kun90}.
			\begin{defi}\label{defi 2.8}
				A stochastic homeomorphisms flow of class $C^{\beta}$ with $\beta\in (0,1)$ on $(\Omega,\mathfrak{F},\mathbb{P},(\mathfrak{F}_t)_{0\leq t \leq T})$ associated to \eqref{2.1} is a map $(s,t,x,\omega)\mapsto X_t^{s,x}(\omega)$, defined for $0\leq s\leq t \leq T$, $x\in \mathbb{R}^d$, $\omega \in \Omega$ with values in $\mathbb{R}^d$, such that 
				\begin{itemize}
					\item[(i)] The process $\{ X_t^{\cdot,x}\} = \{ X_t^{s,x} \}_{s\leq t\leq T}$ is a continuous $(\mathfrak{F}_{s,t})_{s\leq t \leq T}\text{-adapted}$ solution to \eqref{2.1}, for every $s\in [0,T]$, $x\in \mathbb{R}^d$;
					\item[(ii)] The function $X_t^{s,x}$ and $(X_t^{s,x})^{-1}$ are continuous in $(s,t,x)$ and are of class $C^{\beta}$ in $x$ uniformly in $(s,t)$, $\mathbb{P}\text{-a.s.},$ for all $0\leq s\leq t \leq T$;
					\item[(iii)] $X_t^{s,x} = X_t^{s,X_r^{s,x}}$ for all $0\leq s \leq r \leq t \leq T$, $x\in \mathbb{R}^{d}$, $\mathbb{P}\text{-a.s.},$ and $X_s^{s,x}= x$. 
				\end{itemize}
				Moreover, if the map $(s,t,x,\omega)\mapsto X_t^{s,x}(\omega)$ satisfies ${\bf(iv)}$, it is called a stochastic flow of $C^k$ diffeomorphisms
				\begin{itemize}
					\item[(iv)] $(s,t,x,\omega)\mapsto X_t^{s,x}(\omega)$ is $k\text{-times}$ differentiable with respect to $x$ for all $s,t$ and the derivatives are continuous in $(s,t,x)$.
				\end{itemize}
			\end{defi}
			A stochastic flow $(s,t,x,\omega)\mapsto X_t^{s,x}(\omega)$ of homeomorphisms is said to be {\it Sobolev-differentiable} if for all $s,t\leq T$ the maps $X_t^{s,x}$ and $(X_t^{s,x})^{-1}$ are Sobolev-differentiable i.e. $X_t^{s,x}$ and $(X_t^{s,x})^{-1}$ belong to $L^2(\Omega,W^{1,p}(\mathcal{O}))),$ where $\mathcal{O}$ is an open and bounded subset of $\mathbb{R}^d$.
			
			Further, if there is some $\gamma_0\geq 1$ such that 
			the maps $\nabla X_s^{t,x} $ and $\nabla (X_t^{s,x})^{-1}$ are in $C_b(\mathbb{R}^d;C_s([0,T];L^{\gamma_0}(\Omega;C_t([s,T]))))$ and $C_b(\mathbb{R}^d;C_t([0,T];L^{\gamma_0}(\Omega;C_s([0,t]))))$ respectively, then the stochastic flow $(s,t,x,\omega)\mapsto X_t^{s,x}(\omega)$ is called a stochastic quasi-diffeomorphisms flow.



			
			\subsection{Stochastic flow for SDEs with bounded and Dini continuous drifts}
			
			In this section, we recall the result on the existence of a quasi-diffeomorphism flow to the solution of the SDE \eqref{2.1} and derive some bounds on its supremum norm. Since the drift is not smooth, we apply a drift transformation to obtain an SDE with Lipschitz continuous coefficients through the solution to a Kolmogorov equation. More specifically, let us consider the following Cauchy problem with bounded and Dini continuous coefficients.
			\begin{align}\label{PDE1}
				\begin{cases}
					\partial_t u(t,x)=  \frac{1}{2}\Delta u(t,x) + \tilde{f}(t,x)\cdot\nabla u(t,x) + f(t,x),\quad (t,x)\in (0,T)\times\mathbb{R}^d\\
					u(0,x) = 0, \quad x\in  \mathbb{R}^d.
				\end{cases}
			\end{align}
			A solution $u(t,x)$ to \eqref{PDE1} is called a strong solution if $u \in L^{\infty}([0,T];W^{2,\infty}(\mathbb{R}^d))\cap W^{1,\infty}([0,T];L^{\infty}(\mathbb{R}^d))$ such that for a.a.$(t,x)\in [0,T]\times\mathbb{R}^d$ the equation \eqref{PDE1} is satisfied. The following representation of the strong solution to \eqref{PDE1} holds for all $(t,x)\in [0,T]\times\mathbb{R}^d$ 
			\begin{align*}
				u(t,x) = \int_0^t \Gamma(t-s,\cdot)*(\tilde{f}(s,\cdot)\cdot\nabla u(t,\cdot))(x)\mathrm{d}s
				+ \int_0^t \Gamma(t-s,\cdot)*\tilde{f}(s,\cdot)(x)\mathrm{d}s,
			\end{align*}
			where $\Gamma(t,x)= (2\pi t)^{-\frac{d}{2}}e^{-\frac{|x|^2}{2t}},$ $ t>0, x\in \mathbb{R}^d$ is the heat kernel of the operator $\partial_t -\frac{1}{2}\Delta $ in the whole space $\mathbb{R}^d$.
			
			\begin{thm}{\cite[Theorem 2.1]{WeiLvWang}}\label{WeiLvWang}
				Suppose that $f\in L^{\infty}([0,T];C_b(\mathbb{R}^d))$ and $\tilde{f}\in L^{\infty}([0,T];C_b(\mathbb{R}^d;\mathbb{R}^d))$. Suppose that $r_0 \in (0,1)$ and there is a Dini function $h$ such that for all $x \in \mathbb{R}^d$
				\begin{align*}
					|f(t,x)-f(t,y)| + |\tilde{f}(t,x)-\tilde{f}(t,y)| \leq h(|x-y|), \text{ for all } y \in B_{r_0}(x), t \in [0,T].
				\end{align*}
				Then the cauchy problem \eqref{PDE1} has a unique strong solution $u(t,x)$ such that
				\begin{align}\label{bound2.13}
					\|u\|_{L^{\infty}([0,T];C_b^2(\mathbb{R}^d))} \leq C(d,T)\left( 1+ 	\|f\|_{L^{\infty}([0,T];C_b(\mathbb{R}^d))} + 	\|\tilde{f}\|_{L^{\infty}([0,T];C_b(\mathbb{R}^d;\mathbb{R}^d))}\right).
				\end{align}
			\end{thm}

			The following result is taken from \cite[Theorem 3.1]{WeiLvWang}.
			%
			\begin{thm}\label{th 2.5}
				Assume that the drift $b$ satisfies the assumption (A1). Then, for every $s\geq 0$, $x \in \mathbb{R}^d$ the SDE \eqref{2.1} has a unique continuous adapted solution $X^{s,x}= (X_t^{s,x}(\omega), t\geq s, \omega \in \Omega)$ which forms a stochastic quasi-diffeomorphisms flow. For any $p \geq 1$, there is a constant $C$ only depending on $p,d$ and $T$ such that 
				\begin{align*}
					\sup_{x\in \mathbb{R}^d}\sup_{0\leq s \leq T }\mathbb{E} \big[\sup_{ s\leq t\leq T} (|X_t^{s,x}|^p + |\nabla_x X_t^{s,x}|^p)\big] \leq C (1+\|b\|_{L^{\infty}([0,T];C_b(\mathbb{R}^d))})
				\end{align*}
				Moreover, for any $p \geq 1$ and for any $x,y \in \mathbb{R}^d$
				\begin{align*}
					\sup_{0\leq s\leq T}\mathbb{E} \big[\sup_{s\leq t\leq T} |X_t^{s,x}- X_t^{s,y}|^p \big] \leq C(1+\|b\|_{L^{\infty}([0,T];C_b(\mathbb{R}^d))}) |x-y|^p,
				\end{align*}
				and 
				\begin{align*}
					&\sup_{0\leq s\leq T}\mathbb{E} \big[\sup_{s\leq t\leq T} \|\nabla_{x}X_t^{s,x}- \nabla_{x}X_t^{s,y}\|^p \big]\\
					& \leq C \Big[ \int_{0<r\leq |x-y|} \frac{h(r)}{r}\mathrm{d}r + h(|x-y|) +|x-y| \int_{|x-y|<r\leq r_0} \frac{h(r)}{r^2}\mathrm{d}r  \Big]^p 1_{\{|x-y|<r_0\}}+ C |x-y|^p.
				\end{align*}
			\end{thm}
			\begin{remark}\leavevmode
				\begin{itemize}
					\item The above pointwise estimate of the supremum norm of the first variation process suggests that $x\mapsto \nabla_x X^{s,x}$ is uniformly continuous for every $0\leq s\leq t\leq T$, as a valued function of $L^p(\Omega)$.
					\item Note that not much was said about the inverse of the first variation process $x\mapsto (\nabla_{x}X_t^{s,x})^{-1}$ in \cite{WeiLvWang}. However, when the drift $b \in L^p([0,T];L^q(\mathbb{R}^d;\mathbb{R}^d))$ with $p,q\in [2,\infty]$ and $\frac{2}{p}+ \frac{d}{q}<1$, the author in \cite{Reza} demonstrates the existence of Sobolev-differentiable stochastic homeomorphism flows of class $C^{\beta} (0<\beta <1/2)$ for \eqref{2.1}. Moreover, for all $\gamma\geq 1$ the following holds
					\begin{align}\label{est28}
						\sup_{x\in \mathbb{R}^d}\sup_{0\leq t \leq T }\mathbb{E}\left[ \|(\nabla_x X_t^{0,x})^{-1}\|^{\gamma} \right]\leq C.
					\end{align}
					In particular, the estimate \eqref{est28} remains valid for $p=q=+\infty$ i.e., when the drift $b$ is merely bounded and measurable, and the initial state of the diffusion $X$ solution to \eqref{2.1}, is $x$ at time $t=0$.
				\end{itemize}
			\end{remark}
			
			Before concluding this section, we derive the following result on the integrability of the supremum norms of both the Malliavin derivative and the inverse of the first variation process of $(X_t)_{0\leq t\leq T}$, the solution to \eqref{2.1}. This result will play a crucial role in the regularity analysis of both processes $Y$ and $Z$.	Its proof relies on the so-called It\^{o}-Tanaka trick, as introduced in \cite{FlanGubiPrio10} and further developed in \cite{WeiLvWang}.  In summary, this method uses the regularity of solutions for a type of parabolic PDE (PDE \eqref{PDE1} in our case) to transform the SDE \eqref{2.1} into a new one with a drift coefficient possessing an improved regularity.
			
			\begin{prop}\label{cor 5.7}
				Under Assumption (A1), the solution $(X_t^x,0\leq t\leq T)$ to equation \eqref{2.1} is Malliavin differentiable and for any $p\geq 2$ there is a constant $C$ only depending on $p,d ,T$ such that
				\begin{align}\label{5.12}
					\sup_{0\leq s\leq T} \mathbb{E}\big[ \sup_{s\leq t \leq T}|D_sX_t^x|^p  \big] + \sup_{x \in \mathbb{R}^d} \sup_{0\leq s\leq T} \mathbb{E}\big[ \sup_{s\leq t \leq T}|(\nabla_xX_t^x)^{-1}|^p\big] \leq C(1+\|b\|_{L^{\infty}([0,T];C_b(\mathbb{R}^d))}).
				\end{align}
				Furthermore, we also have for all $p\geq 2$
				\begin{align}
					&\sup_{0\leq s\leq T} \mathbb{E}\big[ \sup_{s\leq t \leq T}|D_sX_t^x|^p /\mathfrak{F}_s \big] + \sup_{x \in \mathbb{R}^d} \sup_{0\leq s\leq T} \mathbb{E}\big[ \sup_{s\leq t \leq T}|(\nabla_xX_t^x)^{-1}|^p/\mathfrak{F}_s \big] \leq C.
				\end{align}
				
			\end{prop}
			
			\begin{proof}
				Fix $\lambda > 0$ and consider the following backward Kolmogorov PDE 
				\begin{align}\label{2.8}
					\begin{cases}
						\partial_t u_{\lambda}(t,x) + \frac{1}{2}\Delta u_{\lambda}(t,x) + (b\cdot Du_{\lambda})(t,x)- \lambda u_{\lambda}(t,x) = -b(t,x), \\
						u_{\lambda}(T,x) = 0, \quad (t,x) \in [0,T)\times \mathbb{R}^d.
					\end{cases}
				\end{align}
				From Theorem \ref{WeiLvWang} the equation \eqref{2.8} has a unique strong solution $u_{\lambda} \in L^{\infty}([0,T); C_b^{2}(\mathbb{R}^d;\mathbb{R}^d))\cap W^{1,\infty}([0,T); C_b(\mathbb{R}^d;\mathbb{R}^d))$. In addition for $\lambda$ large enough, the map defined by $\Psi_{\lambda}(t,x) := x + u_{\lambda}(t,x) $ forms a non singular diffeomorphism of class $C^2$ uniformly in $t \in [0,T]$ and its inverse denoted by $\Psi_{\lambda}^{-1}$ has bounded first and second spatial derivatives, uniformly in $t \in [0,T]$ i.e.
				\begin{align}
					\frac{1}{2} \leq \|D\Psi_{\lambda}\|_{\infty} \leq \frac{3}{2},\,\, \text{ and }\,\,
					\frac{2}{3} \leq \|D\Psi_{\lambda}^{-1}\|_{\infty} \leq {2}.
				\end{align}
				Let us consider the following SDE 
				\begin{align}\label{auxi}
					\tilde{X}_t = \tilde{x} + \int_{s}^{t} \tilde{b}(v, \tilde{X}_v)\mathrm{d}v + \int_{s}^{t} \tilde{\sigma}(v, \tilde{X}_v)\mathrm{d}W_v ,\quad t \in [0,T],
				\end{align}
				where $\tilde{b}(t,\tilde{x}) = \lambda u_{\lambda}(t,\Psi_{\lambda}^{-1}(t,\tilde{x}))$ and $\tilde{\sigma}(t,y) = I +D u_{\lambda} (t,\Psi_{\lambda}^{-1}(t,\tilde{x}))$. 
				It is then clear that $\tilde{b}$ and $\tilde{\sigma}$ are only globally Lipschitz continuous. Hence, we cannot refer to  \cite[Theorem 4.4]{Kunita82} to conclude that the Jacobian matrix $x\mapsto \nabla_x \tilde{X}_t$ is non singular for any $x$ a.s. Instead, we invoke the result from  \cite{BouleauHirsch89} to assert this. Moreover, the inverse $(\nabla_x\tilde X_t)^{-1}$ satisfies a linear SDE with bounded coefficients. In particular, for all $p\geq 1$ there exists a constant $C$ such that
				\begin{align*}
					\sup_{0\leq s\leq T} \mathbb{E}[ \sup_{s\leq t \leq T}|(\nabla_x\tilde X_t)^{-1}|^p ] \leq C.
				\end{align*}
				On the other hand, assuming (for simplicity) that the initial conditions of \eqref{auxi} is given by $(0,\tilde{x})$, for any $\tilde{x} \in \mathbb{R}^d$, then from \cite[Theorem 2.2.1]{Nua06}, the equation \eqref{auxi} has a unique strong Malliavin differentiable solution and for any $p\geq 2$ and for $s \leq t$, there is a constant $C$ only depending on $p,d,T$, $\|Du_{\lambda}\|_{\infty}$ and $\|D^2u_{\lambda}\|_{\infty}$ such that:
				\begin{align*}
					\sup_{0\leq s\leq T} \mathbb{E}[ \sup_{s\leq t \leq T}|D_s\tilde X_t|^p ] \leq C .
				\end{align*}
				Let us remark that, from \eqref{bound2.13}, the norms $\|Du_{\lambda}\|_{\infty}$ and $\|D^2u_{\lambda}\|_{\infty}$ are independent of $\lambda$. Thus, besides the dependence on $p,d$ and $T$, the constants $C$ above only depend on $\|b\|_{L^{\infty}([0,T];C_b(\mathbb{R}^d))}$.
				
				At last, $(\tilde{X}_t,0\leq t\leq T)$, the solution to the SDE \eqref{auxi}, is related to $({X}_t,0\leq t\leq T)$, the solution of equation \eqref{2.1}, by the relation $X_t = \Psi_{\lambda}^{-1}(t,\tilde X_t)$. Moreover, $X$ is Malliavin differentiable, since the drift $b$ is at most bounded and continuous (\cite{MMNPZ13}). Hence, by applying the chain rule for Malliavin calculus  and using the fact that $\Psi_{\lambda}^{-1}$ has bounded first derivative, we deduce that the following holds:
				\begin{align*}
					\mathbb{E}[ \sup_{s\leq t \leq T}|D_s X_t|^p ]\leq 2^p \mathbb{E}[ \sup_{s\leq t \leq T}|D_s\tilde X_t|^p ]< \infty. 
				\end{align*}
				Similarly, using the result from \cite{NilssenProske} stating that the stochastic flow $X_t^s$ belongs to $ L^2(\Omega;W^{1,p}_{loc}(\mathbb{R}^d,\mathbb{R}^d))$, we deduce that  
				\begin{align*}
					\mathbb{E}[ \sup_{s\leq t \leq T}|\nabla_x X_t|^{-p} ]\leq \left(\frac{3}{2}\right)^p \mathbb{E}[ \sup_{s\leq t \leq T}|\nabla_x\tilde X_t|^{-p} ]< \infty. 
				\end{align*}
				Thus, \eqref{5.12} follows. By following the same strategy as above the conditional expectation version bounds can be derived as well. This finishes the proof.
			\end{proof}
			\subsection{Euler-Maruyama scheme for SDEs with bounded and Dini continuous drifts} The main aim in this subsection is to provide an
			$L^p$ version of Dareiotis and Gerencs\'er's theorem on strong convergence of the Euler approximation of solutions to SDEs  \eqref{2.1} when drift is bounded and Dini continuous.
			\begin{thm}\label{thm210}
				Assume that the drift $b$ satisfies Assumption (H1). Let $\gamma \in (0,1)$. Then for all $p\geq 1$, $N\in \mathbb{N}$ we obtain that
				\begin{align*}
					\mathbb{E}\Big[	\sup_{0\leq t\leq 1}|X_t^N-X_t|^{2p}\Big] \leq C N^{(-1+\gamma)p},
				\end{align*}
			\end{thm}
			where the process $X^N$ stands for the Euler-Maruyama scheme given by
			\begin{align}\label{euler}
				\mathrm{d}X_t^N = b(t,X^N_{k_N(t)})\mathrm{d}t +\mathrm{d}W_t, \quad X_0^N = x,
			\end{align}
			where, $k_N(t) = \frac{\lfloor N t\rfloor}{N},$ and the symbol $\lfloor \cdot \rfloor$ denotes the integer part.
			
			Since the pivotal concept for establishing the above result was only recently introduced in the literature (see \cite{Le22}), we will briefly recall some key ideas here for the sake of clarity and better understanding.
			\begin{defi}\label{defi28}
				A real valued  adapted right continuous process with left limits $(V_t)_{t\in [0,\tau]}$ is called a bounded mean oscillation $(\bmo)$ if 
				\begin{align*}
					[V]_{\bmo}:= \sup_{0\leq s\leq S\leq S'\leq \tau}\|\mathbb{E}[|V_{S'}-V_{S-}|/\mathfrak{F}_{ S}]\|_{\infty} <\infty,
				\end{align*}
				where the supremum is taken over all the stopping times $S,S'$; $V_{S-} =\lim_{r\uparrow S}V_r$ and we set $V_{0-} =V_0$, by convention.
			\end{defi}
			The \text{\it modulus of mean oscillation} $\rho(V):\{(s,t)\in [0,\tau]^2: s\leq t\} \rightarrow [0,\infty)$ of each $\bmo\text{-process } V$ is defined as
			\begin{align*}
				\rho_{s,t}(V) = \sup_{0\leq s\leq S\leq S'\leq t}\|\mathbb{E}[|V_{S'}-V_{S-}|/\mathfrak{F}_{S}]\|_{\infty},\quad 0\leq s\leq t\leq \tau,
			\end{align*}
			where the supremum is taken over all stopping times $S,S'$ satisfying $s\leq S\leq S'\leq t.$
			
			The following can be found in \cite[Theorem 2.3]{Le22}
			\begin{thm}[John-Nirenberg inequality] Given $V$ a $\bmo\text{-process}$ and let $r$ be a fixed number in $[0,\tau]$. Then, for every integer $p\geq 1$
				\begin{align*}
					\|\mathbb{E}[\sup_{r\leq t\leq \tau}|V_t-V_r|^p/\mathfrak{F}_{r}]\|_{\infty} \leq p!(11\rho_{r,\tau}(V))^p.
				\end{align*}
			\end{thm}
			Together with the author in \cite{Le22}, we observe that for any bounded and measurable function $b$ and for each integer $N \in \mathbb{N}$, the process $V_t^N:= \int_0^t b(s,W_s)-b(s,W_{k_N(s)})\mathrm{d}s$ defines a $\bmo\text{-process}$. In fact, from \cite[Lemma 2.1]{DaGe20} the following bound holds for all $s\leq t$
			\begin{align}
				\mathbb{E}[|V_{t}^N-V_{s}^N|/\mathfrak{F}_{s}] \leq \mathbb{E}[|V_{t}^N-V_{s}^N|^2/\mathfrak{F}_{s}]^{\frac{1}{2}} \leq C(N^{-1}\log(N+1))^{\frac{1}{2}}.
			\end{align}
			Therefore, the John-Nirenberg inequality guarantees the following strong estimate 
			\begin{align}\label{bound213}
				\essup_{\omega\in \Omega}\left(\mathbb{E}\sup_{s\leq t \leq 1}\Big|\int_s^t b(r,W_r)-b(r,W_{k_N(r)})\mathrm{d}r\Big|^{2p}\big/\mathfrak{F}_s\right) \leq p! \big(121CN^{-1}\log(N+1)\big)^p.
			\end{align}
			\begin{proof}[Proof of Theorem \ref{thm210}] We will start by providing an $L^p$ bound of the following bounded process
				\begin{align*}
					\mathit{H}(X^N) = \sup_{0\leq t\leq 1}\Big|\int_{0}^{t}(f(s,X^N_s)-f(s,X^N_{k_N(s)}))\mathrm{d}s\Big|,
				\end{align*}
				where we recall that the bound \eqref{bound213} provides a similar bound for the process $\mathit{H}(W_t^x)$, and $W_t^x$ denotes the Brownian motion starting at $x \in \mathbb{R}^d.$ Applying  H\"older's inequality, we obtain that 
				\begin{align*}
					\mathbb{E}\mathit{H}(X^N)^{2p} &\leq C \mathbb{E}\Big[(\mathit{H}(X^N))^{2p-\gamma}\mathcal{E}_N^{\frac{2p-\gamma}{2p}}\mathcal{E}_N^{\frac{\gamma -2p}{2p}}\Big]
					\leq C \mathbb{E}\Big[\mathit{H}(X^N)^{2p}\mathcal{E}_N\Big]^{\frac{2p-\gamma}{2p}}\mathbb{E}\Big[\mathcal{E}_N^{\frac{\gamma-2p}{\gamma}}\Big]^{\frac{\gamma}{2p}}=\mathbb{E}^N\Big[\mathit{H}(X^N)^{2p}\Big]^{\frac{2p-\gamma}{2p}}\mathbb{E}\Big[\mathcal{E}_N^{\frac{\gamma-2p}{\gamma}}\Big]^{\frac{\gamma}{2p}}
				\end{align*}
				where $\mathcal{E}_N:=\mathcal{E}_N(\int_0^t b(X_s^N)\mathrm{d}W_s)$ stands for the Dol\'eans-Dade exponential under which the measure $\mathrm{d}\mathbb{P}^N:= \mathcal{E}_N\mathrm{d}\mathbb{P}$ defines an equivalent measure of the underlying probability measure $\mathbb{P}$. $\mathbb{E}^N$ stands for the expectation with respect to $\mathbb{P}^N$. Moreover, thanks to Girsanov theorem the process $(X_t^N-x)$ defines a Brownian motion under $\mathbb{P}^N$. Therefore, from \eqref{bound213}, using  $[N^{-1}\log(N+1)]^{\frac{2p-\gamma}{2}} \leq C N^{(-1+\gamma)p}$ and the uniform bound of the drift $b$ we deduce that
				\begin{align}\label{bound214}
					\mathbb{E}\mathit{H}(X^N)^{2p} \leq C \Big[p!\big(121CN^{-1}\log(N+1)\big)^p\Big]^{\frac{2p-\gamma}{2p}}\leq C(p) N^{(-1+\gamma)p}.
				\end{align}
				Let us now turn to the core of the proof. We start by choosing $T_0<1$ sufficiently small. For $i=1,\ldots, d$, applying a time reversal argument, we know from Theorem \ref{WeiLvWang} that there is a unique strong solution $u$ to the following PDE
				\begin{align*}
					\partial_t u^i(t,x) + \frac{1}{2}\Delta u^i(t,x) + (b\cdot\nabla u^i)(t,x) = -b^i(t,x), \quad
					u^i(T,x) = 0, \quad \forall x\in  \mathbb{R}^d,
				\end{align*}
				such that 
				\begin{align}\label{bound220}
					\|u^i\|_{L^{\infty}([0,T];C_b^2(\mathbb{R}^d))} \leq C(d,T)\left( 1+ 	\|b\|_{L^{\infty}([0,T];C_b(\mathbb{R}^d))} \right).
				\end{align}
				Let us observe that since $b$ is bounded and Dini, one can show as in \cite{DaGe20} (see also \cite{PaTa17}) that 
				\begin{align}\label{eqboderu1}
					\|\nabla u^i\| \leq C(h,d,\|b\|_\infty)\sqrt{T}.
				\end{align}
				In the following $C(p)$ is a generic constant that might change from one line to the other. Applying It\^o's formula to $u^i(t,X_t)$ and $u^i(t,X_t^n)$, respectively and using the PDE above, we obtain that 
				\begin{align}
					\int_0^t b^i(s,X_s)\mathrm{d}s= &u^i(0,x) -u^i(t,X_t) +\sum_{j=1}^{d} \int_0^t  u^i_{x_j}(s,X_s)\mathrm{d}W^j_s  \label{eqintb1}\\
					\int_0^t b^i(s,X_s^N)\mathrm{d}s=	 &u^i(0,x) -u^i(t,X_t^N) + \sum_{j=1}^{d} \int_0^t  u^i_{x_j}(s,X_s^N)\mathrm{d}W^j_s \notag\\
					&+ \sum_{j=1}^{d}\int_0^t  u^i_{x_j}(s,X_s^N)(b^j(s,X^N_{k_n(s)}) - b^j(s,X^N_s))\mathrm{d}s. \label{eqintb2}
				\end{align}
				We have by its definition
				\begin{align}
					&|X_t^N-X_t|^{2p}\notag\\
					&\leq 2^{2p-1}\Big( \Big|\int_0^t b(s,X_s)\mathrm{d}s - \int_0^t b(s,X^N_s)\mathrm{d}s \Big|^{2p} + \Big| \int_0^t b(s,X^N_s)\mathrm{d}s - \int_0^t b(s,X^N_{k_N(s)})\mathrm{d}s \Big|^{2p}\Big).
				\end{align}
				By taking the supremum and then the expectation we obtain that
				\begin{align}\label{bound217}
					&\mathbb{E}\sup_{0\leq t\leq T_0}|X_t^N-X_t|^{2p}\notag\\
					&\leq C(p) \mathbb{E}\sup_{0\leq t\leq T_0}\Big|\int_0^t b(s,X_s)\mathrm{d}s - \int_0^t b(s,X^N_s)\mathrm{d}s \Big|^{2p} + C(p) \mathbb{E}\sup_{0\leq t\leq T_0}\Big| \int_0^t b(s,X^N_s)\mathrm{d}s - \int_0^t b(s,X^N_{k_N(s)})\mathrm{d}s \Big|^{2p} \notag\\
					&\leq C(p) (\mathbb{E}\sup_{0\leq t\leq T_0}J_t^N +  N^{(-1+\gamma)p}),
				\end{align}
				where we have used the bound \eqref{bound214} to obtain the last inequality. 
				Using \eqref{eqboderu1}, \eqref{eqintb1} and \eqref{eqintb2} and the boundedness of $b$ give
				\begin{align}\label{bound218}
					J_t^N &\leq C(p) \sum_{i=1}^{d}|u^i(t,X_t)-u^i(t,X_t^N)|^{2p} + C(p)\sum_{i,j=1}^{d}\Big|\int_0^t u^i_{x_j}(s,X_s)- u^i_{x_j}(s,X_s^N)\mathrm{d}W^j_s\Big|^{2p}\notag\\
					&\quad+ C(p)\sum_{i,j=1}^{d} \Big|\int_0^t (b^j u^i_{x_j})(s,X_s^N)- (b^j u^i_{x_j})(s,X_{k_N(s)}^N)\mathrm{d}s\Big|^{2p}\notag\\
					&\quad + C(p)\sum_{i,j=1}^{d}\int_0^t |b^j(X_{k_N(s)}^N)|^{2p} | u^i_{x_j}(s,X_s^N)- u^i_{x_j}(s,X_{k_N(s)}^N)|^{2p}\mathrm{d}s\notag\\
					&	\leq C(p,d)\sqrt{T_0} |X_t-X_t^N|^{2p} + C(p)\sum_{i,j=1}^{d}\Big|\int_0^t u^i_{x_j}(s,X_s)- u^i_{x_j}(s,X_s^N)\mathrm{d}W^j_s\Big|^{2p}\notag\\
					&\quad+ C(p)\sum_{i,j=1}^{d} \Big|\int_0^t (b^j u^i_{x_j})(s,X_s^N)- (b^j u^i_{x_j})(s,X_{k_N(s)}^N)\mathrm{d}s\Big|^{2p}\notag\\
					&\quad + C(p,d)\sum_{i,j=1}^{d}\int_0^t | X_s^N- X_{k_N(s)}^N|^{2p}\mathrm{d}s.
				\end{align}
				Taking the supremum and the expectation on both sides, and applying the bound \eqref{bound214}, the Burkholder-Davis-Gundy inequality and the boundedness of the drift $b$ we deduce that
				\begin{align*}
					\mathbb{E} \sup_{0\leq t\leq T_0}J_t^N 
					&\leq C(p,d)\sqrt{T_0} \mathbb{E}\sup_{0\leq t\leq {T_0} }|X_t-X_t^N|^{2p} + C(p,d)\Big(\int_0^{T_0} \mathbb{E}|X_s- X_s^N|^{2}\mathrm{d}s\Big)^p \\
					&\quad +C(p) N^{(-1+\gamma)p} + C(p,d)\int_0^{T_0} \mathbb{E}|X_s^N- X_{k_N(s)}^N|^{2p}\mathrm{d}s\\
					&\leq C(p,d)\sqrt{T_0} \mathbb{E}\sup_{0\leq t\leq {T_0} }|X_t-X_t^N|^{2p} + C(p,d)\int_0^{T_0} \mathbb{E}\sup_{0\leq r\leq s} |X_r- X_r^N|^{2p}\mathrm{d}s \\
					&\quad +C(p) N^{(-1+\gamma)p} + C(p,d)N^{-p}
				\end{align*}
				Combining this with \eqref{bound217} and choosing $T_0$ small enough such that $C(p,d)\sqrt{T_0}<\frac{1}{2}$. 
				\begin{align}
					\mathbb{E}\sup_{0\leq t\leq T_0}|X_t^N-X_t|^{2p}
					\leq C(p,d) \Big(\int_0^{T_0} \mathbb{E}\sup_{0\leq r\leq s} |X_r- X_r^N|^{2p}\mathrm{d}s+  N^{(-1+\gamma)p}\Big).
				\end{align}
				
				The desired bound is obtained by applying the Gronwall's lemma to the function $t\mapsto 	\mathbb{E}\sup_{0\leq r\leq t}|X_r^N-X_r|^{2p}$. This concludes the proof for $T_0$ sufficiently small. The case of $T_0$ not sufficiently small can be handle as in \cite{DaGe20, PaTa17}. 
			\end{proof}	
			The following result is also interested on its own, since it provides a stability result for the Euler-Maruyama scheme for SDEs with non-smooth drifts
			\begin{cor}\label{Cor210}
				Let the assumptions of Theorem be in force. Assume further that for all $i=1,2$ $X_t^{i,N}$ is the Euler approximation of the solution to the SDE \eqref{2.1} given by \eqref{euler}. Then the following holds
				\begin{align*}
					\mathbb{E}\Big[	\sup_{0\leq t\leq 1}|X_t^{1,N}-X_t^{2,N}|^{2p}\Big] \leq C N^{(-1+\gamma)p}.
				\end{align*}
			\end{cor}
			\begin{proof}
				Following the same arguments as in the proof of the previous theorem, we will simply assume that $T_0 < 1$ sufficiently small such that for all $p\geq 1$
				\begin{align*}
					&\mathbb{E}\Big[	\sup_{0\leq t\leq T_0}|X_t^{1,N}-X_t^{2,N}|^{2p}\Big]\\ &\leq C(p)N^{(-1+\gamma)p} + \mathbb{E}\Big|\int_0^{T_0}b(s,X_s^{1,N})\mathrm{d}s - \int_0^{T_0}b(s,X_s^{2,N})\mathrm{d}s\Big|^{2p}\\
					&\leq  C(p)N^{(-1+\gamma)p} + C(p,d)\sqrt{T_0}\mathbb{E}\Big[	\sup_{0\leq t\leq T_0}|X_t^{1,N}-X_t^{2,N}|^{2p}\Big] + C(p,d)\int_0^{T_0} \mathbb{E}\sup_{0\leq r\leq s} |X_r^{1,N}- X_r^{2,N}|^{2p}\mathrm{d}s.
				\end{align*}
				Hence, applying once more the Gronwall's lemma we obtain the result. This concludes the proof.
			\end{proof}
			\section{Notations and Preliminary results}\label{secpremres}
			In this section, we start this by recalling some notations, and definitions of some basic concepts. We also revisit several results on the differentiability of Forward-Backward Stochastic Differential Equations (FBSDEs) with rough drift and quadratic drivers, as well as the differentiability of quadratic BSDEs with discrete functional terminal values.
			
			\subsection{Some notations and definitions}
			
			For fixed $T> 0,d \in \mathbb{N},$ $p\in [2,\infty)$, we denote by:
			\begin{itemize}
				\item $\mathbb{D}$ stands for the space of all c\`adl\`ag functions defined on $[0,T]$;
				\item $L^p(\mathbb{R}^d)$ the space of $\mathfrak{F}_T\text{-adapted}$ random variables  $X$ such that $\Vert X \Vert^p_{\tiny L^p }:= \mathbb{E}|X|^p < \infty$; 
				\item $L^{\infty}(\mathbb{R}^d)$ the space of bounded random variables with norm $\Vert X \Vert_{\tiny L^{\infty} }:= \essup_{\omega\in \Omega}|X(\omega)|$; 
				\item $\mathcal{S}^p(\mathbb{R}^d)$ the space of all adapted continuous $\mathbb{R}^d$-valued processes $X$ such that $\Vert X \Vert^p_{\tiny \mathcal{S}^p(\mathbb{R}^d) }:= \mathbb{E}\sup_{t\in [0,T]} |X_t|^p < \infty$; 
				\item $\mathcal{H}^p(\mathbb{R}^d)$ the space of all predictable $\mathbb{R}^d$-valued processes $Z$  such that $\Vert Z \Vert^p_{\tiny \mathcal{H}^p(\mathbb{R}^d) }:= \mathbb{E}(\int_0^T |Z_s|^2\mathrm{d}s)^{p/2}  < \infty$; 
				\item $\mathcal{S}^{\infty}(\mathbb{R}^d)$ the space of continuous $\{ \mathfrak{F}_s \}_{0\leq t\leq T} $-adapted processes $Y:\Omega\times[0,T]\rightarrow \mathbb{R}^d$ such that $\Vert Y \Vert_{\tiny \infty }:= \essup_{\omega\in \Omega} \sup_{t\in [0,T]} |Y_t(\omega)| < \infty;$
				\item BMO($\mathbb{P}$) the space of square integrable martingales $M$ with $M_0 = 0$ such that $ \Vert M\Vert_{\text{BMO}(\mathbb{P})} = \sup_{\tau \in [0,T]} \Vert \mathbb{E} [\langle M \rangle_T  - \langle M \rangle_{\tau}]/\mathfrak{F}_{\tau} \Vert_{\infty}^{1/2} < \infty,$ the supremum is taken over all stopping times $\tau \in [0,T];$
				\item $\mathcal{H}_{\text{BMO}}$ the space of $\mathbb{R}^d\text{- valued}$ $\mathcal{H}^p\text{-integrable} $ processes $(Z_t)_{t\in [0,T]}$ for all $p\geq 2$ such that $Z*B = \int_0 Z_s\mathrm{d}B_s \in \text{BMO}(\mathbb{P}).$ We define $\Vert Z \Vert_{\mathcal{H}_{\text{BMO}}}:= \Vert \int Z\mathrm{d}B \Vert_{\text{BMO}(\mathbb{P})}$;
				\item $L^{\infty}([0,T];C_b^{\beta}(\mathbb{R}^d;\mathbb{R}^d))$ the space of all vector fields $b:[0,T]\times \mathbb{R}^d\rightarrow \mathbb{R}^d$ having all components in $L^{\infty}([0,T];C_b^{\beta}(\mathbb{R}^d))$ and $L^{\infty}([0,T];C_b^{\beta}(\mathbb{R}^d))$ stands for the set of all bounded Borel functions $b:[0,T]\times \mathbb{R}^d\rightarrow \mathbb{R}$ such that 
				\[ [b]_{\beta,T} = \sup_{t \in [0,T]} \sup_{x\neq y\in \mathbb{R}^d} \frac{|b(t,x) -b(t,y)|}{|x-y|^\beta} < \infty . \]
			\end{itemize}
			
			\begin{defi}\label{defi}
				A functional $\Phi: \mathbb{D}^d \rightarrow \mathbb{R}$ is called $L^{\infty}\text{-Lipschitz}$, if there exists a constant $\Lambda_{\Phi}> 0$ such that 
				\begin{align}\label{Lips1}
					|\Phi(\bold{x}_1)-\Phi(\bold{x}_2)| \leq \Lambda_{\Phi} \sup_{ 0\leq t\leq T}|\bold{x}_1(t) -\bold{x}_2(t)|, \quad  \forall \bold{x}_1,\bold{x}_2 \in \mathbb{D}^d; 
				\end{align}
				and $\Phi: \mathbb{D}^d \rightarrow \mathbb{R}$ is called $L^{1}\text{-Lipschitz}$, if it satisfies
				\begin{align}\label{Lips2}
					|\Phi(\bold{x}_1)-\Phi(\bold{x}_2)| \leq \Lambda_{\Phi} \int_{ 0}^{T}|\bold{x}_1(t) -\bold{x}_2(t)|\mathrm{d}t, \quad  \forall \bold{x}_1,\bold{x}_2 \in \mathbb{D}^d.
				\end{align}
			\end{defi}
			As pointed in \cite{Zhang041}, two typical examples of $L^{\infty}\text{-Lipschitz}$ and $L^{1}\text{-Lipschitz}$ continuous functionals are given by $\Phi(\bold{x})= \max_{0\leq t \leq T}|\bold{x}(t)|$ and $\Phi(\bold{x})= \int_0^T\bold{x}(t)\mathrm{d}t$ which for example could represent the payoff of a lookback options and Asian options, respectively.
			Let us recall the following approximation result from \cite{MaZhang01} and \cite{Zhang041}
			\begin{lemm}\label{5lem22}
				Suppose that $\Phi$ is an $L^{\infty}\text{-Lipschitz}$ functional satisfying condition \eqref{Lips1}\footnote{A similar approximation procedure in the case where $\Phi$ is an $L^{1}\text{-Lipschitz}$ can be found in \cite[Lemma 5.2]{MaZhang01}}. Let $\Pi = \{\pi\}$ be a family of partitions of $[0,T]$. Then there is a family of discrete functionals $\{ \phi^{\pi}: \pi \in \Pi \}$ such that 
				\begin{itemize}
					\item[(i)] for each $\pi \in \Pi$, assuming $\pi:0=t_0<\cdots < t_n=T$ we have that $\phi^{\pi} \in C_b^{\infty}(\mathbb{R}^{d(n+1)})$ and satisfies 
					\begin{equation}
						\sum_{i=0}^n |\nabla_{x_i}\phi^{\pi}(x)| \leq S({\phi}), \quad \forall x=(x_0,\cdots,x_n) \in \mathbb{R}^{d(n+1)},
					\end{equation}
					where the constant $S(\phi)$ is such that $S(\phi) < C \Lambda_{\Phi}$ and $\Lambda_{\Phi}$ is given in \eqref{Lips1}.
					\item[(ii)] for any $\bold{x} \in \mathbb{D}^d$, it holds that 
					\begin{equation}
						\lim_{|\pi|\rightarrow 0} \left|\phi^{\pi}(\bold{x}(t_0),\cdots,\bold{x}(t_n)) - \Phi(\bold{x})   \right| = 0.
					\end{equation}
				\end{itemize}
			\end{lemm} 
			\begin{remark}\label{remark24} For bounded terminal value $\Phi$, we will then consider the bounded approximation of the $L^{\infty}\text{-Lipschitz}$ or $L^{1}\text{-Lipschitz}$  functional $\Phi$ as proposed in \cite[Remark 3.3]{Zhou}:
				\begin{equation}\label{appro2.8}
					\hat{\phi}^{\pi} = \left( \|\xi\|_{L^{\infty}} \wedge \phi^{\pi} \vee (-\|\xi\|_{L^{\infty}})   \right)*\rho_r,
				\end{equation}
				where $\{\rho_r\}_r$ stands for a sequence of mollifiers $\rho_r: \mathbb{R}^{d\times r} \rightarrow \mathbb{R}$ with $\int \rho_r \mathrm{d}x_1\cdots \mathrm{d}x_r$. Therefore, for any continuous process $X$ defined on a filtered probability space $(\Omega,\mathfrak{F},\mathbb{P},\{\mathfrak{F}_t\}_{t})$ it holds 
				\begin{equation}\label{appro2.9}
					\lim_{|\pi|\rightarrow 0}\left|\hat{\phi}^{\pi}\left(X_{t_0},\cdots,X_{t_r}\right) - \Phi(X)\right|= 0.
				\end{equation}
			\end{remark}
			In the sequel, the function $\phi: \mathbb{R}^{d\times r} \rightarrow \mathbb{R}$ stands for any Lipschitz continuous function such that
			\begin{align}\label{2.6}
				|\phi(x_1,x_2,\cdots,x_r) - \phi(y_1,y_2,\cdots,y_r) | \leq \sum_{i=1}^{r} \Lambda_{\phi}^{(i)}|x_i-y_i| ,
			\end{align}
			for any $x_1,x_2,\cdots,x_r \in \mathbb{R}^d$ and $y_1,y_2,\cdots,y_r \in \mathbb{R}^d$, and the sum of Lipschitz coefficients will be denoted by 
			\begin{align}\label{2.7}
				S(\phi) = \sum_{i=1}^{r}\Lambda_{\phi}^{(i)}.
			\end{align}
			\subsection{Differentiability of quadratic FBSDEs with Dini-continuous drift}
			Here, we extend the general results derived in \cite{ImkRhossOliv} and \cite{Zhou}, by assuming that the drift of the forward process is bounded and Dini continuous, and the terminal value is a functional of multiple final valued processes $(X_T^{(i)})$ for $1\leq i\leq r$.

			We start by studying the differentiability with respect to the starting point of the forward process $X$.
			
			\begin{prop}\label{ClassYM}
				Suppose Assumption \ref{assum2.1} holds. Fix $r >0$ and Let $\bold{x} = (x^{(1)},x^{(2)},\cdots,x^{(r)})$ and $p > 1$. Assume in addition the terminal condition $\xi$ is of the following form $\xi = \Phi(X_T^{(1)},X_T^{(2)},\cdots,X_T^{(r)})$, where $\Phi\in C^1(\mathbb{R}^{d\times r};\mathbb{R})$ is bounded and $(X_T^{(i)})_{1\leq i\leq r}$ solves \eqref{2.1} for each drift $b^{i}$, $1\leq i\leq r$ satisfying Assumption (A1). Assume furthermore that $g$ is continuously differentiable in $x$, $y$ and $z$. Then the couple $(Y,Z)$ solution to \eqref{2.2} is differentiable and for all $1\leq i\leq r$, the derivative process $(\nabla_{{x^{(i)}}} Y,\nabla_{{x^{(i)}}} Z) \in \mathcal{S}^p\times\mathcal{H}^p$ solves the following BSDE for any $p> p_0'$
				\begin{align}\label{eq31}
					\nabla_{{x^{(i)}}} Y_t =& \nabla_{{x^{(i)}}} \Phi(X_T^{(1)},X_T^{(2)},\cdots,X_T^{(r)})\nabla_{{x^{(i)}}} X_T^{(i)} - \int_{t}^{T} \nabla_{{x^{(i)}}} Z_s \mathrm{d}W_s\notag\\
					&+ \int_{t}^{T} \langle \nabla g(s,{X}^{(i)}_s,Y_s,Z_s),\nabla_{{x^{(i)}}} \bold{X}_s^{(i)} \rangle \mathrm{d}s, 
				\end{align}
				where $\nabla_{{x^{(i)}}}\bold{X}^{(i)} = (\nabla_{{x^{(i)}}}X^{(i)}, \nabla_{{x^{(i)}}}Y,\nabla_{{x^{(i)}}}Z)$ and $\nabla g = (\nabla_{x^{(i)}} g, \nabla_y g,\nabla_z g)$ .
			\end{prop}
			Before proving Proposition \ref{ClassYM}, we first highlight some important integrability properties of BMO martingales within our framework. For more details, we refer the interested reader to \cite{Kazamaki}. These properties are essential in the proofs of several results presented in this work.
			\begin{remark}\label{rem3.1}
				Let assumptions of Proposition \ref{ClassYM} be in force. 
				\begin{enumerate}
					\item Let $(e_t)_{t\geq 0}$ be the process defined by $e_t = \exp(\int_0^t \nabla_y g(s,\bold{X}^{(i)}_s)\mathrm{d}s)$.  From Assumption \ref{assum2.1} $g$ is not uniformly Lipschitz continuous in $y$ (as opposed to  \cite{Zhou,MaZhang01,Zhang041}), then the process $(e_t)_{t\geq 0}$ is not bounded. Hence for all $0\leq t\leq s \leq T$ we have, $e_se_t^{-1} \leq A_T:= \exp\Big(2\int_0^T \Lambda_y(1+ |Z_u|^{\alpha}\mathrm{d}u\Big)$, with $\alpha\in (0,2)$ and from the $\bmo$ property of $Z$, we deduce that $(e_t)_{t\geq 0} \in \mathcal{S}^p $ for all $p\geq 1$.
					\item The Dol\'{e}ans-Dade exponential $\mathcal{E}(\nabla_z g *W) \in L^{p_0}$ and $\mathcal{E}(\nabla_z g *W)^{-1}= \mathcal{E}(-\nabla_z g *W^{\mathbb{Q}})  \in L^{p_1}$ satisfy the reverse H\"{o}lder inequality and we denote by $p_0^{'}$ and $p_1^{'}$ the conjugates of $p_0$ and $p_1$, respectively. Moreover, the process $W^{\mathbb{Q}}_t = W_t - \int_0^t \nabla_z g(s,\bold{X}^{(i)}_s)\mathrm{d}s$ is a Brownian motion under the measure $\mathbb{Q}$ defined by $\mathrm{d}\mathbb{Q} := \mathcal{E}(\nabla_z g *W)^{-1}\mathrm{d}\mathbb{P}$. 
				\end{enumerate}
			\end{remark}
			\begin{proof}[Proof of Proposition \ref{ClassYM}]
				From \cite[Theorem 3]{NilssenProske}, the forward process $X^{(i)}$ is Sobolev differentiable for each $1\leq i \leq r,$ since  the drift $b^{i}$ is bounded and measurable for each $1\leq i \leq r, $  i.e. $X^{(i)} \in L^2(\Omega,W^{1,p}_{loc})$.
				Thanks to Theorem \ref{th 2.5}, for each $1\leq i \leq r$, the process $(\nabla_{{x^{(i)}}} X_T^{(i)})$ is continuous, only as an $L^p(\Omega)$ valued function thus fails to satisfy the condition (C3) of  \cite[Assumption 4.9]{ImkRhOliv24}. Nevertheless, a careful inspection of the proof in \cite{ImkRhOliv24} shows that the continuity requirement of the terminal value $\nabla \xi$ can be relaxed to the continuity as an $L^p$-valued function in order to guarantee the continuity in the Banach space  $\mathcal{S}^p\times\mathcal{H}^p$ of the partial derivatives. On the other hand, from the assumptions of the theorem, for all $p>1$ we deduce
				\[ \sup_{\bold{x}}\mathbb{E} \left[ \big| \nabla_{{x^{(i)}}} \Phi(X_T^{(1)},X_T^{(2)},\cdots,X_T^{(r)})\nabla_{{x^{(i)}}} X_T^{(i)}    \big|^p \right] \leq \Lambda_{\Phi}^p \sup\mathbb{E} \|\nabla_{{x^{(i)}}} X_T^{(i)}\|^p < \infty. \]
				Hence, the condition (C2) in \cite[Assumption 4.9.]{ImkRhOliv24} is satisfied by the terminal value $\nabla \xi$. 
				In addition, one can also check that the drivers of the BSDE \eqref{2.2} satisfy condition (C1) in \cite[Assumption 4.9.]{ImkRhOliv24}. Thus the solution $(X^{(i)},Y,Z)$ to the FBSDE \eqref{2.1}-\eqref{2.2} is differentiable with respect to $x^{i}= (x_1^{(i)},\cdots,x_d^{(i)})$ for all $1\leq i \leq r$ and the derivative process $(\nabla_{x^{(i)}}X^{(i)},\nabla_{x^{(i)}}Y,\nabla_{x^{(i)}}Z)$ is the unique solution to the linear equation \eqref{eq31}. This proof is completed.
			\end{proof}
			Let us turn now to the variational differentiability in the sense of Malliavin.
			\begin{prop}\label{Maldiff}
				Let Assumptions of Proposition \ref{ClassYM} be in force. Then the couple process $ (Y,Z)$ solution to \eqref{2.2} is Malliavin differentiable and the derivative process $(D_u Y_t,D_u Z_t)$ solves the affine BSDE for all $0\leq u \leq t \leq T.$
				\begin{align}\label{5.15}
					D_u Y_t =& \sum_{i=1}^{r} \nabla_{{x^{(i)}}} \phi(X_T^{(1)},X_T^{(2)},\cdots,X_T^{(r)})D_u X_T^{(i)} - \int_{t}^{T} D_u Z_s \mathrm{d}B_s\notag\\
					&+ \int_{t}^{T} \langle \nabla g(s,{X}^{(i)}_s,Y_s,Z_s),D_u \bold{X}_s^{(i)} \rangle \mathrm{d}s~, 
				\end{align}
				where, $D_u \bold{X}_s^{(i)}= (D_u{X}_s^{(i)}, D_uY_s,D_uZ_s)$, $\nabla g = (\nabla_{x^{(i)}} g, \nabla_y g,\nabla_z g)$.

				In addition, the following representations hold for any $u \in [0,t]$
				\begin{align}
					D_uY_t &= \sum_{i=1}^r ( \nabla_{{x^{(i)}}}Y_t)^{\bold{T}}( \nabla_{{x^{(i)}}}X_u^{(i)})^{-1}, \\
					Z_t=D_tY_t &= \sum_{i=1}^r ( \nabla_{{x^{(i)}}}Y_t)^{\bold{T}}( \nabla_{{x^{(i)}}}X_t^{(i)})^{-1},\, \text{a.s.},
				\end{align}
				where, $( \nabla_{{x^{(i)}}}X_t^{(i)})^{-1}$ stands for the inverse of the first variation process $ \nabla_{{x^{(i)}}}X_t^{(i)}$.
			\end{prop}
			\begin{proof}
				Let $(g^{\epsilon})_{\epsilon > 0}$ be a standard mollifier of the drivers $g$ such that for all $(t,x)\in [0,T]\times\mathbb{R}^d$, $(y,z)\in K$, $\lim_{\epsilon\rightarrow 0}\sup_{K}|g^{\epsilon}(t,x,y,z)- g(t,x,y,z)|=0 $, where $K$ is any compact subset of $\mathbb{R} \times \mathbb{R}^d$. Thus $(g^{\epsilon})_{\epsilon > 0}$ is infinitely differentiable with bounded derivatives of any order in $x,y$ and $z$. Then from classical result in the theory of BSDEs, the following equation with bounded terminal value $\xi= \phi(X_T^{(1)},\cdots,X_T^{(i)})$
				\begin{align}
					Y_t^{\epsilon} = \xi + \int_t^T g^{\epsilon}(s,X^{(i)}_s,Y_s^{\epsilon},Z_s^{\epsilon})\mathrm{d}s - \int_t^T Z_s^{\epsilon}\mathrm{d}W_s
				\end{align}
				has a unique Malliavin differentiable solution $(X^{(i)}_t,Y_t^{\epsilon},Z_t^{\epsilon} )\in \mathcal{S}^{2p}\times \mathcal{S}^{\infty}\times \mathcal{H}_{BMO}$ such that the BMO-norm of the martingale $Z^{\epsilon}*W:= \int_0^{\cdot}Z^{\epsilon}_s\mathrm{d}W_s$ does not dependent on $\epsilon > 0$. In addition, the derivative process $D_u \bold{X}_s^{(i,\epsilon)}=(D_uX^{(i)}_t, D_uY_t^{\epsilon},D_uZ_t^{\epsilon} )$ solves the linear equation
				\begin{align}\label{3.6}
					D_uY_t^{\epsilon} =& \sum_{i=1}^{r} \nabla_{{x^{(i)}}} \phi(X_T^{(1)},X_T^{(2)},\cdots,X_T^{(r)})D_u X_T^{(i)} - \int_{t}^{T} D_u Z_s^{\epsilon} \mathrm{d}W_s\notag\\
					&+ \int_{t}^{T} \langle \nabla g^{\epsilon}(s,{X}^{(i)}_s,Y_s^{\epsilon},Z_s^{\epsilon}),D_u \bold{X}_s^{(i,\epsilon)} \rangle \mathrm{d}s.
				\end{align}
				From \cite[Lemma 3.7]{ImkRhOliv24} for $p> 1$, with $p>p_0$ ($p_0$ is such that $\mathcal{E}(Z^{\epsilon}*W) \in L^{p_0}$) there is $q \in (1,\infty)$ only depending on $T,p$ and $\|Z^{\epsilon}*W\|_{\bmo}$ such that
				\begin{align*}
					&\mathbb{E}\Big[ \sup_{t  \in [0,T]} |D_uY^{\epsilon}_t|^{2p} + \Big( \int_0^T |D_uZ^{\epsilon}_t|^2\mathrm{d}t\Big)^p   \Big]^q  \notag\\
					\leq & C \mathbb{E}\Big[ \Big| \sum_{i=1}^{r} \nabla_{{x^{(i)}}} \phi(X_T^{(1)},X_T^{(2)},\cdots,X_T^{(r)})D_u X_T^{(i)}\Big|^{2pq} + \Big(\int_0^T \nabla_{x^{(i)}}g^{\epsilon}(t,\bold{X}_t^{(i,\epsilon)}) D_uX_t^{(i,\epsilon)} \mathrm{d}t \Big)^{2pq}      \Big]\\
					\leq & C \Big(\sum_{i=1}^{r}\Lambda_{\phi}^{(i)}\Big)^{2pq}\sup_{i}\mathbb{E}|D_u X_T^{(i)}|^{2pq} + C \mathbb{E}\Big( \sup_{ u\leq t\leq T}|D_u X_t^{(i)}| \int_0^T (1+ |Y_t^{\epsilon}| + \ell^{\epsilon}(|Y_t^{\epsilon}|) |Z_t^{\epsilon}|^{\alpha})\mathrm{d}t  \Big)^{2pq},
				\end{align*}
				where the last inequality comes from the properties of the drivers $g^{\epsilon},$ $\ell^{\epsilon}$ stands for the mollifier of the function $\ell$. Using Young's inequality and the uniform bound of $Y^{\epsilon}$ we deduce that 
				\begin{align}
					&\mathbb{E}\Big[ \sup_{t  \in [0,T]} |D_uY^{\epsilon}_t|^{2p} + \Big( \int_0^T |D_uZ^{\epsilon}_t|^2\mathrm{d}t\Big)^p   \Big]^q  \notag\\
					\leq 	&C \Big(\sum_{i=1}^{r}\Lambda_{\phi}^{(i)}\Big)^{2pq}\sup_{i}\mathbb{E}|D_u X_T^{(i)}|^{2pq} + C \mathbb{E} (\sup_{u\leq t \leq T}|D_u X_t^{(i)}|^{4pq}) + C\mathbb{E}\Big( \int_0^T (1+ |Y_t^{\epsilon}| + |Z_t^{\epsilon}|)\mathrm{d}t  \Big)^{4pq}\notag\\
					\leq	& C(p,q,T,\|Z^{\epsilon}*W\|_{\bmo})(1+ S(\phi)^{2pq}), 
				\end{align}
				which is finite uniformly in $\epsilon > 0$. We deduce that $\sup_{\epsilon > 0}\mathbb{E}\int_0^T|D_uY_t^{\epsilon}|\mathrm{d}t < \infty$, and by construction, for each $t \in [0,T]$, $Y_t^{\epsilon}$ converges to $Y_t$ in $L^2$ as $\epsilon$ goes to $0$. This implies that the backward component $Y$ of equation \eqref{2.2} is Malliavin differentiable i.e., $Y_t \in \mathbb{D}^{1,2}(\mathbb{R})$ for all $t \in [0,T]$ (see \cite[Lemma 1.2.3]{Nua06}). Next, we establish the Malliavin differentiability of the control process $Z$. It suffice to prove that the stochastic integral  $A = \int_0^T Z_t\mathrm{d}W_t$ is Malliavin differentiable (see \cite[Lemma 1.3.4]{Nua06}). Letting $A^{\epsilon} = \int_0^T Z^{\epsilon}_t\mathrm{d}t$, we clearly have  $A^{\epsilon}\rightarrow A$ in $L^2(\Omega,[0,T])$ as $\epsilon \rightarrow 0$ since $Z^{\epsilon} \rightarrow Z$ in $\mathcal{H}^2$. Moreover, for all $u\in [0,T]$ we have 
				\begin{align*}
					\sup_{\epsilon > 0}\mathbb{E}\Big[\int_0^T |D_uA^{\epsilon}|^2\mathrm{d}u\Big]=& \mathbb{E}\Big[ \int_0^T \Big| Z_u^{\epsilon} + \int_u^T D_u Z^{\epsilon}_t\mathrm{d}W_t \Big|^2 \mathrm{d}u \Big]\\
					\leq	&  2 \mathbb{E}\Big[ \int_0^T |Z_u^{\epsilon}|^2 \mathrm{d}u \Big] + 2 \mathbb{E}\Big[  \int_0^T  \int_u^T | D_u Z^{\epsilon}_t|^2 \mathrm{d}t\mathrm{d}u \Big].
				\end{align*}
				Using the fact $Z_t^{\epsilon} \in \mathbb{D}^{1,2}$ for $\text{a.e. } t \in [0,T]$ and the above estimate, we deduce from \cite[Lemma 1.3.4]{Nua06} that $A$ is Malliavin differentiable, and thus $Z_t \in \mathbb{D}^{1,2}$ for $\text{a.e.} t \in [0,T]$. By applying the dominated convergence theorem,  one can prove that each term in \eqref{3.6} converges to its corresponding counterpart in equation \eqref{5.15}.
				
				From \cite{BMBPD17} the following representation holds:
				$$
				D_uX_t^{(i)} = (\nabla_{x^{(i)}} X_t^{(i)})) (\nabla_{x^{(i)}} X_u^{(i)})^{-1}
				$$
				for all $0\leq u \leq t\leq T$  and for all $0\leq i \leq r$. Furthermore, from the uniqueness of solutions to equation \eqref{5.15} we also deduce that \[ D_uY_t = \sum_{i=1}^r (\nabla_{x^{(i)}} Y_t)^{\bold{T}}  (\nabla_{x^{(i)}} X_u^{(i)})^{-1}.\]
				Moreover, for each $0\leq v\leq t \leq T$, we have $Y_t - Y_v = -\int _v^t g(s,X_s^{(i)},Y_s,Z_s)\mathrm{d}s + \int_v^t Z_s\mathrm{d}W_s$ then differentiating both sides of the above equation, we get 
				$$
				D_uY_t = -\int _u^t D_ug(s,X_s^{(i)},Y_s,Z_s)\mathrm{d}s + Z_u + \int_u^t D_uZ_s \mathrm{d}W_s
				$$ 
				for all $0\leq v \leq u\leq t \leq T$. In particular, for $u=t$ we obtain that 
				\[ Z_t= \sum_{i=1}^r (\nabla_{x^{(i)}} Y_t)^{\bold{T}}  (\nabla_{x^{(i)}} X_t^{(i)})^{-1}.\]
				This conclude the proof.
			\end{proof}

			

			\subsection{Quadratic BSDEs with discrete functional terminal value}
			
			In this subsection, we present some essential properties of solutions to quadratic BSDEs with a discrete functional terminal value. We begin by  discussing their differentiability. The following lemma, derived from Propositions \ref{ClassYM} and \ref{Maldiff}, is straightforward, so we omit its proof here (for comparison, see \cite[Lemma 3.2.8]{Zhou}).
			
			\begin{lemm}\label{lemm4.1}
				Let us assume that the conditions of Proposition \ref{ClassYM} are in force except that the terminal value $\xi$ is of the form $\xi = \phi(X_{s_0},\cdots,X_{s_{r-1}},X_{s_{r}})$ with $0\leq s_0\leq s_1\leq \cdots\leq s_{r-1}\leq s_r = T$, satisfies \eqref{2.6}--\eqref{2.7} and $X^{(i)} = X^{(j)}$ for all $i,j \in \{1,\cdots,r\}$.  Then there exists a unique solution $(\nabla^{j}Y,\nabla^{j}Z)$ to the BSDE 
				\begin{align}\label{BSDE310}
					\nabla^{j}Y_t =& \sum_{i\geq j}^{r} \nabla_{{x^{(i)}}} \phi(X_{s_0},\cdots,X_{s_{r-1}},X_{s_{r}})\nabla_x X_{s_i} - \int_{t}^{T} \nabla^{j} Z_s \mathrm{d}B_s\notag\\
					&\quad + \int_{t}^{T} \langle \nabla g(s,\bold{X}_s),\nabla_{x} \bold{X}_s \rangle \mathrm{d}s.
				\end{align}
				In addition, the following representations hold 
				\begin{align*}
					D_uY_t &= \sum_{i=0}^r y_i(t)\left(\nabla_{x}X_u\right)^{-1} \bold{1}_{\{u\geq s_i\}}, \quad 0\leq u \leq t,\\
					Z_t &= D_tY_t = \sum_{i=0}^r y_i(t)\left(\nabla_{x}X_t\right)^{-1} \bold{1}_{\{t\geq s_i\}},
				\end{align*}
				where the operator $D$ denotes the Malliavin derivative.
				Moreover, given a fixed time $t$ such that $t \in [s_j,s_{j+1}),0\leq j \leq N$ with the convention $[s_r,s_{r+1})=\{T\}$, we have
				\begin{equation}\label{rep4.5}
					D_uY_t = \left(\nabla^{j}Y_t\right)^{\bold{T}}\left( \nabla_xX_u\right)^{-1},\quad Z_t = \left(\nabla^{j}Y_t\right)^{\bold{T}}\left( \nabla_xX_t\right)^{-1},
				\end{equation}
				and for all $1\leq i \leq r$,  $(y_0,z_0)$ and $(y_i,z_i)\in \mathbb{R}^m\times \mathbb{R}^{d\times m}$ are the adapted solutions to the following linear BSDEs 
				\begin{align}
					y_0(t) &= - \int_t^T z_0(s)\mathrm{d}W_s  + \int_t^T \left( \nabla_y g(s,\bold{X}_s)y_0(s) + \nabla_z g(s,\bold{X}_s)z_0(s) \right) \mathrm{d}s \notag\\
					&\quad + \int_t^T  \nabla_x g(s,\bold{X}_s)\nabla{X}_s\mathrm{d}s, \label{eq4.1}\\
					y_i(t) &= \xi_i - \int_t^T z_i(s)\mathrm{d}W_s + \int_t^T \left(  \nabla_y g(s,\bold{X}_s)y_i(s) + \nabla_z g(s,\bold{X}_s)z_i(s) \right)\mathrm{d}s, \label{eq4.2}
				\end{align}
				respectively, and $\xi_i = \left( (\nabla_{x}\phi)(X_{s_0},\cdots,X_{s_{r-1}},X_{s_{r}}) \right)^{\bold{T}} \nabla_{x}X_{s_i} \in \mathbb{R}^m$.
			\end{lemm}
			\begin{remark}\label{rem4.2}
				From Girsanov's theorem and the reverse H\"{o}lder inequality we have for all $0\leq i \leq r$
				\begin{align*}
					\mathbb{E}^{\mathbb{Q}}\Big[ \int_0^T |e_se_t^{-1}|^2|z_i(s)|^2\mathrm{d}s\Big]
					\leq C(p_0) \mathbb{E}\Big[\sup_{ 0\leq s\leq T}[|e_se_t^{-1}|^{4p_0^{'}}]+ \Big(\int_0^T|z_i(s)|^2\mathrm{d}s\Big)^{2p_0^{'}} \Big]^{\frac{1}{p_0^{'}}} < \infty.
				\end{align*}
				Then, the stochastic integral $\int_0 e_se_t^{-1} z_i(s) \mathrm{d}W_s^{\mathbb{Q}}$ defines a true martingale under the probability measure $\mathbb{Q}$.
				From the classical linearisation technique, the equations \eqref{eq4.1} and \eqref{eq4.2} can be rewritten as
				\begin{align}
					y_0(t) &=- \int_{t}^{T} e_se_t^{-1} z_0(s) \mathrm{d}W_s^{\mathbb{Q}}
					+ \int_{t}^{T} e_se_t^{-1} \nabla_{{x}} g(s,\bold{X}_s)\nabla_x {X}_s  \mathrm{d}s, \label{eq4.4}\\
					y_i(t) &= e_Te_t^{-1} y_i(T)- \int_{t}^{T} e_se_t^{-1} z_i(s) \mathrm{d}W_s^{\mathbb{Q}}.\label{eq4.5}
				\end{align}
			\end{remark}
			The following lemma links the estimate of a sum of quadratic variation processes to the sum of the terminal values of martingales. A version of this result is provided in \cite[Lemma 3.3]{Zhang041} for Lipschitz continuous coefficients. For technical reasons, we will reproduce the version developed in \cite[Lemma 3.2.10]{Zhou}	\begin{lemm}\label{lemmZhou}
				Let $y_i(t) \in L^{2p}(\Omega,\mathbb{R}^m)$ be an $\mathbb{R}^m\text{-valued}$ adapted martingale for all $0\leq i \leq r$ and $p\geq 1$. Then by the martingale representation theorem, there is a square integrable process $z_i(t)$ such that, $y_i(t) = \beta_i + \int_0^t z_i(s)\mathrm{d}W_s$. For any increasing function $\eta:\{0,1,\cdots,N-1\}\rightarrow \{0,\cdots,r\}$, we define $\tilde{\eta}:\{0,\cdots,r\} \rightarrow \{-1,0,1,\cdots,N-1\}$ as $\tilde{\eta}=\sup\{j; \eta(j)\leq i\}$ with the convention $\sup\emptyset = -1.$ Then, there is a constant $C(p,m)$ independent of $r$ such that 
				\begin{align*}
					\mathbb{E}\Big[\Big( \sum_{j=0}^{N-1} \int_{t_j}^{t_{j+1}} \Big| \sum_{i=\eta(j)}^r z_i(s) \Big|^2\mathrm{d}s   \Big)^p\Big] \leq C(p,m)\mathbb{E}\Big[ \Big|\sum_{i=0}^r y_i(t_{\tilde{\eta}(i)+1}) - y_i(0)\Big|^{2p}  \Big] .
				\end{align*}
			\end{lemm}
			The following result addresses the integrability of the supremum norm of the process $Z$, which is crucial for proving the main results of this paper. Its extension to the case of functional terminal values will be discussed in Remark \ref{remext1}.
			\begin{lemm} \label{boundZ} Let assumptions of Lemma \ref{lemm4.1} hold. Then for any $p >p_0'> 1$ there is a constant $C$ only depending on $p,d,T$ and $\|b\|_{L^{\infty}([0,T];C_b(\mathbb{R}^d))}$ such that 
				\begin{equation}\label{bound3.5}
					\mathbb{E}\big[\sup_{ 0\leq t\leq T}|Z_t|^p\big] \leq C(1 + S^p({\phi})),
				\end{equation}
				where $p_0'$ is the conjugate of $p_0$ and $p_0$ is such that $\mathcal{E}\left(\nabla_z g*W\right) \in L^{p_0}$(see Remark \ref{rem3.1}).
				In particular, if $g$ is Lipschitz continuous in $y$ then the control process $Z$ is uniformly bounded i.e. there is a constant $C>0$ such that
				\begin{align*}
					|Z_t|\leq Ce^{\|\nabla_y g\|_{\infty}}(1+S(\varphi)).
				\end{align*} 
			\end{lemm}

			\begin{proof} 
				First, we assume that the coefficients $b,g$ and $\xi$ are continuously differentiable and the terminal value $\xi$ is of the form $\xi = \phi(X_{s_0},\cdots,X_{s_{r-1}},X_{s_{r}})$ with $0\leq s_0\leq s_1\leq \cdots\leq s_{r-1}\leq s_r = T$, satisfying \eqref{2.6}-\eqref{2.7}. From Proposition \ref{ClassYM}, the solution $(Y,Z)$ is differentiable, and the derivative process $(\nabla^j Y,\nabla^j Z)$ satisfies the linear BSDE \eqref{BSDE310} for all $1\leq j \leq  r$. To derive the desired estimate, it is enough to establish a bound for $|(\nabla^j Y_t)^{\bold{T}}  (\nabla X_t)^{-1}|$ for all $1\leq j \leq r$. Given the assumption on the driver $g$, we have, $|\nabla_z g(s,\bold{X}_s)| \leq \Lambda_z(1+ 4\ell(|Y_s|)|Z_s|) \in \mathcal{H}_{BMO}.$ Using a standard linearisation technique and applying It\^{o}'s formula, we deduce 
				\begin{align}\label{3.8}
					e_t\nabla^j Y_t &= e_T \sum_{i= j}^{r} \nabla_{{x^{(i)}}} \phi(X_{s_0},\cdots,X_{s_{r-1}},X_{s_{r}})\nabla_x X_{s_i} - \int_{t}^{T} e_s \nabla Z_s \mathrm{d}W_s^{\mathbb{Q}} \notag\\
					&\quad+ \int_{t}^{T} e_s \nabla_x g(s,{X}_s,Y_s,Z_s)\cdot\nabla {X}_s  \mathrm{d}s,
				\end{align}
				where the measure $\mathbb{Q}$ with its associated Radon-Nikodym derivative $\mathrm{d}\mathbb{Q} := \mathcal{E}(\nabla_z g *W)\mathrm{d}\mathbb{P}$ and $W^{\mathbb{Q}}$ are all defined in Remark \ref{rem3.1}. 
				By taking the conditional expectation in \eqref{3.8}
				and using the fact that the process $(\nabla {X}_t)^{-1}$ is $\mathfrak{F}_t\text{-measurable}$ we obtain that 
				\begin{align*}
					&\big|\nabla^j Y_t (\nabla {X}_t)^{-1}\big|\\
					\leq & |(\nabla {X}_t)^{-1}|\mathbb{E}^{\mathbb{Q}}\Big[ A_T \Big(\big|\sum_{i=j}^{r} \nabla_{{x^{(i)}}} \phi(X_{s_0},\cdots,X_{s_{r-1}},X_{s_{r}})\nabla_x X_{s_i}\big|\\
					&+ \Big|\int_{t}^{T} \nabla_x g(s,{X}_s,Y_s,Z_s) \cdot\nabla {X}_s \mathrm{d}s\Big|\Big) \Big/ \mathfrak{F}_t\Big]
				\end{align*}
				Using $p_0' \in (1,p)$ we obtain from the reverse H\"{o}lder inequality the existence of a constant $C$ depending on $p_0$ with $1/p_0 + 1/p_0^{'} =1$ such that 
				\begin{align*}
					&\big|\nabla^j Y_t (\nabla {X}_t)^{-1}\big| \notag\\
					&\leq C(p_{0}) |(\nabla {X}_t)^{-1}|\mathbb{E}\Big[ A_T^{p_0'} \Big(\Big|\sum_{i=j}^{r} \nabla_{{x^{(i)}}} \phi(X_{s_0},\cdots,X_{s_{r-1}},X_{s_{r}})\nabla_x X_{s_i}\Big| \\
					&\quad+ \Big|\int_{t}^{T} \nabla_x g(s,{X}_s,Y_s,Z_s) \cdot\nabla {X}_s \mathrm{d}s\Big|\Big)^{p_0'} \Big/ \mathfrak{F}_t\Big]^{\frac{1}{p_0'}}.
				\end{align*}
				Therefore, 
				\begin{align*}
					&\mathbb{E}\Big[ \sup_{ 0\leq t\leq T}\big|\nabla^j Y_t (\nabla {X}_t)^{-1}\big|^p\Big]\\
					&\leq C(p,p_0') \mathbb{E}[\sup_{ 0\leq t\leq T}|(\nabla {X}_t)^{-1}|^{2p}]
					+ C(p,p_0') \mathbb{E}\Bigg[\sup_{ 0\leq t\leq T}\mathbb{E}\Big[ A_T^{p_0'} \Big(\Big|\sum_{i=j}^{r} \nabla_{{x^{(i)}}} \phi(X_{s_0},\cdots,X_{s_{r-1}},X_{s_{r}})\nabla_x X_{s_i}\Big| \\
					&\quad+ \Big|\int_{t}^{T} \nabla_x g(s,{X}_s,Y_s,Z_s) \cdot\nabla {X}_s \mathrm{d}s\Big|\Big)^{p_0'} \Big/ \mathfrak{F}_t\Big]^{\frac{2}{p_0'}}\Bigg]^p
				\end{align*}
				Successive application of the Doob's maximal and H\"{o}lder's inequalities gives 
				\begin{align*}
					&\mathbb{E}\Big[ \sup_{ 0\leq t\leq T}\big|\nabla^j Y_t (\nabla {X}_t)^{-1}\big|^p\Big]\\
					&\leq C(p,p_0') \mathbb{E}[\sup_{ 0\leq t\leq T}|(\nabla {X}_t)^{-1}|^{2p}]
					+ C(p'_{0},p_0,p) \mathbb{E}\Big[  A_T^{p} \Big\{\Big|\sum_{i=j}^{r} \nabla_{{x^{(i)}}} \phi(X_{s_0},\cdots,X_{s_{r-1}},X_{s_{r}})\nabla_x X_{s_i}\Big| \\
					&\quad+ \Big|\int_{0}^{T} \nabla_x g(s,{X}_s,Y_s,Z_s) \cdot\nabla {X}_s \mathrm{d}s\Big|\Big\}^{2p}\Big]\notag\\
					&\leq C(p,p_0') \mathbb{E}[\sup_{ 0\leq t\leq T}|(\nabla {X}_t)^{-1}|^{2p}]
					+  C(p'_{0},p_0,p) \mathbb{E}\Big[ \Big\{\Big|\sum_{i=j}^{r} \nabla_{{x^{(i)}}} \phi(X_{s_0},\cdots,X_{s_{r-1}},X_{s_{r}})\nabla_x X_{s_i}\Big| \\
					&\quad+ \Big|\int_{0}^{T} \nabla_x g(s,{X}_s,Y_s,Z_s) \cdot\nabla {X}_s \mathrm{d}s\Big|\Big\}^{4p}\Big]^{\frac{1}{2}}.	
				\end{align*}
				By using the growth of $g$ we obtain that 
				\begin{align}
					&\mathbb{E}\Big[ \sup_{ 0\leq t\leq T}\big|\nabla^j Y_t (\nabla {X}_t)^{-1}\big|^p\Big]\notag\\
					&\leq C(p,p_0') \mathbb{E}[\sup_{ 0\leq t\leq T}|(\nabla {X}_t)^{-1}|^{2p}] + C(p'_{0},p_0,p) \mathbb{E}\Big[ \Big|\sum_{i=j}^{r} \nabla_{{x^{(i)}}} \phi(X_{s_0},\cdots,X_{s_{r-1}},X_{s_{r}})\nabla_x X_{s_i}\Big|^{4p} \notag\\
					&\quad+ \sup_{ 0\leq s\leq T} |\nabla {X}_s|^{8p} + \Big(\int_{0}^{T} \Lambda_x(1+ |Y_s| +\ell(|Y_s|)|Z_s|^{\alpha})  \mathrm{d}s\Big)^{8p} \Big]^{\frac{1}{2}}\notag\\
					&\leq C(p'_{0},p_0,p,T,\|Z*W\|_{\bmo},\|b\|_{L^{\infty}([0,T];C_b(\mathbb{R}^d))})(1 + S^p(\phi) ),\label{bound416}
				\end{align}
				where we used Theorem \ref{th 2.5}, equation \eqref{5.12}, the boundedness of $Y$ and the energy inequality to obtain the last inequality. By applying a standard approximation procedure, one can show that the bound remains valid even when the coefficients are not differentiable. 
				
				For the second point, if $g$ is Lipschitz continuous in $y$ then the process $A_T$ is uniformly bounded by $e^{\|\nabla_y g\|_{\infty}}$. There exists a constant $C>0$ such that
				\begin{align*}
					&\big|\nabla^j Y_t (\nabla {X}_t)^{-1}\big|\\
					&\leq Ce^{\|\nabla_y g\|_{\infty}}\Big(  S(\varphi)+\mathbb{E}^{\mathbb{Q}}\Big[\int_{t}^{T}(1+\ell(|Y_s|)|Z_s|^{\alpha})^2\mathrm{d}s \Big/ \mathfrak{F}_t\Big]^{\frac{1}{2}}\Big)\mathbb{E}^{\mathbb{Q}}\Big[ \sup_{t\leq s\leq T}|\nabla {X}_s(\nabla{X}_t)^{-1}|^2 \Big/ \mathfrak{F}_t\Big]^{\frac{1}{2}}.\\
					&\leq Ce^{\|\nabla_y g\|_{\infty}}(1+S(\varphi)),
				\end{align*}
				where we applied the fact that $Y$ is bounded, the energy's inequality and Proposition \ref{cor 5.7}  to obtain the last inequality, $C>0$ is a constant only depending on T and $\|Z*W\|_{\bmo}$ and $\|b\|_{L^{\infty}([0,T];C_b(\mathbb{R}^d))}$.  
				This concludes the proof.
			\end{proof}

			\begin{remark}\label{remext1}\leavevmode
				\begin{itemize}
					\item The above lemma remains valid for the general case when $\xi = \Phi(X.)$ by using standard stability arguments of BSDEs and the Fatou's Lemma. This is achieved by considering the bounded smooth approximation $\{\hat{\phi}^{\pi}\}_{\pi}$ of $\Phi(X)$ (as discussed in Remark \ref{remark24}), with the partition $\{\pi\}$ of $[0,T]$, satisfying $|\pi|\rightarrow 0$ and choosing $S(\varphi)< C\Lambda_{\Phi}$. In other words, for any $p >p_0'> 1$ there exists a constant $C$ depending only on $p,d,T$ and $\|b\|_{L^{\infty}([0,T];C_b(\mathbb{R}^d))}$ such that 
					\begin{equation}\label{boundZ0}
						\mathbb{E}\big[\sup_{ 0\leq t\leq T}|Z_t|^p\big] \leq C(1 + \Lambda^p_{\Phi}),
					\end{equation}
					where $p_0'$ is the conjugate of $p_0$ and $p_0$ is such that $\mathcal{E}\left(\nabla_z g*W\right) \in L^{p_0}$(see Remark \ref{rem3.1}).
					\item Similarly, for $\xi = \Phi(X.)$ and if $g$ is uniformly Lipschitz continuous in $y$, we obtain the following uniform bound of $Z_t$ for all $t\in [0,T]$ 
					\begin{equation}\label{boundZ1}
						|Z_t| \leq Ce^{\|\nabla_y g\|_{\infty}}(1 + \Lambda_{\Phi}).
					\end{equation}
					\item In the context of path-dependent FBSDE,	the result above was first derived in \cite[Lemma 3.2]{Zhang041} under the Cauchy-Lipschitz framework with an additional $C^{1/2,1}$ assumption on the coefficients. Recently, \cite[Lemma 3.2.7]{Zhou}, extended this to quadratic FBSDEs with $(t,y,z)\mapsto g(t,y,z) \in C^{1/2,1,1}$ and further assumed that the drift $b\in C^{1+\alpha}$ for some $\alpha >0$. Beyond the $L^{\infty}\text{-Lipschitz}$ (or $L^1\text{-Lipschitz}$) condition on the terminal value, the current result applies to quadratic FBSDEs with non-differentiable coefficients and non-smooth drift.
				\end{itemize}
			\end{remark}
			
			\subsection{Representation theorem revisited}
			
			In this subsection, we assume that the diffusion process $X$ in \eqref{2.1} starts at $x \in \mathbb{R}^d$ at a fixed time $s \in [0,T]$, and we consider the Brownian filtration $\mathfrak{F}_t^s = \sigma(W_u-W_s, s\leq u\leq t)$ with the usual $\mathbb{P}$ augmentation. We further assume that the coefficients $b$ and $g$ satisfy Assumption \ref{assum2.1}, and that the terminal value $\xi$ is of the form $\xi = \phi(X_{t_1},X_{t_2},\ldots,X_{t_n})$, where $\phi$ is any bounded and Lipschitz continuous function.
			
			\begin{thm} \label{threpr1}
				Suppose $\xi = \phi(X_{t_1},X_{t_2},\ldots,X_{t_n})$, satisfies \eqref{2.6}--\eqref{2.7}, where $t\leq t_1<t_2<\ldots<t_n\leq T$ is any partition of $[t,T]$. Suppose Assumption \ref{assum2.1} holds and let $(X,Y,Z)$ be the solution to the corresponding FBSDE \eqref{2.1}-\eqref{2.2}. Then, on each interval $(t_{i-1},t_i),  i=1,\ldots,n$ the following holds $\mathbb{P}\text{-almost surely}$
				\begin{align}\label{A5}
					Z_t = \mathbb{E}\Big[ \phi(X_{t_1},X_{t_2},\cdots,X_{t_n})N_{t_i}^t + \int_t^T g(v,X_v,Y_v,Z_v)N_{v\wedge t_i}^t\mathrm{d}v\Big/\mathfrak{F}_t^s \Big],
				\end{align}
				and \begin{equation}\label{A1}
					N^t_v = \frac{1}{v-t}(M_v^t)^{\bold{T}}(\nabla X_t)^{-1}, \quad 0\leq s \leq t < v \leq T ,
				\end{equation}
				and the two parameters process $M_v^s$ is given by 
				\begin{equation}\label{A2}
					M_v^t = \int_t^v (\nabla X_{\tau} )^{\bold{T}}\mathrm{d}W_{\tau}.
				\end{equation}
				$\nabla X = (\nabla^1 X,\cdots,\nabla ^d X)$ stands for the first variation of the diffusion $X$ as defined in Definition \ref{defi 2.8}. There exists a version of $Z$ that satisfy the following properties
				\begin{itemize}
					\item[(i)] the mapping $t\mapsto Z_t$ is a.s. continuous on each interval $(t_{i-1},t_i)$, $i=1,\cdots,n$;
					\item[(ii)] both limits $Z_{t_i-}:= \lim_{s \uparrow t_i}Z_s$ and $Z_{t_i+}:= \lim_{s \downarrow t_i}Z_s$ exist;
					\item[(iii)] for all $p>0$, there is a constant $C_p>0$ such that 
					\begin{align}\label{jump}
						\mathbb{E}|\Delta Z_{t_i}|^p \leq C_p.
					\end{align}
				\end{itemize}
				Consequently, the process $Z$ has both c\`adl\`ag and c\`agl\`ad versions, with discontinuities $t_0,\ldots,t_n$ and jump sizes satisfying \eqref{jump}.
			\end{thm}
			The proof of this theorem can be carried out using standard approximation techniques similar to those in \cite{ImkRhOliv24}. We will not reproduce the details here.
			We conclude this section by recalling the following useful result. 
			\begin{lemm}
				Let $T> 0$ be fixed and suppose that $X$ and $Y$ are two processes such that $X$ is a square integrable martingale and $Y$ is also square integrable. Then the following holds
				\begin{equation}\label{elem}
					\mathbb{E}\left[X_t Y_t \right] = \mathbb{E}\left[X_T Y_t \right], \text{ for all } 0\leq t \leq T.
				\end{equation}
			\end{lemm}
			\section{Proof of the main results}\label{secproofmain}
			This section is devoted to the proofs of the main results.
			\subsection{Proofs of Theorem \ref{thm1.1}}\label{sectprofmain1}
			This section is dedicated to the proofs of Theorems \ref{thm1.1}, \ref{thm1.2}, and Proposition \ref{cor 1.3}, below. As previously mentioned, the proof of Theorem \ref{thm1.1} follows from Theorem \ref{thm1.2} using approximation and stability results for quadratic BSDEs. One of the  challenges is to find an appropriate approximation for the path-dependent terminal value $\Phi$, as outlined for instance in Remark \ref{remark24}.
			\begin{proof}[Proof of Theorem \ref{thm1.2}]
				As in the proof of Lemma \ref{boundZ}, we will first assume that the coefficients $b,g,\xi$ are continuously differentiable.
				We recall that for all $ s\leq v \leq t $,
				\[ Y_s = \xi + \int_s^t g(v,X_v,Y_v,Z_v)\mathrm{d}v -\int_s^t Z_v \mathrm{d}B_v. \]
				Then from BDG inequality we deduce that 
				\begin{align*}
					\mathbb{E}\Big[ \sup_{s\leq r\leq t }|Y_r-Y_s|^{p}  \Big]\leq C(p) \Big\{ \mathbb{E}\Big( \int_s^t |g(v,X_v,Y_v,Z_v)|\mathrm{d}v  \Big)^{p}  + \mathbb{E}\Big( \int_{s}^{t} |Z_v|^2 \mathrm{d}v\Big)^{p/2}   \Big\}.
				\end{align*}
				By using the growth of $g$ and the fact that $Y_v$ is bounded we obtain that:
				\begin{align*}
					\mathbb{E}\Big[ \sup_{s\leq r\leq t }|Y_r-Y_s|^{p}  \Big] \leq C(p)\Big\{ |t-s|^{p} + \mathbb{E}\Big[ \Big( \int_{s}^{t} (1+ \ell(|Y_v|))|Z_v|^2\mathrm{d}v \Big)^{p} + \Big( \int_{s}^{t} |Z_v|^2 \mathrm{d}v\Big)^{p/2}  \Big]  \Big\}.
				\end{align*} 
				From to the local boundedness of $\ell$ and the bound \eqref{bound3.5}, we have
				\begin{align*}
					&\mathbb{E}\Big[ \sup_{s\leq r\leq t }|Y_r-Y_s|^{p}  \Big]\notag\\
					\leq& C(p)\Big\{ |t-s|^{p} + |t-s|^{p}\mathbb{E} [\sup_{s\leq v\leq t}|Z_v|^{2p}] +|t-s|^{p/2}\mathbb{E} [\sup_{s\leq v\leq t}|Z_v|^{2p}]  \Big\}\\
					\leq	& C(p)(1+S(\phi)^{2p})\Big\{ |t-s|^{p} + |t-s|^{p/2} \Big\}.
				\end{align*}
				Let us focus now on the regularity of the process $Z$. We recall that our main aim is to establish the following bound
				\[ 	\sum_{j=0}^{N-1}\mathbb{E}\Big(\int_{t_j}^{t_{j+1}}|Z_t - Z_{t_j}|^2 \mathrm{d}t \Big)^p \leq C h^p ,\]
				where   $\pi_0: 0=s_0 <\cdots < s_{r-1}<s_r=T$ is a partition of $[0,T]$, 
				and $\pi_1: 0=t_0 <t_1<\cdots < t_N=T$ is another partition of $[0,T]$  and $h= \sup_j|t_{j+1}-t_j|$.
				From Jensen's inequality,
				\begin{align*}
					\sum_{j=0}^{N-1}\mathbb{E}\Big(\int_{t_j}^{t_{j+1}}|Z_t - Z_{t_j}|^2 \mathrm{d}t \Big)^p \leq C h^{p-1} 	\sum_{j=0}^{N-1}\mathbb{E}\int_{t_j}^{t_{j+1}} |Z_t - Z_{t_j}|^{2p} \mathrm{d}t.
				\end{align*}
				Thus, it is sufficient to evaluate the quantity $\mathbb{E}|Z_t-Z_{t_j}|^{2p}$, for all $p>1$. We will establish the result for the case where $\pi_0$ is finer than $\pi_1$; the converse can be handled similarly without significant difficulty (see \cite{Rhoss23}). Assume $t_i= s_{l_i}$ for some increasing sequence $(l_i)_{i=1}^N$. For technical reasons, we will use of the expressions provided by equation \eqref{rep4.5}. Specifically, for any $t \in [s_k,s_{k+1}) \subset [s_{l_j},s_{l_{j+1}}) = [t_j, t_{j+1})$, we write
				\begin{align*}
					Z_t - Z_{t_j} &= \left(\nabla^{k}Y_t\right)^{\bold{T}}\left( \nabla_xX_t\right)^{-1} - \left(\nabla^{l_j}Y_{t_j}\right)^{\bold{T}}\left( \nabla_xX_{t_j}\right)^{-1}\\
					&= I_{t}^{1,j} + I_{t}^{2,j},
				\end{align*}
				where $I_{t}^{1,j}:= \left(\nabla^{l_j}Y_t\right)^{\bold{T}}\left( (\nabla_xX_t)^{-1} - (\nabla_xX_{t_j})^{-1} \right)$ and $I_{t}^{2,j}:=  \left(\nabla^{k}Y_t-\nabla^{l_j}Y_{t_j}\right)^{\bold{T}}\left( \nabla_xX_{t_j}\right)^{-1}$. Thus,
				\begin{align*}
					C_p h^{p-1}\sum_{j=0}^{N-1}\mathbb{E}\Big[ \int_{t_j}^{t_{j+1}}|Z_t - Z_{t_j}|^{2p} \mathrm{d}t \Big] \leq C_p h^{p-1} \sum_{j=0}^{N-1}\mathbb{E}\Big[ \int_{t_j}^{t_{j+1}}\left( |I_{t}^{1,j}|^{2p} + |I_{t}^{2,j}|^{2p} \right) \mathrm{d}t \Big].
				\end{align*} 
				The contribution of $\mathbb{E}[|I_{t}^{1,j}|^{2p}]$ can be derived similarly to the approach in \cite[Page 4786]{ImkRhOliv24}. For completeness, we briefly reproduce it here. Since $b$ is assumed to be continuously differentiable,
				the map $t\mapsto (\nabla_x X_t)^{-1}$ satisfies a linear equation. Therefore, for all $0 \leq t_j \leq t$ one can write
				\begin{align}\label{decomp}  
					&|(\nabla_x X_t)^{-1} - (\nabla_x X_{t_j})^{-1}| \notag\\
					& = \Big|\sum_{k=1}^{\infty} \int_{t_{j}<t_{j_1}<\ldots<t_{j_k}<t}  b'(t_{j_1},X_{t_{j_1}}^x):\cdots: b'(t_{j_k},X_{t_{j_k}}^x)\mathrm{d}t_{j_1}\cdots\mathrm{d}t_{j_k} \Big|,
				\end{align}
				where the symbol $``:"$ stands for the matrix multiplication.
				For any $p\geq 2$, H\"{o}lder's inequality and \eqref{rep4.5} yield
				\begin{align}
					&\mathbb{E}\left[ |I_{t}^{1,j}|^{2p}  \right]\leq  \mathbb{E}\Big[\sup_{ 0\leq t\leq T} |\nabla_x X_t|^{8p}\Big] ^{\frac{1}{4}} \mathbb{E}\Big[\sup_{ 0\leq t\leq T} |Z_t|^{8p}\Big] ^{\frac{1}{4}}  \mathbb{E}\Big[| (\nabla_x X_t)^{-1} -(\nabla_x X_{t_j})^{-1}|^{4p}\Big]^{\frac{1}{2}}.\notag
				\end{align}
				Substitute \eqref{decomp} into the above inequality and apply Theorem \ref{th 2.5}, the bound \eqref{bound3.5}, and  H\"{o}lder's inequality again to obtain with $h_j = |t_{j+1}-t_j|$
				\begin{align}\label{T_0}
					&\mathbb{E}\left[ |I_{t}^{1,j}|^{2p}  \right]\notag\\
					\leq & \mathbb{E}\Big[\sup_{ 0\leq t\leq T} |\nabla_x X_t|^{8p}\Big] ^{\frac{1}{4}} \mathbb{E}\Big[\sup_{ 0\leq t\leq T} |Z_t|^{8p}\Big] ^{\frac{1}{4}}\notag\\
					&\quad\times \mathbb{E}\Big[ \Big| \sum_{k=1}^{\infty} \int_{t_{j}<t_{j_1}<\ldots<t_{j_k}<t}  b'(t_{j_1},X_{t_{j_1}}):\cdots: b'(t_{j_k},X_{t_{j_k}})\mathrm{d}t_{j_1}\cdots\mathrm{d}t_{j_k}  \Big|^{4p} \Big]^{\frac{1}{2}}\notag\\
					\leq& C(1+\|b\|_{L^{\infty}([0,T];C_b(\mathbb{R}^d))})(1+S(\phi)^{2p}) \notag\\
					&\quad\times \mathbb{E}\Big[ \Big| \sum_{k=1}^{\infty} \int_{t_{j}<t_{j_1}<\ldots<t_{j_k}<t}  b'(t_{j_1},X_{t_{j_1}}):\cdots: b'(t_{j_k},X_{t_{j_k}})\mathrm{d}t_{j_1}\cdots\mathrm{d}t_{j_k}  \Big|^{4p} \Big]^{\frac{1}{2}}\notag\\
					\leq& C(1+\|b\|_{L^{\infty}([0,T];C_b(\mathbb{R}^d))})(1+S(\phi)^{2p}) h_j^{p},
				\end{align}
				where, we used  Girsanov's theorem and \cite[Proposition 3.7]{ MMNPZ13} to derive the last inequality. We emphasize that the constant above does not depend on the Lipschitz norm of the drift $b$, unlike in \cite{Zhang041} or \cite{Zhou}.
				Let us turn now to the control of $I^{2,j}_t$. 
				From Lemma \ref{lemm4.1}, we deduce that 
				\begin{align}\label{5.3}
					\nabla^{k}Y_t - \nabla^{l_j}Y_{t_j} = y_0(t)-y_0(t_j) + \sum_{i=k}^{r} \big(y_i(t)-y_i(t_j)\big) - \sum_{i\geq l_j}^{k-1} y_i(t_j).
				\end{align}
				Using the mean value theorem, Remark \ref{rem4.2}  and the definition of $W^{\mathbb{Q}}$, we observe that for $\tau \in [t_j,t]$
				\begin{align*}
					y_i(t)-y_i(t_j) &= e_{t_j}^{-1}\left[ \big(e_ty_i(t)-e_{t_j}y_i(t_j)\big) + (e_{t_j}- e_t) y_i(t) \right]\\
					&=- (t-t_j)\nabla_yg(\tau,\bold{X}_{\tau}) e_{\tau}e_{t_j}^{-1}y_i(t_j) + \int_{{t_j}}^t e_{\tau}e_{t_j}^{-1}z_i({\tau}) \mathrm{d}W_{\tau}\\&\quad - \int_{{t_j}}^t e_{\tau}e_{t_j}^{-1} \nabla_z g({\tau},\bold{X}_{\tau})z_i({\tau})\mathrm{d}{\tau} .
				\end{align*}	
				Similarly, for $\tau \in [t_j,t]$
				\begin{align*}
					y_0(t)-y_0(t_j) &=\int_{{t_j}}^t e_{\tau}e_{t_j}^{-1}z_0({\tau}) \mathrm{d}W_{\tau} - \int_{{t_j}}^t e_{\tau}e_{t_j}^{-1} \nabla_x g({\tau},\bold{X}_{\tau})\nabla X_{\tau}\mathrm{d}{\tau}\\
					&\quad - \int_{{t_j}}^t e_{\tau}e_{t_j}^{-1}z_0({\tau}) \nabla_z g({\tau},\bold{X}_{\tau})\mathrm{d}{\tau}.
				\end{align*}
				Thus, \eqref{5.3} becomes	
				\begin{align}
					\nabla^{k}Y_t - \nabla^{l_j}Y_{t_j}
					&=  -e_{\tau}e_{t_j}^{-1}\nabla_yg(\tau,\bold{X}_{\tau})(t-t_j)\sum_{i=k}^{r}  y_i(t_j)  - \sum_{i\geq l_j}^{k-1} y_i(t_j)  \notag\\
					&\quad - \sum_{i= k}^r \Big( \int_{{t_j}}^{t} e_{\tau} e_{t_j}^{-1} \nabla_{{z}} g({\tau},\bold{X}_{\tau})z_i({\tau})  \mathrm{d}{\tau}- \int_{{t_j}}^{t} e_{\tau}  e_{t_j}^{-1} z_i({\tau}) \mathrm{d}W_{\tau} \Big)\notag	\\
					&\quad - \int_{{t_j}}^{t} e_{\tau} e_{t_j}^{-1} \nabla_{{x}} g({\tau},\bold{X}_{\tau})\nabla_x {X}_{\tau}  \mathrm{d}{\tau}  -\int_{{t_j}}^{t} e_{\tau} e_{t_j}^{-1} \nabla_{{z}} g({\tau},\bold{X}_{\tau})z_0({\tau})  \mathrm{d}{\tau}\notag\\& \quad+ \int_{{t_j}}^{t} e_{\tau}  e_{t_j}^{-1} z_0({\tau}) \mathrm{d}W_{\tau}. 
				\end{align}
				Observe that the process $e_{t_j} = \exp(\int_0^{t_j}\Lambda_y(1+|Z_u|^{\alpha_0})\mathrm{d}u)$ is bounded from below by one. Then by taking the $p\text{-th}$ power, then the supremum, and applying the conditional BDG inequality, we deduce that
				\begin{align}\label{eqsup1}
					&\mathbb{E}\Big[\big|\nabla^{k}Y_t - \nabla^{l_j}Y_{t_j}\big|^{2p}/\mathfrak{F}_{t_j}\Big]\notag\\
					\leq & C_p  \mathbb{E}\Big[ \Big|(t_{j+1}-t_j)\nabla_yg(\tau,\bold{X}_{\tau})\sum_{i=k}^{r} e_{\tau}  y_i(t_j) \Big|^{2p} +  \Big|\sum_{i\geq l_j}^{k-1} e_{t_j}y_i(t_j)\Big|^{2p} \notag\\
					&+ \Big(\int_{{t_j}}^{t}\sum_{i\geq k}^r e_{\tau} \nabla_{{z}}g({\tau},\bold{X}_{\tau})z_i({\tau})\mathrm{d}{\tau} \Big)^{2p} + \Big( \int_{{t_j}}^{t_{j+1}} \Big|\sum_{i\geq k}^r e_{\tau} z_i({\tau})\Big|^2 \mathrm{d}{\tau} \Big)^p \notag\\
					& + \Big( \int_{{t_j}}^{t_{j+1}} \big| e_{\tau}  z_0({\tau})\big|^2 \mathrm{d}{\tau} \Big)^p +\Big( \int_{{t_j}}^{t_{j+1}} |e_{\tau}  \nabla_{{x}} g({\tau},\bold{X}_{\tau})\nabla_x {X}_{\tau}|  \mathrm{d}{\tau}\Big)^{2p}\notag\\
					&\quad + \Big( \int_{{t_j}}^{t_{j+1}} |e_{\tau} \nabla_{{z}} g({\tau},\bold{X}_{\tau})z_0({\tau})| \mathrm{d}{\tau}\Big)^{2p}
					\Big/\mathfrak{F}_{t_j}\Big].
				\end{align}
				Let us recall that the expression of $I_t^{2,j}$ is given by $I_{t}^{2,j}:=  \left(\nabla^{k}Y_t-\nabla^{l_j}Y_{t_j}\right)^{\bold{T}}\left( \nabla_xX_{t_j}\right)^{-1}$ and using the fact that the process $(\nabla_x {X}_{t_j})^{-1}$ is $\mathfrak{F}_{t_j}\text{-measurable}$, we deduce that
				\begin{align}\label{4.6}
					&\sum_{j=0}^{N-1}\mathbb{E}\Big[\int_{t_j}^{t_{j+1}}\big|I_t^{2,j}\big|^{2p}\mathrm{d}t\Big]\leq \sum_{j=0}^{N-1}\sum_{k=l_j}^{l_{j+1}-1}\mathbb{E}\Big[ \int_{s_k}^{s_{k+1}} \mathbb{E}\big[\big|\nabla^{k}Y_t - \nabla^{l_j}Y_{t_j}\big|^{2p}/\mathfrak{F}_{t_j}\big]|(\nabla X_{t_j})^{-1}|^{2p} \mathrm{d}t \Big],
				\end{align}
				which can be written as $\sum_{j=0}^{N-1}\mathbb{E}\Big[\int_{t_j}^{t_{j+1}}\big|I_t^{2,j}\big|^{2p}\mathrm{d}t\Big]\leq C_p \sum_{\nu=1}^{7} T_{\nu},$	and each term $T_{\nu}$ for $1\leq \nu \leq 7$ comes from \eqref{eqsup1}.

				We introduce another partition $\pi_2$ finer than $\pi_1$ defined by: $\pi_2: 0=\mu_0 < \mu_1 < \cdots < \mu_{\hat{\rho}_1}< \cdots < \mu_{\hat{\rho}_2}< \cdots \mu_{\hat{\rho}_N} = T$ and satisfying the following properties (see \cite{Zhou,Zhang041}).
				\begin{itemize}
					\item $\mu_{\hat{\rho}_i} = s_{l_i}= t_i$ with $\hat{\rho}_{i+1}-\hat{\rho}_i = \gamma$ for some $\gamma >0$ large enough, for all $0\leq i \leq r$;
					\item there exists $\rho_{k} = \frac{\mu_{\hat{\rho}_i+k+1}-\mu_{\hat{\rho}_i+k}}{t_{i+1}-t_i}$ independent of $i$ for all $0\leq k \leq \gamma -1$ such that $\sum_{k=0}^{\gamma-1}\rho_k =1$;
					\item $\pi_1$ belongs to $\pi_2$  and there exist a surjective function $\eta(i,k): \{0,\cdots,N-1 \}\times \{0,\cdots,\gamma\}\rightarrow \{0,1,\cdots,r\}$ such that $\mu_{\hat{\rho}_{i}+k} = s_{\eta(i,k)}$ for all $0\leq i\leq N-1$ and $0\leq k \leq \gamma-1$;
					\item $l_i \leq \eta(i,k) \leq l_{j+1}-1$ with $\eta(i,k+1) \geq \eta(i,k)$, $\eta(i+1,k)\geq \eta(i,\iota)$ for all $i,k,\iota$.
				\end{itemize}
				Hence, we observe for instance with the term $T_4$ that
				\begin{align*}
					T_4 =& \sum_{j=0}^{N-1}\sum_{k=l_j}^{l_{j+1}-1} (s_{k+1}-{s_k})  \mathbb{E}\Big[\Big( \int_{{t_j}}^{t_{j+1}} \Big|\sum_{i= k}^r e_{\tau} z_i({\tau})\Big|^2 \mathrm{d}{\tau} \Big)^p \big|(\nabla X_{t_j})^{-1}\big|^{2p} \Big]\\
					=& \sum_{j=0}^{N-1}\sum_{k=0}^{\gamma-1} \frac{(\mu_{k+1}-{\mu_k})}{t_{j+1}-t_j}  \mathbb{E}\Big[\Big( \int_{{t_j}}^{t_{j+1}} \Big|\sum_{i=\eta(j,k)}^r e_{\tau}  z_i({\tau})\Big|^2 \mathrm{d}{\tau} \Big)^p \big|(\nabla X_{t_j})^{-1}\big|^{2p} \Big]h_j.
				\end{align*}
				Thus  $T_4$ can be estimated as follows
				\begin{align}
					T_4\leq & \sum_{k=0}^{\gamma-1} \rho_k \mathbb{E}\Big[\sum_{j=0}^{N-1}\Big( \int_{{t_j}}^{t_{j+1}} \Big|\sum_{i=\eta(j,k)}^r e_{\tau} z_i({\tau})\Big|^2 \mathrm{d}{\tau} \Big)^p \big|(\nabla X_{t_j})^{-1}\big|^{2p} \Big]h \notag\\
					\leq&  \sum_{k=0}^{\gamma-1} \rho_k \mathbb{E}\Big[\sum_{j=0}^{N-1}\Big( \int_{{t_j}}^{t_{j+1}} \Big|\sum_{i=\eta(j,k)}^r e_{\tau}  z_i({\tau})\Big|^2 \mathrm{d}{\tau} \Big)^{2p}  \Big]^{\frac{1}{2}} \mathbb{E}\Big[ \sup_j\big|(\nabla X_{t_j})^{-1}\big|^{4p}  \Big]^{\frac{1}{2}}h.	 \label{rep4.2}
				\end{align}
				Note that $\mathbb{E}[ \sup_j|(\nabla X_{t_j})^{-1}|^{4p}  ]<C(1+\|b\|_{L^{\infty}([0,T];C_b(\mathbb{R}^d))}) $, thanks to Proposition \ref{cor 5.7}. 
				On the other hand, by first applying the reverse H\"{o}lder inequality and then using the Lemma \ref{lemmZhou}, we deduce that 
				\begin{align*}
					\mathbb{E}\Big[\sum_{j=0}^{N-1}\Big( \int_{{t_j}}^{t_{j+1}} \Big|\sum_{i=\eta(j,k)}^r e_{\tau} z_i({\tau})\Big|^2 \mathrm{d}{\tau} \Big)^{2p}  \Big] 
					\leq & C(p_1) \mathbb{E}^{\mathbb{Q}}\Big[\sum_{j=0}^{N-1}\Big( \int_{{t_j}}^{t_{j+1}} \Big|\sum_{i=\eta(j,k)}^r e_{\tau}  z_i({\tau})\Big|^2 \mathrm{d}{\tau} \Big)^{2pp_1^{'}}  \Big]^{\frac{1}{p_1^{'}}}\\
					\leq & C(p,p_1) \mathbb{E}^{\mathbb{Q}}\Big[\Big|\sum_{i=0}^r e_{t_{\tilde{\eta}(j,k)+1}}^{-1} y_i(t_{\tilde{\eta}(j,k)+1})- e_0^{-1}y_i(0)\Big|^{4pp_1^{'}}  \Big]^{\frac{1}{p_1^{'}}}.
				\end{align*}
				Using the expression of $y_i$ given by \eqref{eq4.5} in Remark \ref{rem4.2} under the probability measure $\mathbb{Q}$, we obtain that
				\begin{align*}
					& \mathbb{E}\Big[\sum_{j=0}^{N-1}\Big( \int_{{t_j}}^{t_{j+1}} \Big|\sum_{i=\eta(j,k)}^r e_{\tau} z_i({\tau})\Big|^2 \mathrm{d}{\tau} \Big)^{2p}  \Big] \\
					\leq & C(p,p_1) \mathbb{E}^{\mathbb{Q}}\Big[\Big|\sum_{i=0}^r \mathbb{E}^{\mathbb{Q}}\big[e_T y_i(T)/\mathfrak{F}_{(t_{\tilde{\eta}(j,k)+1})}\big]\Big|^{4pp_1^{'}} + \Big|\sum_{i=0}^r \mathbb{E}^{\mathbb{Q}} [e_Ty_i(T)]\Big|^{4pp_1^{'}} \Big]^{\frac{1}{p_1^{'}}}\\
					\leq & C(p,p_1) \mathbb{E}^{\mathbb{Q}}\Big[S(\phi)^{4pp_1^{'}}\Big(\sup_{ 0\leq s\leq T} \mathbb{E}^{\mathbb{Q}}\big[e_T\sup_{0\leq s_i\leq T}|\nabla_{x}X_{s_i}|/\mathfrak{F}_{s}\big]\Big)^{4pp_1^{'}}\\ &+S(\phi)^{4pp_1^{'}} \mathbb{E}^{\mathbb{Q}}\Big[|e_T\sup_{0\leq s_i\leq T}|\nabla_{x}X_{s_i}|\Big]^{4pp_1^{'}} \Big]^{\frac{1}{p_1^{'}}},
				\end{align*}
				where the last inequality follows from \eqref{eq4.5} in Remark \ref{rem4.2}, the definition of the terminal value and noting that $S(\varphi) =\sum_{i=1}^r\Lambda_{\varphi}^i$ denotes the sum of the Lipschitz constant of the functional $\varphi$.
				By applying the  Doob's maximal inequality and the reverse H\"{o}lder's inequality we obtain that
				\begin{align}
					\mathbb{E}\Big[\sum_{j=0}^{N-1}\Big( \int_{{t_j}}^{t_{j+1}} \Big|\sum_{i=\eta(j,k)}^r e_{\tau} z_i(\tau)\Big|^2 \mathrm{d}\tau \Big)^{2p}  \Big] 
					\leq & C(p,p_1,p_1') \Big( S(\phi)^{4p} \mathbb{E}^{\mathbb{Q}}\big[e_T^{4p}\sup_{0\leq s_i\leq T}|\nabla_{x}X_{s_i}|^{4p}\big] \Big)\notag\\
					\leq & C(p,p_1,p_1') C(p_0) \Big( S(\phi)^{4p} \mathbb{E}\big[e_T^{4pp_0'}\sup_{0\leq s_i\leq T}|\nabla_{x}X_{s_i}|^{4pp_0'}\big]^{\frac{1}{p_0'}} \Big)\notag\\
					\leq & C(1+\|b\|_{L^{\infty}([0,T];C_b(\mathbb{R}^d))}) S(\phi)^{4p},\label{rep4.3}
				\end{align}
				where, the constant $C$ depends only on $p,p_1,p_1',p_0'$ and the BMO norm of $Z$. Hence, by plugging \eqref{rep4.3} in \eqref{rep4.2} we obtain an bound for the term $T_4$ which is given by:
				\begin{align}\label{T_4}
					T_4 &\leq C \sum_{k=0}^{\gamma-1} \rho_k (1+\|b\|_{L^{\infty}([0,T];C_b(\mathbb{R}^d))}) S(\phi)^{2p}h \notag\\
					&\leq C (1+\|b\|_{L^{\infty}([0,T];C_b(\mathbb{R}^d))}) S(\phi)^{2p}h.
				\end{align}
				Similarly, a bound for the term $T_3$ can be derived as follows
				\begin{align*}
					T_3 
					\leq  \sum_{k=0}^{\gamma-1} \rho_k \mathbb{E}\Big[\sum_{j=0}^{N-1}\Big( \int_{{t_j}}^{t_{j+1}} \Big|\sum_{i=\eta(j,k)}^r e_{\tau}  z_i(\tau)\nabla_zg(\tau,\bold{X}_{\tau})\Big|^2 \mathrm{d}\tau \Big)^{2p}  \Big]^{\frac{1}{2}} \mathbb{E}\Big[ \sup_j\big|(\nabla X_{t_j})^{-1}\big|^{4p}  \Big]^{\frac{1}{2}}h	. 
				\end{align*}
				Let us notice that
				\begin{align*}
					&\mathbb{E}\Big[\sum_{j=0}^{N-1}\Big( \int_{{t_j}}^{t_{j+1}} \Big|\sum_{i=\eta(j,k)}^r e_{\tau}  z_i(\tau)\nabla_zg(\tau,\bold{X}_{\tau})\Big|^2 \mathrm{d}\tau \Big)^{2p} \Big]\\
					\leq & \mathbb{E}\Big[\sum_{j=0}^{N-1}\sup_{t_j\leq \tau \leq t_{j+1}}|\nabla_zg(\tau,\bold{X}_{\tau})|^{4p}\Big( \int_{{t_j}}^{t_{j+1}} \Big|\sum_{i=\eta(j,k)}^r e_{\tau}  z_i(\tau)\Big|^2 \mathrm{d}\tau \Big)^{2p} \Big]\\
					\leq & C_p\mathbb{E}\Big[1+\sup_{0\leq \tau \leq T}(\ell(|Y_{\tau}|)|{Z}_{\tau}|)^{8p}\Big]^{\frac{1}{2}}
					\times \mathbb{E}\Big[ \sum_{j=0}^{N-1}\Big(\int_{{t_j}}^{t_{j+1}} \Big|\sum_{i=\eta(j,k)}^r e_{\tau}  z_i(\tau)\Big|^2 \mathrm{d}\tau \Big)^{4p}  \Big]^{\frac{1}{2}}\\
					\leq  & C_p(1+\|b\|_{L^{\infty}([0,T];C_b(\mathbb{R}^d))}) S(\phi)^{4p},
				\end{align*}
				where we used the growth $|\nabla_zg(\tau,\bold{X}_{\tau})| \leq  \Lambda_z(1+\ell(|Y_{\tau}|)|Z_{\tau}|)$ and the inequality \eqref{rep4.3}, to derive the last two inequalities. Thus, there is a constant $C$ (independent of $h$)such that the term $T_3$ is controlled by 
				\begin{equation}\label{T_3}
					T_3 \leq C h.
				\end{equation}
				Let us turn now to establish the estimate of the terms $T_5, T_6$ and $T_7$. Analogously to \eqref{rep4.2}, we observe that 
				\begin{align}\label{T_5}
					T_5 &\leq  \sum_{k=0}^{\gamma-1} \rho_k \mathbb{E}\Big[\sum_{j=0}^{N-1}\Big( \int_{{t_j}}^{t_{j+1}} \big| e_{\tau} z_0(\tau)\big|^2 \mathrm{d}\tau \Big)^{2p}  \Big]^{\frac{1}{2}} \mathbb{E}\Big[ \sup_j|(\nabla X_{t_j})^{-1}|^{4p}  \Big]^{\frac{1}{2}}h\leq Ch.
				\end{align}
				Indeed, using the reverse H\"{o}lder inequality once again and applying Lemma \ref{lemmZhou} lead to
				\begin{align}
					\mathbb{E}\Big[\sum_{j=0}^{N-1}\Big( \int_{{t_j}}^{t_{j+1}} \big| e_{\tau} z_0(\tau)\big|^2 \mathrm{d}\tau \Big)^{2p}  \Big]
					\leq & C(p_1)\mathbb{E}^{\mathbb{Q}}\Big[\sum_{j=0}^{N-1}\Big( \int_{{t_j}}^{t_{j+1}} \big| e_{\tau} z_0(\tau)\big|^2 \mathrm{d}\tau \Big)^{2pp_1'}  \Big]^{\frac{1}{p_1'}}\notag\\
					\leq &  C(p_1) \mathbb{E}^{\mathbb{Q}}\Big[ \Big| e_{T}y_{0}(T) -y_0(0)  \Big|^{4pp_1'}  \Big]^{\frac{1}{p_1'}}\notag\\
					\leq & C(p_1)\mathbb{E}^{\mathbb{Q}}\Big[ \mathbb{E}^{\mathbb{Q}}\Big(\int_0^T |e_{\tau}\nabla_x g(\tau,\bold{X}_{\tau})\nabla X_{\tau}|\mathrm{d}\tau\Big)^{4p_1p_1'}   \Big]^{\frac{1}{p_1'}}\notag.
				\end{align}
				By using the growth of $\nabla_x g$ and the integrability of $\nabla X$, we deduce that
				\begin{align}
					\mathbb{E}\Big[\sum_{j=0}^{N-1}\Big( \int_{{t_j}}^{t_{j+1}} \big| e_{\tau} z_0(\tau)\big|^2 \mathrm{d}\tau \Big)^{2p}  \Big]\leq C(p_1,p_0,T,\|b\|_{L^{\infty}([0,T];C_b(\mathbb{R}^d))},\|Z*W\|_{\bmo}).
				\end{align}
				The bounds of $T_6$ and $T_7$ follow similarly i.e. there exist a constant $C$ independent of $h$ such that 
				\begin{align}\label{T_6}
					T_6+T_7 \leq C h.
				\end{align}
				Indeed, using similar estimates as above $T_7$ for instance can be controlled by
				\begin{align*}
					T_7 &\leq  \sum_{k=0}^{\gamma-1}\rho_k\mathbb{E}\Big[ \big|\sup_{0\leq \tau \leq T}\nabla_{{z}}g(\tau,\bold{X}_{\tau})\big|^{8p}+ \Big(\sum_{j=0}^{N-1}\int_{{t_j}}^{t_{t_{j+1}}} \Big|e_{\tau}  z_0(\tau)\Big|^2\mathrm{d}\tau \Big)^{4p} \Big]^\frac{1}{2}\mathbb{E}\left[ \big|(\nabla X_{t_j})^{-1}\big|^{4p}\right]^\frac{1}{2}h\\
					&\leq  C(p) \sum_{k=0}^{\gamma-1}\rho_k\mathbb{E}\Big[ \big|\sup_{0\leq \tau \leq T}\nabla_{{z}}g(\tau,\bold{X}_{\tau})\big|^{8p}+ \Big(\sum_{j=0}^{N-1}\int_{{t_j}}^{t_{t_{j+1}}}|e_{\tau} z_0(\tau)|^2\mathrm{d}\tau \Big)^{4p} \Big]^\frac{1}{2}h \\
					&\leq C h.
				\end{align*}
				It remains to estimate $T_1$ and $T_2$ in order to conclude the proof. Let us start with $T_2$. Using the same technique as before again, we deduce that 
				\begin{align*}
					T_2 &\leq  \sum_{k=0}^{\gamma-1} \rho_k \mathbb{E}\Big[\sum_{j=0}^{N-1} \Big|\sum_{i=l_j}^{\eta(j,k)-1} e_{t_j}y_i(t_j)\Big|^{4p}  \Big]^{\frac{1}{2}} \mathbb{E}\Big[ \sup_j\big|(\nabla X_{t_j})^{-1}\big|^{4p}  \Big]^{\frac{1}{2}}h\\
					&\leq  \sum_{k=0}^{\gamma-1} \rho_k \mathbb{E}\Big[\sum_{j=0}^{N-1}\Big|\mathbb{E}^{\mathbb{Q}}\Big[\sum_{i=l_j}^{\eta(j,k)-1} e_{T}y_i(T)\big/\mathfrak{F}_{t_j}\Big]\Big|^{4p}  \Big]^{\frac{1}{2}} \mathbb{E}\Big[ \sup_j\big|(\nabla X_{t_j})^{-1}\big|^{4p}  \Big]^{\frac{1}{2}}h.
				\end{align*}
				Recall that $\mathbb{E}\Big[ \sup_j\big|(\nabla X_{t_j})^{-1}\big|^{4p}  \Big] \leq C$, where the constant $C$ does not depend on $h$ (see \eqref{5.12}). Therefore by the definition of the terminal value
				\begin{align*}
					T_2 &\leq C(p_1,p)\sum_{k=0}^{\gamma-1} \rho_k \mathbb{E}^{\mathbb{Q}}\Big[\sum_{j=0}^{N-1}\Big|\mathbb{E}^{\mathbb{Q}}\Big[\sum_{i=l_j}^{\eta(j,k)-1} e_Ty_i(T)\big/\mathfrak{F}_{t_j}\Big]\Big|^{4pp_1'}  \Big]^{\frac{1}{2p_1'}}h\\
					&\leq C(p_1,p)\sum_{k=0}^{\gamma-1} \rho_k \mathbb{E}^{\mathbb{Q}}\Big[\Big(\sum_{j=0}^{N-1}\sum_{i=l_j}^{\eta(j,k)-1} S(\phi)\sup_{ 0\leq s\leq T} \mathbb{E}^{\mathbb{Q}}\big[A_T\sup_{0\leq s_i\leq T}|\nabla_{x}X_{s_i}|/\mathfrak{F}_{s}\big]\Big)^{4pp_1'}  \Big]^{\frac{1}{2p_1'}}h.
				\end{align*}
				Using the Doob's maximal inequality and the integrability of $\nabla_x X$ we deduce that 
				\begin{align}\label{T_2}
					T_2 &\leq C(p_1,p)(1+ S(\phi)^{2p} )h.
				\end{align}
				Finally, for  $\tau \in (t_j,t_{j+1})$ the term $T_1$ can be rewritten as follow
				\begin{align*}
					T_1 &=\sum_{j=0}^{N-1}\sum_{k=l_j}^{l_{j+1}-1}\mathbb{E} \int_{s_k}^{s_{k+1}} \mathbb{E} \Big|h_j\nabla_{{y}} g(\tau,\bold{X}_{\tau})\sum_{i=k}^{r} e_{\tau} y_i(t)\Big|^{2p} \big|(\nabla X_{t_j})^{-1}\big|^{2p} \mathrm{d}t\\
					&\leq C\sum_{j=0}^{N-1}\sum_{k=l_j}^{l_{j+1}-1}\mathbb{E} \int_{s_k}^{s_{k+1}} \mathbb{E} \Big|h_je_{\tau}\nabla_{{y}} g(\tau,\bold{X}_{\tau})\Big|^{4p} \big|(\nabla X_{t_j})^{-1}\big|^{2p} \mathrm{d}t\\&\quad + C\sum_{j=0}^{N-1}\sum_{k=l_j}^{l_{j+1}-1}\mathbb{E} \int_{s_k}^{s_{k+1}}\mathbb{E}\Big|\sum_{i=k}^{r}  y_i(t_j)\Big|^{4p}\big|(\nabla X_{t_j})^{-1}\big|^{2p} \mathrm{d}t,
				\end{align*}
				where we used the elementary inequality $2ab\leq a^2+b^2$ to obtain the inequality above. Therefore, by using once more the same technique we deduce that
				\begin{align*}
					T_1	&\leq C_p  \sum_{k=0}^{\gamma-1} \rho_k \mathbb{E}\Big[ \Big(\sum_{j=0}^{N-1}h_j\Big)^{8p} \big|e_{\tau} \nabla_{{y}} g(\tau,\bold{X}_{\tau})\big|^{8p} \Big]^{\frac{1}{2}}h\\&\quad + \sum_{j=0}^{N-1}\sum_{k=l_j}^{l_{j+1}-1}\mathbb{E} \int_{s_k}^{s_{k+1}}\mathbb{E}\Big|\sum_{i=k}^{r}  y_i(t_j)\Big|^{8p}\Big]^{\frac{1}{2}}\mathbb{E}\Big[\sup_j\big|(\nabla X_{t_j})^{-1}\big|^{4p}\Big]^{\frac{1}{2}} \mathrm{d}t\\
					&\leq C_p  \sum_{k=0}^{\gamma-1} \rho_k T^{4p} \mathbb{E}\Big[\big|e_{\tau} \nabla_{{y}} g(\tau,\bold{X}_{\tau})\big|^{8p} \Big]^{\frac{1}{2}}h 
					+ C(p_1) \sum_{k=0}^{\gamma-1} \rho_kS(\varphi)^{4p}\mathbb{E}^{\mathbb{Q}}\Big[\sup_i\big|\nabla X_{s_i}\big|^{8p}\Big]^{\frac{1}{2}} h.
				\end{align*}
				From the growth of $\nabla_y g$, the boundedness of $Y$, the integrability of $e_{\tau}$ and the control from $T_2$ we have that
				\begin{align}\label{T_1}
					T_1 &\leq C_p T^{4p}  \mathbb{E}\Big[(1+ \ell(|Y_{\tau}|)\sup_{ 0\leq \tau\leq T}|Z_{\tau}|^2)^{16p} \Big]^{\frac{1}{4}}h +C(p_1,p_0)(1+\|b\|_{L^{\infty}([0,T];C_b(\mathbb{R}^d))}) S(\phi)^{4p}h\notag\\
					&\leq C(p,T)(1+ S(\phi)^{8p} )h + C(p_1,p_0)(1+\|b\|_{L^{\infty}([0,T];C_b(\mathbb{R}^d))}) S(\phi)^{4p}h.
				\end{align}
				In summary, combining \eqref{T_0}, \eqref{T_4}, \eqref{T_3},  \eqref{T_5}, \eqref{T_6}, \eqref{T_2} and \eqref{T_1}, there is a constant $C$ not depending on $h$ and the Lipschitz constants of the coefficients such that
				\begin{align*}
					\sum_{j=0}^{N-1}\mathbb{E}\Big(\int_{t_j}^{t_{j+1}}|Z_t - Z_{t_j}|^2 \mathrm{d}t \Big)^p  &\leq C_p h^{p-1} \Big[ \sum_{j=0}^{N-1}\int_{t_j}^{t_{j+1}} h_j^{2p}\mathrm{d}t + \sum_{\nu =1}^{7}T_{\nu}  \Big]\\ &\leq C_p(1+ \|b\|_{L^{\infty}([0,T];C_b(\mathbb{R}^d))})h^p.
				\end{align*}
				By using standard smooth approximation techniques in combination with result on stability for quadratic BSDEs (see for instance \cite[Corollary 3.9]{ImkRhOliv24}), the above estimate remains valid for non continuously differentiable coefficients. This concludes the proof.
			\end{proof}
			Let us briefly provide below some arguments for the proof of Theorem \ref{thm1.1}.
			\begin{proof}[Proof of Theorem \ref{thm1.1}]
				In this context, clearly the path regularity of the process $Y$ follows similarly as it was done in the case of the discrete-type functional terminal value, since the control $Z$ satisfies the bound \eqref{boundZ0}. Thus
				\begin{align*}
					\mathbb{E}\Big[ \sup_{s\leq r\leq t }|Y_r-Y_s|^{p}  \Big]\leq  C(p)(1+\Lambda_{\Phi}^{p})\Big\{ |t-s|^{p} + |t-s|^{p/2} \Big\}.
				\end{align*}
				We focus on the control process part. Let us consider the following sequence of partitions $(\pi(r))_{r\geq 1}:0=s_0<s_1 \cdots < s_{r}=T$ of the interval $[0,T]$ such that $\lim_{r\rightarrow \infty} |\pi(r)|= 0$. We define accordingly the smooth bounded function $\hat\varphi^{\pi(r)}(X_{s_0},\cdots,X_{s_r})$ approximating the functional $\Phi$ (see Remark \ref{remark24}). In addition, we impose that the Lipschitz constant of $\hat \varphi^{\pi(r)}$ satisfies $S(\hat \varphi^{\pi(r)}) < C \Lambda_{\Phi}$. If we denote by $(Y^{\pi(r)},Z^{\pi(r)})$ the solution corresponding to the terminal value $\hat \varphi^{\pi(r)}(X_{s_0},\cdots,X_{s_r})$, then invoking once more a stability result for quadratic BSDEs (see \cite[Corollary 3.9]{ImkRhOliv24}), it is readily seen that
					$\lim_{|\pi(r)|\rightarrow 0}\mathbb{E}\int_0^T |Z^{\pi(r)}_t-Z_t|^2\mathrm{d}t =0.$
				Then, there is a subsequence still indexed by ${\pi(r)}$ such that $\lim_{|\pi(r)|\rightarrow 0}\mathbb{E}|Z^{\pi(r)}_t-Z_t|^2 =0$ $\mathrm{d}t\text{-a.e.}$ Moreover,
				one can extract another subsequence that we will still denote by $Z^{\pi(r)}$ such that $\mathbb{P}\text{-a.s.}$ $\text{Leb}\{t\in [0,T), Z_t^{\pi(r)} = Z_t \}$ goes to $T$ as $|\pi(r)|$ tends to $0$.
				In addition, for all $p>p_0'$ there is a constant $C>0$ not depending on the particular choice of the partition such that $\mathbb{E}|Z_t^{\pi(r)}|^{2p} \leq C$, for all $t\in [0,T]$. This leads to the fact   $G^{\pi(r)}(t):= \mathbb{E}|Z_t^{\pi(r)}-Z_t|^{2p}$ goes to $0$ as $|\pi(r)|\rightarrow 0$, $\mathrm{d}t\text{ a.e.}$ by the dominated convergence theorem. Thus $G^{\pi(r)}(t)$ is bounded $\mathrm{d}t\text{ a.e.}$ and uniformly in $\pi(r)$. Let us now establish the path regularity result for the limit process $Z_t$. We observe for all $t\in [t_j,t_{j+1}]$ that 
				\begin{align}
					&\sum_{j=0}^{N-1}\mathbb{E}\Big(\int_{t_j}^{t_{j+1}}|Z_t - Z_{t_j}|^2 \mathrm{d}t \Big)^p \notag\\
					&\leq C(p)\sum_{j=0}^{N-1}\int_{t_j}^{t_{j+1}}\mathbb{E}(|Z_t^{\pi(r)} -Z_t|^{2p}+ |Z_t^{\pi(r)} -Z_{t_j}^{\pi(r)}|^{2p} + |Z_{t_j}^{\pi(r)} -Z_{t_j}|^{2p})\mathrm{d}t
				\end{align}
				By using once more the dominated convergence theorem, we observe that the first and the last terms on the right hand side from the above inequality tends to 0 as $\pi(r)\rightarrow 0$. Therefore, there is a constant $C$ not depending on $\pi(r)$ 
				\begin{align}
					\sum_{j=0}^{N-1}\mathbb{E}\Big(\int_{t_j}^{t_{j+1}}|Z_t - Z_{t_j}|^2 \mathrm{d}t \Big)^p
					\leq C(p)\sum_{j=0}^{N-1}\int_{t_j}^{t_{j+1}}\mathbb{E} |Z_t^{\pi(r)}-Z_{t_j}^{\pi(r)}|^{2p}\mathrm{d}t \leq C h^{p}.
				\end{align} This finishes the proof.
			\end{proof}
			\begin{proof}[Proof of Proposition \ref{cor 1.3}] We mimic the proof of \cite[Proposition 3.2.12]{Zhou}. We will first assume that the terminal value is given as in the proof of Theorem \ref{thm1.2}. Furthermore, we assume that the coefficients involved are continuously differentiable with bounded derivatives. We underline the fact that, thanks to the assumptions in force, the trade between the underlying probability $\mathbb{P}$ and its equivalent measure $\mathbb{Q}^N$ is possible at a cost of the constant $\epsilon \in (0,1)$.
				
				We first observe that
				\begin{align*}
					|Z_s-\hat{Z}_{t_j}|^2 &\leq 2|Z_s-{Z}_{t_j}|^2 +2 \left|Z_{t_j}-\frac{1}{\delta t_j} \mathbb{E}\Big[\int_{t_j}^{t_{j+1}} Z_s \mathrm{d}s\big/\mathfrak{F}_{t_j} \Big]\right|^2\\
					&\leq 2|Z_s-{Z}_{t_j}|^2 + \frac{2}{\epsilon\delta t_j}\mathbb{E}^{\mathbb{Q}^{N}}\Big[\int_{t_j}^{t_{j+1}} |Z_{t_j}-Z_s|^2 \mathrm{d}s\big/\mathfrak{F}_{t_j} \Big]	
				\end{align*}
				Integrating from $t_j$ to $t_{j+1}$ and taking the $\eta^{th}\text{power}$, for any $\eta >1$ we have
				\begin{align*}
					\Big(\int_{t_j}^{t_{j+1}}|Z_s-\hat{Z}_{t_j}|^2\mathrm{d}s\Big)^{\eta} \leq \Big(\int_{t_j}^{t_{j+1}}|Z_s-{Z}_{t_j}|^2\mathrm{d}s\Big)^{\eta} + \Big(\frac{2}{\epsilon}\mathbb{E}^{\mathbb{Q}^{N}}\Big[\int_{t_j}^{t_{j+1}} |Z_{t_j}-Z_s|^2 \mathrm{d}s\big/\mathfrak{F}_{t_j} \Big]\Big)^{\eta}.
				\end{align*}
				For any $p\geq 2$, and using Jensen's inequality, there is a constant $C(\epsilon,p,\eta)>0$ such that
				\begin{align*}
					\mathbb{E}\Big[\sup_{0\leq i\leq N-1}\mathbb{E}^{\mathbb{Q}^{N}}\Big\{\sum_{j=i}^{N-1}\Big(\int_{t_j}^{t_{j+1}}|Z_s-\hat{Z}_{t_j}|^2\mathrm{d}s\Big)^{\eta}\big/\mathfrak{F}_{t_i}\Big\}^p\Big] \leq C(\epsilon,p,\eta) h^{p(\eta -1)} I,
				\end{align*}
				where 
				\begin{align*}
					I = \mathbb{E}\Big[\sup_{0\leq i\leq N-1}\mathbb{E}^{\mathbb{Q}^{N}}\Big[\sum_{j=i}^{N-1}\Big(\int_{t_j}^{t_{j+1}}|Z_s-{Z}_{t_j}|^{2\eta}\mathrm{d}s\Big)\big/\mathfrak{F}_{t_i}\Big]^p\Big].
				\end{align*}
				Let us assume without loss of generality that $\pi_1$ is finer than $\pi_0$ and that $t_i= s_{l_i}$ for some increasing sequence $(l_i)_{i=1}^N$. For any $t \in [s_k,s_{k+1}) \subset [s_{l_j},s_{l_{j+1}}) = [t_j, t_{j+1})$, we write
				\begin{align*}
					Z_s - Z_{t_j} &= \left(\nabla^{k}Y_s\right)^{\bold{T}}\left( \nabla_xX_s\right)^{-1} - \left(\nabla^{l_j}Y_{t_j}\right)^{\bold{T}}\left( \nabla_xX_{t_j}\right)^{-1}\\
					&= I_{s}^{1,j} + I_{s}^{2,j},
				\end{align*}
				where $I_{s}^{1,j}:= \left(\nabla^{l_j}Y_s\right)^{\bold{T}}\left( (\nabla_xX_s)^{-1} - (\nabla_xX_{t_j})^{-1} \right)$ and $I_{s}^{2,j}:=  \left(\nabla^{k}Y_s-\nabla^{l_j}Y_{t_j}\right)^{\bold{T}}\left( \nabla_xX_{t_j}\right)^{-1}.$ Therefore, from Girsanov's theorem and by applying the reverse H\"older inequality we deduce that
				\begin{align*}
					I &\leq C_{pq^{*}} \mathbb{E}\Big[ \Big(\sum_{j=0}^{N-1}\int_{t_j}^{t_{j+1}}|I_s^{1,j}|^{2\eta}\mathrm{d}s\Big)^{pq^{*}} \Big]^{\frac{1}{q^{*}}} +  \mathbb{E}\Big[\sup_{0\leq i\leq N-1}\mathbb{E}^{\mathbb{Q}^{N}}\Big[\sum_{j=i}^{N-1}\Big(\int_{t_j}^{t_{j+1}}|I_s^{2,j}|^{2\eta}\mathrm{d}s\Big)\big/\mathfrak{F}_{t_i}\Big]^p\Big]\\
					&\leq C_{pq^{*}} \Big(\int_{0}^{T}\sum_{j=0}^{N-1}\mathbb{E}(|I_s^{1,j}|^{2pq^{*}\eta}){\bf{1}_{[t_j,t_{j+1})}}\mathrm{d}s\Big)  +  \mathbb{E}\Big[\sup_{0\leq i\leq N-1}\mathbb{E}^{\mathbb{Q}^{N}}\Big[\sum_{j=i}^{N-1}\Big(\int_{t_j}^{t_{j+1}}|I_s^{2,j}|^{2\eta}\mathrm{d}s\Big)\big/\mathfrak{F}_{t_i}\Big]^p\Big]\\
					&\leq C_{pq^{*}} h^{pq^{*}\eta} +  \mathbb{E}\Big[\sup_{0\leq i\leq N-1}\mathbb{E}^{\mathbb{Q}^{N}}\Big[\sum_{j=i}^{N-1}\Big(\int_{t_j}^{t_{j+1}}|I_s^{2,j}|^{2\eta}\mathrm{d}s\Big)\big/\mathfrak{F}_{t_i}\Big]^p\Big],
				\end{align*}
				where one can use similar computation as in \eqref{T_0} to obtain the first bound of the inequality above and the constant $C_{pq^{*}}$ does not depend on the Lipschitz constant of $b$.
				
				On the other hand, we will follow the same strategy we used to estimate for instance \eqref{eqsup1}, by considering the same finer partition $\pi_2$. We first observe that
				\begin{align*}
					&\mathbb{E}^{\mathbb{Q}^{N}}\Big[\sum_{j=i}^{N-1}\Big(\int_{t_j}^{t_{j+1}}|I_s^{2,j}|^{2\eta}\mathrm{d}s\Big)\big/\mathfrak{F}_{t_i}\Big]\\
					&\leq (2-\epsilon)\mathbb{E}^{\mathbb{Q}^{N}}\Big[\sum_{j=i}^{N-1}\mathbb{E}\Big[\Big(\int_{t_j}^{t_{j+1}}|I_s^{2,j}|^{2\eta}\mathrm{d}s\Big)\big/\mathfrak{F}_{t_j}\Big]\big/\mathfrak{F}_{t_i}\Big]\\
					&\leq (2-\epsilon)\mathbb{E}^{\mathbb{Q}^{N}} \Big[\sum_{j=i}^{N-1} \mathbb{E}\int_{t_j}^{t_{j+1}}\big[\big|\nabla^{k}Y_s - \nabla^{l_j}Y_{t_j}\big|^{2\eta}/\mathfrak{F}_{t_j}\big]|(\nabla X_{t_j})^{-1}|^{2\eta}\mathrm{d}s \big/\mathfrak{F}_{t_i} \Big].
				\end{align*}  
				By using the expression given by \eqref{eqsup1} and applying once more the Girsanov theorem, we deduce that
				\begin{align*}
					&\mathbb{E}^{\mathbb{Q}^{N}}\Big[\sum_{j=i}^{N-1}\Big(\int_{t_j}^{t_{j+1}}|I_s^{2,j}|^{2\eta}\mathrm{d}s\Big)\big/\mathfrak{F}_{t_i}\Big]\\
					\leq & C_{\eta} (2-\epsilon)\mathbb{E}^{\mathbb{Q}^{N}}\Bigg[\sum_{k=0}^{\gamma-1}\rho_k\sum_{j=i}^{N-1}\frac{1}{\epsilon} \mathbb{E}^{\mathbb{Q}^{N}}\Big[ \Big|(t_{j+1}-t_j)\nabla_yg(\tau,\bold{X}_{\tau})\sum_{i=k}^{r} e_{\tau}  y_i(t_j) \Big|^{2\eta}\\
					& +  \Big|\sum_{i\geq l_j}^{k-1} e_{t_j}y_i(t_j)\Big|^{2\eta} + \Big(\int_{{t_j}}^{t}\sum_{i\geq k}^r e_{\tau} \nabla_{{z}}g({\tau},\bold{X}_{\tau})z_i({\tau})\mathrm{d}{\tau} \Big)^{2\eta} + \Big( \int_{{t_j}}^{t_{j+1}} \Big|\sum_{i\geq k}^r e_{\tau} z_i({\tau})\Big|^2 \mathrm{d}{\tau} \Big)^{\eta} \notag\\
					& + \Big( \int_{{t_j}}^{t_{j+1}} \big| e_{\tau}  z_0({\tau})\big|^2 \mathrm{d}{\tau} \Big)^{\eta} +\Big( \int_{{t_j}}^{t_{j+1}} |e_{\tau}  \nabla_{{x}} g({\tau},\bold{X}_{\tau})\nabla_x {X}_{\tau}|  \mathrm{d}{\tau}\Big)^{2\eta}\notag\\
					&\quad + \Big( \int_{{t_j}}^{t_{j+1}} |e_{\tau} \nabla_{{z}} g({\tau},\bold{X}_{\tau})z_0({\tau})| \mathrm{d}{\tau}\Big)^{2\eta}
					\Big/\mathfrak{F}_{t_j}|(\nabla X_{t_j})^{-1}|^{2\eta} \big/\mathfrak{F}_{t_i} \Big]\Bigg]h.
				\end{align*}
				By using the H\"older's inequality, we deduce the following for any $p\geq q^{*}$
				\begin{align*}
					&\mathbb{E}\Big[\sup_{0\leq i\leq N-1}\mathbb{E}^{\mathbb{Q}^{N}}\Big[\sum_{j=i}^{N-1}\Big(\int_{t_j}^{t_{j+1}}|I_s^{2,j}|^{2\eta}\mathrm{d}s\Big)\big/\mathfrak{F}_{t_i}\Big]^p\Big]\\
					\leq & C({\eta},p,\epsilon)\mathbb{E}\Bigg[\sup_{0\leq i\leq N-1} \mathbb{E}^{\mathbb{Q}^{N}}\Bigg[\sum_{k=0}^{\gamma-1}\rho_k\sum_{j=i}^{N-1}\frac{1}{\epsilon} \mathbb{E}^{\mathbb{Q}^{N}}\Bigg( \Big|(t_{j+1}-t_j)\nabla_yg(\tau,\bold{X}_{\tau})\sum_{i=k}^{r} e_{\tau}  y_i(t) \Big|^{2\eta} \\
					&+  \Big|\sum_{i\geq l_j}^{k-1} e_{t_j}y_i(t_j)\Big|^{2\eta} + \Big(\int_{{t_j}}^{t}\sum_{i\geq k}^r e_{\tau} \nabla_{{z}}g({\tau},\bold{X}_{\tau})z_i({\tau})\mathrm{d}{\tau} \Big)^{2\eta} + \Big( \int_{{t_j}}^{t_{j+1}} \Big|\sum_{i\geq k}^r e_{\tau} z_i({\tau})\Big|^2 \mathrm{d}{\tau} \Big)^{\eta} \notag\\
					& + \Big( \int_{{t_j}}^{t_{j+1}} \big| e_{\tau}  z_0({\tau})\big|^2 \mathrm{d}{\tau} \Big)^{\eta} +\Big( \int_{{t_j}}^{t_{j+1}} |e_{\tau}  \nabla_{{x}} g({\tau},\bold{X}_{\tau})\nabla_x {X}_{\tau}|  \mathrm{d}{\tau}\Big)^{2\eta}\notag\\
					&\quad + \Big( \int_{{t_j}}^{t_{j+1}} |e_{\tau} \nabla_{{z}} g({\tau},\bold{X}_{\tau})z_0({\tau})| \mathrm{d}{\tau}\Big)^{2\eta}
					\Big/\mathfrak{F}_{t_j}\Bigg)|(\nabla X_{t_j})^{-1}|^{2\eta} \Big/\mathfrak{F}_{t_i} \Bigg]^p\Bigg]h^p.
				\end{align*}
				Moreover, by using the tower property and the reverse H\"older inequality we deduce that 
				\begin{align*}
					&\mathbb{E}\Big[\sup_{0\leq i\leq N-1}\mathbb{E}^{\mathbb{Q}^{N}}\Big[\sum_{j=i}^{N-1}\Big(\int_{t_j}^{t_{j+1}}|I_s^{2,j}|^{2\eta}\mathrm{d}s\Big)\big/\mathfrak{F}_{t_i}\Big]^p\Big]\\
					\leq & C({\eta},p,\epsilon)\mathbb{E}\Bigg[\sup_{0\leq i\leq N-1} \mathbb{E}\Bigg[\Bigg(\sum_{k=0}^{\gamma-1}\rho_k\sum_{j=i}^{N-1}\Bigg\{ \Big|(t_{j+1}-t_j)\nabla_yg(\tau,\bold{X}_{\tau})\sum_{i=k}^{r} e_{\tau}  y_i(t) \Big|^{2\eta} \\
					&+  \Big|\sum_{i\geq l_j}^{k-1} e_{t_j}y_i(t_j)\Big|^{2\eta} + \Big(\int_{{t_j}}^{t}\sum_{i\geq k}^r e_{\tau} \nabla_{{z}}g({\tau},\bold{X}_{\tau})z_i({\tau})\mathrm{d}{\tau} \Big)^{2\eta} + \Big( \int_{{t_j}}^{t_{j+1}} \Big|\sum_{i\geq k}^r e_{\tau} z_i({\tau})\Big|^2 \mathrm{d}{\tau} \Big)^{\eta} \notag\\
					& + \Big( \int_{{t_j}}^{t_{j+1}} \big| e_{\tau}  z_0({\tau})\big|^2 \mathrm{d}{\tau} \Big)^{\eta} +\Big( \int_{{t_j}}^{t_{j+1}} |e_{\tau}  \nabla_{{x}} g({\tau},\bold{X}_{\tau})\nabla_x {X}_{\tau}|  \mathrm{d}{\tau}\Big)^{2\eta}\notag\\
					&\quad + \Big( \int_{{t_j}}^{t_{j+1}} |e_{\tau} \nabla_{{z}} g({\tau},\bold{X}_{\tau})z_0({\tau})| \mathrm{d}{\tau}\Big)^{2\eta}\Bigg\}
					\Bigg)^{q^{*}}\sup_{0\leq j\leq N-1}|(\nabla X_{t_j})^{-1}|^{2\eta q^{*}} \Big/\mathfrak{F}_{t_i} \Bigg]^{\frac{p}{q^{*}}}\Bigg]h^p.
				\end{align*}
				This leads to
				\begin{align*}
					&\mathbb{E}\Big[\sup_{0\leq i\leq N-1}\mathbb{E}^{\mathbb{Q}^{N}}\Big[\sum_{j=i}^{N-1}\Big(\int_{t_j}^{t_{j+1}}|I_s^{2,j}|^{2\eta}\mathrm{d}s\Big)\big/\mathfrak{F}_{t_i}\Big]^p\Big]\\
					\leq & C({\eta},p,\epsilon,q^{*})\mathbb{E}\Bigg[\Bigg(\sum_{k=0}^{\gamma-1}\rho_k\sum_{j=i}^{N-1}\Bigg\{ \Big|(t_{j+1}-t_j)\nabla_yg(\tau,\bold{X}_{\tau})\sum_{i=k}^{r} e_{\tau}  y_i(t) \Big|^{2\eta} \\
					&+  \Big|\sum_{i\geq l_j}^{k-1} e_{t_j}y_i(t_j)\Big|^{2\eta} + \Big(\int_{{t_j}}^{t}\sum_{i\geq k}^r e_{\tau} \nabla_{{z}}g({\tau},\bold{X}_{\tau})z_i({\tau})\mathrm{d}{\tau} \Big)^{2\eta} + \Big( \int_{{t_j}}^{t_{j+1}} \Big|\sum_{i\geq k}^r e_{\tau} z_i({\tau})\Big|^2 \mathrm{d}{\tau} \Big)^{\eta} \notag\\
					& + \Big( \int_{{t_j}}^{t_{j+1}} \big| e_{\tau}  z_0({\tau})\big|^2 \mathrm{d}{\tau} \Big)^{\eta} +\Big( \int_{{t_j}}^{t_{j+1}} |e_{\tau}  \nabla_{{x}} g({\tau},\bold{X}_{\tau})\nabla_x {X}_{\tau}|  \mathrm{d}{\tau}\Big)^{2\eta}\notag\\
					&\quad + \Big( \int_{{t_j}}^{t_{j+1}} |e_{\tau} \nabla_{{z}} g({\tau},\bold{X}_{\tau})z_0({\tau})| \mathrm{d}{\tau}\Big)^{2\eta}\Bigg\}
					\Bigg)^{p}\sup_{0\leq j\leq N-1}|(\nabla X_{t_j})^{-1}|^{2\eta p} \Bigg]h^p.
				\end{align*}
				Hence, using once more the H\"older inequality we obtain the existence of the constant $C({\eta},p,\epsilon,q^{*})$ not depending on the Lipschitz constants of the coefficients such that the following holds
				\begin{align*}
					&\mathbb{E}\Big[\sup_{0\leq i\leq N-1}\mathbb{E}^{\mathbb{Q}^{N}}\Big[\sum_{j=i}^{N-1}\Big(\int_{t_j}^{t_{j+1}}|I_s^{2,j}|^{2\eta}\mathrm{d}s\Big)\big/\mathfrak{F}_{t_i}\Big]^p\Big]\\
					\leq& C({\eta},p,\epsilon,q^{*})h^p \mathbb{E}\Bigg[\sum_{k=0}^{\gamma-1}\rho_k\Bigg(\sum_{j=i}^{N-1}\Bigg\{ \Big|(t_{j+1}-t_j)\nabla_yg(\tau,\bold{X}_{\tau})\sum_{i=k}^{r} e_{\tau}  y_i(t) \Big|^{2\eta} \\
					&+  \Big|\sum_{i\geq l_j}^{k-1} e_{t_j}y_i(t_j)\Big|^{2\eta} + \Big(\int_{{t_j}}^{t}\sum_{i\geq k}^r e_{\tau} \nabla_{{z}}g({\tau},\bold{X}_{\tau})z_i({\tau})\mathrm{d}{\tau} \Big)^{2\eta} + \Big( \int_{{t_j}}^{t_{j+1}} \Big|\sum_{i\geq k}^r e_{\tau} z_i({\tau})\Big|^2 \mathrm{d}{\tau} \Big)^{\eta} \notag\\
					& + \Big( \int_{{t_j}}^{t_{j+1}} \big| e_{\tau}  z_0({\tau})\big|^2 \mathrm{d}{\tau} \Big)^{\eta} +\Big( \int_{{t_j}}^{t_{j+1}} |e_{\tau}  \nabla_{{x}} g({\tau},\bold{X}_{\tau})\nabla_x {X}_{\tau}|  \mathrm{d}{\tau}\Big)^{2\eta}\notag\\
					&\quad + \Big( \int_{{t_j}}^{t_{j+1}} |e_{\tau} \nabla_{{z}} g({\tau},\bold{X}_{\tau})z_0({\tau})| \mathrm{d}{\tau}\Big)^{2\eta}
					\Bigg\}\Bigg)^{2p}\Bigg]^{\frac{1}{2}}.
				\end{align*}
				By rearranging the terms above, we observe that
				\begin{align*}
					&\mathbb{E}\Big[\sup_{0\leq i\leq N-1}\mathbb{E}^{\mathbb{Q}^{N}}\Big[\sum_{j=i}^{N-1}\Big(\int_{t_j}^{t_{j+1}}|I_s^{2,j}|^{2\eta}\mathrm{d}s\Big)\big/\mathfrak{F}_{t_i}\Big]^p\Big]\\
					\leq& C({\eta},p,\epsilon,q^{*})h^p \mathbb{E}\Bigg[\sum_{k=0}^{\gamma-1}\rho_k\Bigg(\sum_{j=i}^{N-1}\Bigg\{ \Big|(t_{j+1}-t_j)\nabla_yg(\tau,\bold{X}_{\tau})\sum_{i=k}^{r} e_{\tau}  y_i(t) \Big|^{2\eta} \\
					&+  \Big|\sum_{i\geq l_j}^{k-1} e_{t_j}y_i(t_j)\Big|^{2\eta} + \Big(\int_{{t_j}}^{t}\sum_{i\geq k}^r e_{\tau} \nabla_{{z}}g({\tau},\bold{X}_{\tau})z_i({\tau})\mathrm{d}{\tau} \Big)^{2\eta} + \Big( \int_{{t_j}}^{t_{j+1}} \Big|\sum_{i\geq k}^r e_{\tau} z_i({\tau})\Big|^2 \mathrm{d}{\tau} \Big)^{\eta}\Bigg\}\Bigg)^{2p}\Bigg]^{\frac{1}{2}} \notag\\
					& + C({\eta},p,\epsilon,q^{*})h^p \mathbb{E}\Bigg[ \Bigg( \int_{{t_i}}^{T} \big| e_{\tau}  z_0({\tau})\big|^2 \mathrm{d}{\tau} \Bigg)^{2p\eta} +\Bigg( \int_{{t_i}}^{T} |e_{\tau}  \nabla_{{x}} g({\tau},\bold{X}_{\tau})\nabla_x {X}_{\tau}|  \mathrm{d}{\tau}\Bigg)^{4p\eta}\notag\\
					&\quad + \Bigg( \int_{{t_i}}^{T} |e_{\tau} \nabla_{{z}} g({\tau},\bold{X}_{\tau})z_0({\tau})| \mathrm{d}{\tau}\Bigg)^{4p\eta}
					\Bigg]^{\frac{1}{2}}.
				\end{align*}
				By using the same scheme as in the proof of Theorem \ref{thm1.2}, all the terms from the right hand side of the above inequality can be approximated similarly. Hence
				\begin{align*}
					\mathbb{E}\Big[\sup_{0\leq i\leq N-1}\mathbb{E}^{\mathbb{Q}^{N}}\Big[\sum_{j=i}^{N-1}\Big(\int_{t_j}^{t_{j+1}}|I_s^{2,j}|^{2\eta}\mathrm{d}s\Big)\big/\mathfrak{F}_{t_i}\Big]^p\Big]
					\leq& C({\eta},p,\epsilon,q^{*})(h^{p\eta} + h^p)
				\end{align*}
				Therefore, there is a constant $C>0$ not depending on $h$ and on the Lipschitz constants of the coefficients such that for any $p >q^{*}$ and $\eta >1$ the following holds
				\begin{align*}
					\mathbb{E}\Big[\sup_{0\leq i\leq N-1}\mathbb{E}^{\mathbb{Q}^{N}}\Big\{\sum_{j=i}^{N-1}\Big(\int_{t_j}^{t_{j+1}}|Z_s-\hat{Z}_{t_j}|^2\mathrm{d}s\Big)^{\eta}\big/\mathfrak{F}_{t_i}\Big\}^p\Big]
					\leq C h^{p(\eta-1)}(h^{p\eta} + h^p) \leq C h^{p\eta}.
				\end{align*}
				Let us turn now to the general case. By following the same density argument as developed in the proof of Theorem \ref{thm1.1}, we recall that $\lim_{|\pi(r)|\rightarrow 0} Z_t^{\pi(r)} = Z_t$ $\mathrm{d}t \otimes \mathrm{d}\mathbb{P}\text{-a.s.}$ It is then enough to control the terms below to derive the sought estimate. Using the reverse H\"older inequality for all $p>q^{*}$, there is a constant $C>0$ not depending on $\pi(r)$ such that 
				\begin{align*}
					&\mathbb{E}\Big[\sup_{0\leq i\leq N-1}\mathbb{E}^{\mathbb{Q}^{N}}\Big[\sum_{j=i}^{N-1}\Big(\int_{t_j}^{t_{j+1}}|Z_t^{\pi(r)}-Z_t|^{2\eta} +|Z_{t_j}^{\pi(r)} -Z_{t_j}|^{2\eta}\mathrm{d}s\Big)\big/\mathfrak{F}_{t_i}\Big]^p\Big]\\
					&\leq C \mathbb{E}\Big[ \sup_{0\leq i\leq N-1} \mathbb{E}\Big[\Big(\sum_{j=i}^{N-1}\int_{t_j}^{t_{j+1}}|Z_t^{\pi(r)}-Z_t|^{2\eta} +|Z_{t_j}^{\pi(r)} -Z_{t_j}|^{2\eta}\mathrm{d}s\Big)^{q^{*}}\big/\mathfrak{F}_{t_i}\Big]^p  \Big]^{\frac{1}{q^{*}}}\\
					&\leq  C h^{pq^{*}-1}\left[ \int_{0}^{T}\sum_{j=0}^{N-1}\mathbb{E}[|Z_t^{\pi(r)}-Z_t|^{2\eta pq^{*}}  +|Z_{t_j}^{\pi(r)} -Z_{t_j}|^{2\eta pq^{*}}]{\bf{1}}_{[t_j,t_{j+1})} \mathrm{d}s \right]^{\frac{1}{q^{*}}}\\
					&\leq C h^{pq^{*}-1} \left[\int_{0}^{T}\sum_{j=0}^{N-1}\mathbb{E}(|Z_t^{\pi(r)} -Z_t|^{2\eta pq^{*}}+ |Z_{t_j}^{\pi(r)} -Z_{t_j}|^{2\eta pq^{*}}){\bf{1}}_{[t_j,t_{j+1})} \mathrm{d}s\right]^{\frac{1}{q^{*}}},
				\end{align*}
				the latter tends to $0$ when $\pi(r) \rightarrow 0$ by using the dominated convergence theorem. 
			\end{proof}
			\begin{proof}[Proof of Corollary \ref{cor111}]
				Let us first notice that the Euler scheme \eqref{Euler} can be rewritten as follows on the partition $\pi:0=s_0<s_1<\cdots<s_r =1$
				\begin{align*}
					\mathcal{X}^{\pi}(t) &= x +\int_0^t b(s,\mathcal{X}_{\pi(s)}^{\pi})\mathrm{d}s +W_t\\
					\mathcal{X}^{\pi}_t &:= \mathcal{X}^{\pi}({\pi(t)})\\
					\pi(s)&:= \sum_{i=0}^{r-1}t_i{\bf 1}_{[t_i,t_{i+1})}(s) + {\bf 1}_{T}(s)
				\end{align*}
				From Theorem, the following holds for all $p\geq 1$
				\begin{align}
					\mathbb{E}\big[\sup_{0\leq t\leq 1}|X_t -\mathcal{X}^{\pi}(t)|^{2p}\big] \leq C h^{(1-\gamma)p}
				\end{align}
				On the other hand, for every $p\geq 1$ we obtain that
				\begin{align*}
					\mathbb{E}[\sup_{0\leq t\leq 1}|X_t-\mathcal{X}_t^{\pi}|^{2p}] &\leq C(p)\mathbb{E}[\sup_{0\leq t\leq 1}|X_t-\mathcal{X}^{\pi}(t)|^{2p}] + C(p)\mathbb{E}[ \max_{0\leq i\leq r-1}\sup_{t_{i}\leq t\leq t_{i+1}}|\mathcal{X}_{t_i}^{\pi}-\mathcal{X}_t^{\pi}|^{2p}]\\
					&\leq C(p)\mathbb{E}[\sup_{0\leq t\leq 1}|X_t-\mathcal{X}^{\pi}(t)|^{2p}] +C(p) h^{2p} + C(p)\mathbb{E}[ \max_{0\leq i\leq r-1}\sup_{t_{i}\leq t\leq t_{i+1}}|W_t- W_{t_i}|^{4p}]^{\frac{1}{2}}\\
					&\leq C(p)\Big[h^{(1-\gamma)p} + \Big(h\log\big(\frac{1}{h}\big)\Big)^p\Big].
				\end{align*}
			\end{proof}

			\subsection{Proof of Theorem \ref{prop48}}\label{profmaindisc}
			As quoted earlier, Theorem \ref{prop48} is a consequence of Lemma \ref{lemm 4.9}. Therefore, in this subsection, we prove Lemma \ref{lemm 4.9}. 
			
			\begin{proof}[Proof of Lemma \ref{lemm 4.9}]
				We will denote by $\Lambda>0$ the bound of $H$. Let $(H_n) \in C_0^{\infty}(\mathbb{R}^K\times\mathbb{R}^d)$ be a sequence approximating $H$ such that $\sup_{n\in\mathbb{N}}|H_n| \leq \Lambda$. Set 
				\[ G_n(\bold{y};x):= \mathcal{X}_F(\bold{y})H_n(\bold{y};x),\quad G(\bold{y};x):= \mathcal{X}_F(\bold{y})H(\bold{y};x)   \]
				where $F \subseteq \mathbb{R}^d$ is a Borel set	such that $F^c$ is of Lebesgue measure null. We observe that $\lim_{n\rightarrow\infty}G_n(\bold{y};\cdot) = G(\bold{y};\cdot) ,\text{Leb}_d \text{-a.e.,}$ for all $\bold{y} \in \mathbb{R}^K$. Similarly, we also define $U^n$ with $G_n$ in lieu of $G$. Hence, the couple 
				\[ (Y_s^{\bold{y},n;r,x},Z_s^{\bold{y},n;r,x})= \left(U^n(\bold{y};s,X_s^{r,x}),\nabla_x U^n(\bold{y};s,X_s^{r,x})\right)_{s\in[t,R]} \]
				solves the following BSDE of interest on the interval $[t,R]$
				\begin{align}\label{419}
					Y_t^{\bold{y},n;r,x} = G_n(\bold{y};X_R^{r,x}) + \int_t^R g(s,X_s^{r,x},Y_s^{\bold{y},n;r,x},Z_s^{\bold{y},n;r,x})\mathrm{d}s -\int_t^R Z_s^{\bold{y},n;r,x}\mathrm{d}W_s.
				\end{align}
				Moreover, since $G_n$ is bounded and $g$ is continuous and of quadratic growth in the control variable, for each  $n\in \mathbb{N}$, the BSDE \eqref{419} has a unique solution  $\left(U^n(\bold{y};s,X_s^{r,x}),\nabla_x U^n(\bold{y};s,X_s^{r,x})\right)_{s\in[t,R]} \in \mathcal{S}^{\infty}\times \mathcal{H}_{BMO}$ (see for e.g., \cite{Bahlali19}). 
				
				\subsubsection*{\bf{Uniform bound of $U^n(y;t,x)$ :}}
				If $U^n_{0}(y;t,x)$ stands for the solution to \eqref{419} with null driver, then from \cite[Lemma 3.8]{ImkRhOliv24} we deduce that 
				\begin{align*}
					&|U^n(\bold{y};t,x)-U^n_0(\bold{y};t,x)|\\
					&\leq \mathbb{E}\Big[\Big(\int_t^R |g(s,X_s^{t,x},Y_{s,0}^{\bold{y},n;t,x},Z_{s,0}^{\bold{y},n;t,x})|\mathrm{d}s\Big)^{2p}\Big]^{\frac{1}{2p}}\\
					&\leq \mathbb{E}\Big[\Big(\int_t^R 1+ |\mathbb{E}(G_n(\bold{y};X_R^{t,x})/\mathfrak{F}_s)| +|Z_{s,0}^{\bold{y},n;t,x}|^2\mathrm{d}s\Big)^{2p}\Big]^{\frac{1}{2p}}\\
					&\leq (1+\Lambda)R +\mathbb{E}\Big[\Big(\int_t^R |Z_{s,0}^{\bold{y},n;t,x}|^2\mathrm{d}s\Big)^{2p}\Big]^{\frac{1}{2p}}.
				\end{align*}
				On the other hand, $Z_{s,0}^{\bold{y},n;t,x}$ represents the control variable of the solution to our BSDE of interest when the driver is null. Standard computations lead to the following estimate
				\begin{align*}
					\mathbb{E}\Big(\int_t^R |Z_{s,0}^{\bold{y},n;t,x}|^2\mathrm{d}s\Big)^{2p}	\leq C(p) \mathbb{E}(G_n(\bold{y};X_R^{t,x}))^{2p} \leq C(p,q)\Lambda^{2p}.
				\end{align*}
				Therefore, there exists a constant $C>0$ (independent of $n$) such that 
				\begin{align}\label{bound1}
					|U^n(\bold{y};t,x)| \leq C .
				\end{align}
				
				\subsubsection*{{\bf Uniform BMO bound of $Z^{\bold{y},n;r,x}$ :}} Let us first recall the following refinement concerning Lyapunov functions (see \cite{XingZitko18}).
				\begin{defi}[Lyapunov function]
					Let $g:[0,T]\times \mathbb{R}^d\times \mathbb{R}\times\mathbb{R}^d\rightarrow\mathbb{R}$ be a Borel function and $\bold{c}>0$ be a constant. A pair $(h,k)$ of non negative functions with $h\in W^{1,2}_{loc}(\mathbb{R})$ and $k$ a Borel function is said to be a $\bold{c}\text{-Lyapunov}$ pair for $g$ if $h(0)=0,$ $h'(0)=0$ and 
					\begin{align}
						\frac{1}{2}h''(y)|z|^2-h'(y)g(t,x,y,z) \geq |z|^2 -k,
					\end{align} 
					for all $(t,x,y,z)\in [0,T]\times \mathbb{R}^d\times \mathbb{R}\times\mathbb{R}^d\rightarrow\mathbb{R},$ with $|y|\leq \bold{c}.$ We will denote $(h,k)\in \text{\bf Ly}(g,\bold{c})$.
				\end{defi}
				In contrast to \cite[Definition 2.3]{XingZitko18}, the definition requires the function $h$ to be only differentiable in the weak sense. This approach suits to our context, as the function $\ell$ in Assumption \ref{assum2.1} is not assumed to be continuous.
				\begin{lemm}\label{lyapunov} Assume that $g$ satisfies the growth $(iii)$. Then, there is a Lyapunov pair $(h,k)$ depending on such that $(h,k)\in \text{\bf Ly}(g,\bold{c})$.
				\end{lemm}
				\begin{proof}
					Let us choose $h(y) = h_1(y) -y$, where $$h_1(y)= \int_0^y \exp\Big(\kappa \int_0^w\ell(u)\mathrm{d}u\Big)\mathrm{d}w,$$
					where $\kappa>0$ is a free constant.
					We notice that a.e. $h'(y)= \exp\left(\kappa \int_0^y\ell(u)\mathrm{d}u\right) -1$ and
					$h''(y)= \kappa\ell(y)\exp\left(\kappa \int_0^y\ell(u)\mathrm{d}u\right)$ and clearly satisfy $h(0)=h'(0)=0.$ Then, from the growth of $g$ we deduce that
					\begin{align*}
						&\frac{1}{2}h''(y)|z|^2 -h'(y)g(t,x,y,z)\\
						&= \frac{\kappa}{2}\ell(y)|z|^2\exp\Big(\kappa \int_0^y\ell(u)\mathrm{d}u\Big) - \Big[\exp\Big(\kappa \int_0^y\ell(u)\mathrm{d}u\Big) -1\Big]g(t,x,y,z)\\
						&\geq \frac{\kappa}{2}\ell(y)|z|^2\exp\Big(\kappa \int_0^y\ell(u)\mathrm{d}u\Big) -\Lambda (1+|y|+\ell(y)|z|^2)\Big[\exp\Big(\kappa \int_0^y\ell(u)\mathrm{d}u\Big) -1\Big]\\
						&\geq \Lambda\ell(y)|z|^2  -\Lambda e^{\kappa\|\ell\|_{L^1[0,\bold{c}]}} (1+\bold{c}),
					\end{align*}
					the last inequality follows from the fact that the function $\ell$ is non negative and we choose $\kappa = 4\Lambda$. We conclude the proof by taking $\Lambda = \left(\|\ell\|_{L^1[0,\bold{c}]}\right)^{-1}$.
				\end{proof}
				Then, from the It\^o-Krylov formula for BSDEs (see \cite{Ouknine}), we deduce that
				\begin{align}\label{422}
					&\mathrm{d}h(Y_s^{\bold{y},n;r,x})\notag\\&= \Big(\frac{1}{2}h''(Y_s^{\bold{y},n;r,x})|Z_s^{\bold{y},n;r,x}|^2-h'(Y_s^{\bold{y},n;r,x})g(s,X_s^{r,x},Y_s^{\bold{y},n;r,x},Z_s^{\bold{y},n;r,x}) \Big)\mathrm{d}s +\mathrm{d}M_s^{\bold{y},n;r,x},
				\end{align}
				where $M_s^{\bold{y},n;r,x}$ is a martingale. For any stopping time $\tau \in [t,R]$ and applying the property of the Lyapunov pair $(h,k)$, we deduce that
				\begin{align*}
					\sup_{n\in\mathbb{N}} \mathbb{E}\Big[\int_{\tau}^R |Z_s^{\bold{y},n;t,x}|^2\mathrm{d}s/\mathfrak{F}_{\tau}\Big] &\leq \mathbb{E}\Big[h(G_n(\bold{y};X_R^{t,x}))- h(U^n(\bold{y};\tau,X_{\tau}^{t,x}))+ k(R-\tau)/\mathfrak{F}_{\tau}\Big]\\
					&\leq 2\|h \circ U^n\|_{L^{\infty}} + kR.
				\end{align*}
				
				\subsubsection*{{\bf Uniform bound of $\nabla_xU^n(\bold{y};t,x)$ :}} Let us now establish a bound for the gradient of the solution $U^n(\bold{y};r,x)$, uniformly in $n$. We assume $H\in C^1(\mathbb{R}^{K\times d},\mathbb{R})$ and that the driver $g$ is continuously differentiable, while the drift $b$ is uniformly Lipschitz continuous.
				We will consider the following FBSDE for all $0\leq r\leq t\leq R$
				\begin{align}\label{423}
					\begin{cases}
						\mathrm{d}X_t^{r,x}= \mathrm{d}\tilde{W}_t,\\
						\mathrm{d}Y_t^{\bold{y},n;r,x}= -(Z_t^{\bold{y},n;r,x}b(t,X_t^{r,x})+g(t,X_t^{r,x},Y_t^{\bold{y},n;r,x},Z_t^{\bold{y},n;r,x}))\mathrm{d}t + Z_t^{\bold{y},n;r,x}\mathrm{d}\tilde{W}_t,\\
						X_r^{r,x} = x, Y_R^{\bold{y},n;r,x} = G_n(\bold{y};X_R^{r,x}),
					\end{cases}
				\end{align}
				where $\tilde{W}_t = W_t +\int_0^t b(v,X_v^{r,x})\mathrm{d}v$ is a Brownian motion under the probability measure $\tilde{\mathbb{P}}$ given by $\mathrm{d}\tilde{\mathbb{P}}/\mathrm{d}\mathbb{P}:= \mathcal{E}(\int_0^tb(v,X_v^{r,x})\mathrm{d}W_v)$, and $(Y_{\cdot}^{\bold{y},n;r,x},Z_{\cdot}^{\bold{y},n;r,x}) = (U^n(\bold{y};\cdot,X_{\cdot}^{r,x}),\nabla_xU^n(\bold{y};\cdot,X_{\cdot}^{r,x}))$.
				
				Under the above setting, we know that the FBSDE of interest is differentiable with respect to the initial value $x\in\mathbb{R}^d$ under the probability measure $\tilde{\mathbb{P}}$. Moreover, the dynamics of the derivative process is given by
				\begin{align*}
					\nabla_x Y_t^{\bold{y},n;r,x} &= \nabla_x G_n(\bold{y};X_R^{r,x})\nabla_x X_R^{r,x} + \int_t^R Z_s^{\bold{y},n;r,x}\nabla_x b(s,X_s^{r,x})\nabla_xX_s^{r,x}\mathrm{d}s -\int_t^R \nabla_x Z_s^{\bold{y},n;r,x}\mathrm{d}\tilde{W}_s  \\&\quad+  \int_t^R b(s,X_r^{r,x}) \nabla_x Z_s^{\bold{y},n;r,x}\mathrm{d}s 
					+\int_t^R \langle\nabla g(s,\bold{X}_s^{\bold{y},n;r,x}),\nabla \bold{X}_s^{\bold{y},n;r,x}\rangle\mathrm{d}s,
				\end{align*}
				where $\bold{X}_s^{y,n;r,x}=({X}_s^{r,x},{Y}_s^{\bold{y},n;r,x},{Z}_s^{\bold{y},n;r,x})$. Or, equivalently under the measure $\mathrm{d}\mathbb{Q}/\mathrm{d}\tilde{\mathrm{P}} = \mathcal{E}(\int_0^{\cdot}b(s,X^{r,x}_s)+\nabla_zg(s,\bold{X}_s^{\bold{y},n;r,x})\mathrm{d}\tilde{W}_s)$, it writes
				\begin{align*}
					\nabla_x Y_t^{\bold{y},n;r,x} 
					&= e^{\int_t^R\nabla_yg(s,\bold{X}_s^{\bold{y},n;r,x})\mathrm{d}s}\nabla_x G_n(\bold{y};X_R^{r,x})\nabla_x X_R^{r,x}-\int_t^R e^{\int_t^s\nabla_yg(v,\bold{X}_v^{\bold{y},n;r,x})\mathrm{d}v}\nabla_x Z_s^{\bold{y},n;r,x}\mathrm{d}\bold{B}_s\\&\quad + \int_t^R e^{\int_t^s\nabla_yg(v,\bold{X}_v^{\bold{y},n;r,x})\mathrm{d}v} Z_s^{\bold{y},n;r,x}\nabla_x b(s,X_s^{r,x})\nabla_xX_s^{r,x}\mathrm{d}s \\
					&\quad+\int_t^R e^{\int_t^s\nabla_yg(v,\bold{X}_v^{\bold{y},n;r,x})\mathrm{d}v}\nabla_x g(s,X_s^{r,x},Y_s^{\bold{y},n;r,x},Z_s^{\bold{y},n;r,x})\nabla_x{X}_s^{r,x}\mathrm{d}s,
				\end{align*}
				with $\bold{B}_t:= \tilde{W}_t -\int_0^tb(s,X^{r,x}_s)+\nabla_zg(s,\bold{X}_s^{\bold{y},n;r,x})\mathrm{d}s$, which is a Brownian motion under $\mathbb{Q}$. It is important to note that the dynamics of the first variation process $\nabla_x X_t^{r,x}$ remains unchanged under the new probability measure $\mathbb{Q}$. Hence, by taking the conditional expectation, we obtain that
				\begin{align}
					\nabla_x Y_t^{\bold{y},n;r,x} 
					&=\mathbb{E}^{\mathbb{Q}}\Big[ e^{\int_t^R\nabla_yg(s,\bold{X}_s^{\bold{y},n;r,x})\mathrm{d}s}\nabla_x G_n(\bold{y};X_R^{r,x})
					+ \int_t^R e^{\int_t^s\nabla_yg(v,\bold{X}_v^{\bold{y},n;r,x})\mathrm{d}v} Z_s^{\bold{y},n;r,x}\nabla_x b(s,X_s^{r,x})\mathrm{d}s \notag \\
					&\quad\quad+\int_t^R e^{\int_t^s\nabla_yg(v,\bold{X}_v^{\bold{y},n;r,x})\mathrm{d}v}\nabla_x g(s,X_s^{r,x},Y_s^{\bold{y},n;r,x},Z_s^{\bold{y},n;r,x})\mathrm{d}s \Big/\mathfrak{F}_t\Big].
				\end{align}
				In the case where $H$ is Lipschitz continuous and the driver $g$ is also Lipschitz continuous in $x$ and $y$, we obtain the following uniform bound of $Z_s^{\bold{y},n;r,x}$
				\begin{align*}
					\sup_{n\in\mathbb{N}}|Z_s^{\bold{y},n;r,x}| \leq  C \quad \text{for all } \bold{y}\in \mathbb{R}^K,
				\end{align*}
				where $C$ is a constant depending on $\|\nabla_xb\|_{\infty},\|\nabla_xG_n\|_{\infty}$, $\|Z^{\bold{y},n,r,x}*W\|_{BMO}$ and the Lipschitz constants of $g$ in $x$ and $y$ respectively.
				
				Let us recall that $\nabla_x Y_t^{\bold{y}.n;r,x}(\nabla_x X_t^{r,x})^{-1}$ is a continuous version of the process $Z_t^{y,n;r,x}$ for all $\bold{y}\in \mathbb{R}^K$.
				Define 
				\begin{align}\label{428}
					F_t^{\bold{y},n,r,x}&= e^{\int_r^t\nabla_yg(s,\bold{X}_s^{\bold{y},n;r,x})\mathrm{d}s}\nabla_x Y_t^{\bold{y},n;r,x} + \int_r^t e^{\int_r^s\nabla_yg(v,\bold{X}_v^{\bold{y},n;r,x})\mathrm{d}v} Z_s^{\bold{y},n;r,x}\nabla_x b(s,X_s^{r,x})\mathrm{d}s \notag\\
					&\quad+ \int_r^t e^{\int_r^s\nabla_yg(v,\bold{X}_v^{\bold{y},n;r,x})\mathrm{d}v}\nabla_x g(s,X_s^{r,x},Y_s^{\bold{y},n;r,x},Z_s^{\bold{y},n;r,x})\mathrm{d}s.
				\end{align}
				We notice that for all $t\leq R,$
				$\mathbb{E}^{\mathbb{Q}}[F_R^{\bold{y},n,r,x}/\mathfrak{F}_t] =  F_t^{\bold{y},n,r,x} $ which implies that the process $F^{\bold{y},n,r,x}$ is a $\mathbb{Q}\text{-martingale}$. Therefore, $(F_t^{\bold{y},n,r,x})^2$ is a $\mathbb{Q}\text{-submartingale}$ and the following holds
				\begin{align}
					\mathbb{E}^{\mathbb{Q}}\int_t^R|F_s^{\bold{y},n,r,x}|^2\mathrm{d}s\geq (R-t)|\nabla_x Y_r^{\bold{y},n;r,x}|^2= (R-t)|Z_r^{\bold{y},n;r,x}|^2.
				\end{align}
				Hence, from \eqref{428} there is a constant depending on $\|\nabla b\|_{\infty},\|Z^{\bold{y},n;r,x}*W\|_{\bmo}$ and $e^{\|\nabla_y g\|_{\infty}}$ such that
				$\mathbb{E}^{\mathbb{Q}}\int_t^R|F_s^{\bold{y},n,r,x}|^2\mathrm{d}s \leq C$.
				Note also that the constant $C$ above does not depend on $\|\nabla_xG_n(\bold{y};\cdot)\|_{\infty}$ for all $\bold{y} \in \mathbb{R}^K$. Thus
				\begin{align}\label{eq427}
					\sup_{n\in\mathbb{N}}|Z_r^{\bold{y},n;r,x}| \leq \frac{C}{(R-r)^{1/2}}.
				\end{align}
				Hence from \cite[Theorem 6.2]{MaZhang02}\footnote{Note that the uniform Lipschitz condition on the forward component of the driver can be relaxed by means of standard approximations}, we observe that  the following representation holds for all $r\leq \rho \leq R$ and $\bold{y}\in\mathbb{R}^K$
				\begin{align*}
					&\nabla_x U^n(\bold{y};r,x)\\& = \mathbb{E}^{\tilde{\mathbb{P}}}\Big[U^n(\bold{y};R,X_{R}^{r,x})N_{R}^{r,x}+\int_{r}^{R}\Big( Z_{s}^{\bold{y},n;r,x}b({s},X_{s}^{r,x})-g({s},X_{s}^{r,x},Y_{s}^{\bold{y},n;r,x},Z_{s}^{\bold{y},n;r,x})\Big)N_{{s}}^{r,x}\mathrm{d}s\Big],
				\end{align*}
				where $N_{s}^{r,x}$ stands for the Malliavin weights associated to the forward equation in \eqref{423} under the measure $\tilde{\mathbb{P}}$, which is given by 
				\begin{align}
					N_t^{r,x} = \frac{1}{t-r} \int_r^t \nabla X_s\mathrm{d}\tilde{W}_s[\nabla X_r]^{-1}.
				\end{align}
				Moreover, one can show that the process $N_{s}^{r,x}$ satisfies the following for all $p\geq 1$ and $r \leq t$
				\begin{align}
					\begin{cases}\label{429}
						\mathbb{E}^{\bar{\mathbb{P}}}[|N_t^{r,x}|^p]\leq C(t-r)^{-p/2},\\
						\mathbb{E}^{\bar{\mathbb{P}}}\left[\sup_{s\in[\frac{R+r}{2},R]}|N_s^{r,x}|^p\right]\leq C(R-r)^{-p/2}.
					\end{cases}
				\end{align}
				Hence, using \eqref{bound1},\eqref{429}, the growth of $g$ and the fact that $b$ is bounded, we deduce that 
				\begin{align*}
					&|\nabla_x U^n(y;r,x)|\\& \leq \frac{C}{\sqrt{R-r}}+C\mathbb{E}^{\tilde{\mathbb{P}}}\int_{r}^{\frac{R+r}{2}}\left(1+\|b\|_{\infty}|Z_s^{\bold{y},n;r,x}|+|Y_s^{\bold{y},n;r,x}|+ \ell(|Y_s^{\bold{y},n;r,x}|)|Z_s^{\bold{y},n;r,x}|^2\right)|N_s^{r,x}|\mathrm{d}s\\&\quad + C\mathbb{E}^{\tilde{\mathbb{P}}}\int_{\frac{R+r}{2}}^{R}\left(1+|Y_s^{\bold{y},n;r,x}|+ (1+\ell(|Y_s^{\bold{y},n;r,x}|))|Z_s^{\bold{y},n;r,x}|^2\right)|N_s^{r,x}|\mathrm{d}s \\
					&= \frac{C}{\sqrt{R-r}} + I_1 +I_2.
				\end{align*}
				
				Using mainly the bound \eqref{bound1}and \eqref{429}, the local boundedness of the function $\ell$ and H\"older's inequality, we obtain the following estimate for $I_1$, for all $\bold{y}\in\mathbb{R}^K$
				\begin{align*}
					I_1 &\leq C \mathbb{E}^{\tilde{\mathbb{P}}}\int_{r}^{\frac{R+r}{2}}\left((1+ |Z_s^{\bold{y},n;r,x}|^{3/2})(1+ |Z_s^{\bold{y},n;r,x}|^{1/2})\right)|N_s^{r,x}|\mathrm{d}s\\
					&\leq C \sup_{s\in[r,\frac{R+r}{2}]}(1+ |Z_s^{\bold{y},n;r,x}|^{3/2})\mathbb{E}^{\tilde{\mathbb{P}}}\int_{r}^{\frac{R+r}{2}}\left(1+ |Z_s^{\bold{y},n;r,x}|^{1/2}\right)|N_s^{r,x}|\mathrm{d}s\\
					&\leq C \Big(1+ \frac{(R-r)^{-3/4}}{2}\Big)\mathbb{E}^{\tilde{\mathbb{P}}}\Big[\int_{r}^{\frac{R+r}{2}}(1+ |Z_s^{\bold{y},n;r,x}|^2)\mathrm{d}s\Big]^{1/4}\mathbb{E}^{\tilde{\mathbb{P}}}\Big[\int_{r}^{\frac{R+r}{2}}|N_s^{r,x}|^{4/3}\mathrm{d}s\Big]^{3/4}\\
					&\leq C (R-r)^{-3/4}(R-r)^{1/4},
				\end{align*}
				where the constant $C$ above does not depend on $n$ and not on the Lipschitz constant of the terminal value $G_n$.
				
				On the other hand, the bound for $I_2$ follows by
				\begin{align*}
					I_2 &\leq C \mathbb{E}^{\tilde{\mathbb{P}}}[\sup_{s\in[\frac{R+r}{2},R]}|N_s^{r,x}|]\int_{\frac{R+r}{2}}^{R}\left(1+|Y_s^{\bold{y},n;r,x}|+ (1+\ell(|Y_s^{\bold{y},n;r,x}|))|Z_s^{\bold{y},n;r,x}|^2\right)\mathrm{d}s\\
					&\leq C \mathbb{E}^{\tilde{\mathbb{P}}}\Big[\sup_{s\in[\frac{R+r}{2},R]}|N_s^{r,x}|^q\Big]^{1/q}\mathbb{E}^{\tilde{\mathbb{P}}}\Big[\Big(\int_{\frac{R+r}{2}}^{R}1+ |Z_s^{\bold{y},n;r,x}|^2\mathrm{d}s\Big)^p\Big]^{1/p}\\
					&\leq C (R-r)^{-1/2}.
				\end{align*} 
				Combining the bounds for $I_1$ and $I_2$, we deduce that there is a constant $C$ such that for all $\bold{y}\in  \mathbb{R}^K$
				\begin{align}\label{bound2}
					\sup_{n\in\mathbb{N}}(R-r)^{1/2}|\nabla_x U^n(\bold{y};r,x)| \leq C.
				\end{align}
				Consequently, we deduce that the function $U^n(\bold{y},r,x)$ is measurable on $\mathbb{R}^K\times [r,R]\times\mathbb{R}^d$.
				
				\subsubsection*{{\bf The limit $U(\bold{y};r,x)$ and its gradient $\nabla_x U(\bold{y};r,x)$:}}
				Let us now establish some properties satisfied by the limit function $U$ and its derivative. Applying \cite[Lemma 3.8]{ImkRhOliv24}, we have
				\begin{align*}
					&\|Y_t^{\bold{y};r,x}-U(\bold{y};t,x)\|_2\\
					&\leq \|Y_t^{\bold{y};r,x}-U^n(\bold{y};t,x)\|_2 + \|U^n(\bold{y};t,x)-U(\bold{y};t,x)\|_{2}\\
					&\leq C \|G_n(\bold{y},X_R^{r,x})-G(\bold{y},X_R^{r,x})\|^{2} +\|U^n(\bold{y};t,x)-U(\bold{y};t,x)\|_{2}.
				\end{align*}
				Thanks to \eqref{bound1}, we have that $|U(\bold{y};t,x)| \leq C$. Applying the dominated convergence theorem and considering the smoothness properties of the transition density of the diffusion process $X^{r,x}$, we conclude that both terms on the right side of the inequality tend to zero as $n$ goes to infinity. Consequently,
				\begin{align*}
					Y_t^{\bold{y};r,x} = U(\bold{y};t,X_t^{r,x}) \text{ a.s.}, \text{ for all } t\in[r,R].
				\end{align*}
				On the other hand, we notice that the following representation remains valid under the underlying probability measure $\mathbb{P}$:
				\begin{align*}
					& \nabla_x U^n(\bold{y};r,x)\\
					&=\mathbb{E}\Big[G_n(\bold{y};X_{R}^{r,x})\tilde{N}_{R}^{r,x}+\int_{r}^{R} g({s},X_{s}^{r,x},Y_{s}^{\bold{y},n;r,x},Z_{s}^{\bold{y},n;r,x})\tilde{N}_{{s}}^{r,x}\mathrm{d}s\Big],
				\end{align*}
				where the Malliavin weights $\tilde{N}^{r,x}$ is now given by 
				\begin{align}
					\tilde{N}_t^{r,x} = \frac{1}{t-r} \int_r^t \nabla X_s\mathrm{d}{W}_s[\nabla X_r]^{-1}.
				\end{align}
				and $\nabla X_s$ stands for the first variation process of the diffusion $X$ given by \eqref{2.1} and $[\nabla X_r]^{-1}$ is its inverse. Applying the dominated convergence theorem, we obtain that 
				\begin{align*}
					&\lim_{n\rightarrow \infty} \nabla_x U^n(\bold{y};r,x)\\
					&= \lim_{n\rightarrow \infty}\mathbb{E}\Big[G_n(\bold{y};X_{R}^{r,x})\tilde{N}_{R}^{r,x}+\int_{r}^{R}g({s},X_{s}^{r,x},Y_{s}^{\bold{y},n;r,x},Z_{s}^{\bold{y},n;r,x})\tilde{N}_{{s}}^{r,x}\mathrm{d}s\Big]\\
					&=\mathbb{E}\Big[G(\bold{y};X_{R}^{r,x})\tilde{N}_{R}^{r,x}+\int_{r}^{R}g({s},X_{s}^{r,x},Y_{s}^{\bold{y},n;r,x},Z_{s}^{\bold{y},n;r,x})\tilde{N}_{{s}}^{r,x}\mathrm{d}s\Big]:= V(\bold{y};r,x).
				\end{align*}
				It follows from \eqref{bound2} that for all $\bold{y}\in\mathbb{R}^K$
				\begin{align*}
					V(\bold{y};r,x) \leq C(R-t)^{1/2}.
				\end{align*}
				This implies that for all $\delta\in (0,R)$
				\begin{align*}
					\lim_{n\rightarrow \infty} \int_r^{R-\delta} \|\nabla_x U^n(\bold{y};s,X_s^{r,x}) - V(\bold{y};s,X_s^{r,x})\|_2\mathrm{d}s =0.
				\end{align*}
				Notice also that, from \cite{ImkRhOliv24}, we have that 
				\begin{align*}
					\lim_{n\rightarrow \infty} \int_r^{R} \|\nabla_x U^n(\bold{y};s,X_s^{r,x}) - Z_s^{\bold{y};r,x}\|_2\mathrm{d}s =0.
				\end{align*}
				Consequently, for almost every $s\in [r,R)$ we have 
				\begin{align*}
					Z_s^{\bold{y};r,x} = V(\bold{y};s,X_s^{r,x}),\, \text{ a.s.}
				\end{align*}
			\end{proof}
			\subsection{Proof of Theorem \ref{Main63}}\label{sub4.3}
			To prove Theorem \ref{Main63}, we require some auxiliary results on the centered (and explicit) BTZ schemes and their properties within our framework. For ease of reading, their proofs are deferred to the appendix.
			
			\subsubsection{Centered BTZ scheme and some properties}
			In this section, we revisit some properties of the well known BTZ scheme adapted to our general framework. Specifically, we account for the generator not being uniformly Lipschitz continuous in its backward component, and we address a broader class of nonninearities with quadratic growth. To conduct our analysis, we use the scheme from \cite{Zhou}, whic unlike \cite{ChaRichou16} proposes a modified version of the BTZ scheme, referred to as the centered BTZ scheme.
			
			{\bf Linearisation of the centered BTZ scheme:}
			Let $\pi : 0 = t_0 < t_1 < \cdots < t_N = T$ be a partition, where
			$\delta t_i = t_{i+1} -t_i$ and $h = \sup_{0\leq i\leq N-1} \delta t_i$ . Let $(H_i)_{0\leq i \leq N-1}$ be a sequence of random vectors such that 
			\begin{itemize}
				\item $H_i$ is a $\mathbb{R}^d\text{-valued}$ random vectors and is $\mathfrak{F}_{t_{i+1}}\text{-measurable},$ for all $0\leq i\leq N-1$;
				\item for all $0\leq i\leq N-1$, $H_i$ satisfies 
				\begin{align*}
					\mathbb{E}[H_i/\mathfrak{F}_{t_{i}}]=0,\,\, \mathbb{E}[H_i(H_i)^{\bold T}/\mathfrak{F}_{t_{i}}]= \mathbb{E}[H_i(H_i)^{\bold T}]= \frac{c_i}{\delta t_i}I_{d\times d}
				\end{align*}
				where $C_1 < c_i < C_2$ for some constants $C_1, C_2>0$ independent of $i$.
			\end{itemize}
			We will consider the following modified BTZ centered scheme 
			\begin{align}\label{linBTZ1}
				\begin{cases}
					Y_i^{\pi} = \mathbb{E}\left[ 	Y_{i+1}^{\pi} + G_i(\mathbb{E}[	Y_{i+1}^{\pi}/\mathfrak{F}_{t_{i}}],Z_{i}^{\pi})\delta t_i\big/ \mathfrak{F}_{t_{i}}  \right],\\
					Z_{i}^{\pi}= \mathbb{E}[Y_{i+1}^{\pi}H_i\big/\mathfrak{F}_{t_{i}}],\\
					Y_N^{\pi}= \xi, \,\,\, Z_N^{\pi}= 0.	
				\end{cases}
			\end{align}
			
			To ease the notation, we will denote by $\mathcal{E}[G_i,\xi, H_i]$ the solution of the BTZ scheme \eqref{linBTZ1}. Unless otherwise stated, we will omit the superscript $\pi$ in the sequel. 
			\begin{remark}
				The BTZ scheme \eqref{linBTZ1} can be rewritten without the conditional expectation as follows 
				\begin{align}\label{eq 6.2}
					Y_i^{\pi} = Y_{i+1}^{\pi} + \delta t_i G_i(\mathbb{E}[Y_{i+1}^{\pi}/\mathfrak{F}_{t_i}],Z_i^{\pi}) -\frac{\delta t_i}{c_i}(Z_i^{\pi})^{\bold T}\cdot H_i - \nu_i,
				\end{align}
				where $\nu_i$ is an $\mathfrak{F}_{t_{i+1}}\text{-measurable}$ random variable such that $\mathbb{E}[{\nu}_i/\mathfrak{F}_{t_i}]=\mathbb{E}[{\nu}_iH_i/\mathfrak{F}_{t_i}]=0$ and $\mathbb{E}[({\nu}_i)^2/\mathfrak{F}_{t_i}]<\infty.$
			\end{remark}

			\begin{assum}[Driver locally Lipschitz in $z$ and $y$]\label{assum56} The functions $G_i: \Omega \times \mathbb{R}\times \mathbb{R}^d \rightarrow \mathbb{R}$ are $\mathfrak{F}_{t_i}\otimes\mathcal{B}(\mathbb{R})\otimes\mathcal{B}(\mathbb{R}^d)\text{-measurable}$ and satisfy for some positive constants $\Lambda_y$ and $\Lambda_z(n)$ which do not depend on $i$ but $\Lambda_z(n)$ may depend on $n$,
				\begin{itemize}
					\item[(i)] For all $y,y',z,z'$ and $\alpha_0\in [0,1)$
					\[ |G_i(y,z)-G_i(y',z')|\leq \Lambda_y(1+|z-z'|^{\alpha_0})|y-y'| + \Lambda_z(n)\ell(|y-y'|)|z-z'|,  \]
					with $\ell$ being a non-negative and locally bounded function.
					\item[(ii)] Recall that $h=\sup_i \delta t_i$. There is a constant $\epsilon \in ]0,1[$ which does not depend on $n$ and $N$, such that 
					\[ h \Lambda_y + \Big(\sup_{0\leq i \leq N-1} \delta t_i| H_i|\Big)\ell(0)\Lambda_z(n) \leq 1-\epsilon.\]
				\end{itemize}
			\end{assum}
			\begin{assum}[Bounded terminal value]\label{assum52}
				Let the functions $G_i: \Omega \times \mathbb{R}\times \mathbb{R}^d \rightarrow \mathbb{R}$ being $\mathfrak{F}_{t_i}\otimes\mathcal{B}(\mathbb{R})\otimes\mathcal{B}(\mathbb{R}^d)\text{-measurable}$ and satisfy  Assumption \ref{assum56} for all $1\leq i\leq N$.
				\item[(i)] $\xi \in L^{\infty}(\mathfrak{F}_T)$ and $\sup_{1\leq i\leq N}|G_i(0,0)| \leq \Lambda_0$, where $\Lambda_0$ is a positive constant not depending of $N$.
				\item[(ii)] Moreover, $	\ell(0)\Lambda_z(n)< 1$ where $\Lambda_z(n)$ is the Lipschitz constant of $G_i$ with respect to $z$, for all $1\leq i\leq N$.
			\end{assum}
			
			\begin{assum}[Special quadratic structure]\label{assum53}
				Let the functions $G_i: \Omega \times \mathbb{R}\times \mathbb{R}^d \rightarrow \mathbb{R}$ being $\mathfrak{F}_{t_i}\otimes\mathcal{B}(\mathbb{R})\otimes\mathcal{B}(\mathbb{R}^d)\text{-measurable}$ and satisfy  Assumption \ref{assum52} for all $1\leq i\leq N-1$. Furthermore,
				for all $y,y',z,z'$ and $\alpha_0 \in (0,1)$
				\[ |G_i(y,z)-G_i(y',z')|\leq \Lambda_y(1+|z-z'|^{\alpha_0})|y-y'| + \Lambda_z\Big(1+\ell(|y-y'|)(|z|+|z'|)\Big)|z-z'|. \]
			\end{assum}
			Notice that, under Assumption \ref{assum53}, the driver $G_i$ has necessarily of the following growth 
			\begin{align}
				|G_i(y,z)| \leq \Lambda_0 + \Lambda_y |y| + \Lambda_z(|z|+2\ell(|y|)|z|^2),
			\end{align}
			for all $0\leq i\leq N$.
			
			Some important properties satisfied by the centered scheme \eqref{linBTZ1} need to be assessed under our general set of assumptions. One key property is the comparison theorem, presented in the lemma below. Additionally, we provide a representation result for the difference between two BTZ schemes when the driver satisfies Assumption \ref{assum56}.
			
			{In the sequel, we set $\delta Y_i =  Y_{i}^1- Y_{i}^2$, $\delta Z_i =  Z_{i}^1- Z_{i}^2$ and $\delta G_i = G_i^1(Y_{i}^2, Z_{i}^2)-G_i^2( Y_{i}^2, Z_{i}^2)$, where $( Y_{i}^1, Z_{i}^1)$ and $(Y_{i}^2, Z_{i}^2)$ denote respectively, the solutions to $\mathcal{E}[G_i^1,\xi^1, H_i]$ and $\mathcal{E}[G_i^2,\xi^2, H_i]$.}
			\begin{lemm}\label{Lemm6.4} Let $(Y_{i}^1, Z_{i}^1)$ and $(Y_{i}^2, Z_{i}^2)$ be two solutions to the scheme \eqref{linBTZ1}, respectively, such that $(G_i^1,\xi^1)$ and $(G_i^2,\xi^2)$ satisfy Assumption \ref{assum56}. For $0\leq i\leq n$, let
				\begin{align*}
					\Upsilon_i:= \prod_{j=i}^{N-1}(1+\delta t_j\Gamma_j +\delta t_j\Pi_j \cdot {H}_j),
				\end{align*}
				where 
				\begin{align*}
					\Gamma_j:&= \frac{G_j^1(\mathbb{E}[{Y}_{j+1}^1/\mathfrak{F}_{t_j}],{Z}_j^{1})-G_j^1(\mathbb{E}[{Y}_{j+1}^{2}/\mathfrak{F}_{t_j}],{Z}_j^{1})}{\mathbb{E}[\delta{Y}_{j+1}/\mathfrak{F}_{t_j}]}{\bf 1}_{\{\mathbb{E}[\delta{Y}_{j+1}/\mathfrak{F}_{t_j}]\neq 0\}},\\
					\Pi_j:&= \frac{G_j^1(\mathbb{E}[{Y}_{j+1}^{2}/\mathfrak{F}_{t_j}],{Z}_j^{1})-G_j^1(\mathbb{E}[{Y}_{j+1}^{2}/\mathfrak{F}_{t_j}],{Z}_j^{2})}{|\delta{Z}_{j}|^2}\delta{Z}_{j}{\bf 1}_{\{\delta{Z}_{j}\neq 0\}}.
				\end{align*}
				Then the following representation holds
				\begin{align}\label{eq6.3}
					\delta Y_i = \mathbb{E}\Big[\Upsilon_i\delta Y_{t_N} + \sum_{k=i}^{N-1}\delta t_k\delta G_k \frac{\Upsilon_i}{\Upsilon_k}\Big/\mathfrak{F}_{t_i}\Big].
				\end{align}
				Moreover, the following comparison principle holds true: $ Y_i^1\geq Y_i^2$,
				whenever $Y_N^1 \geq Y_N^2$ and $G_j^1(Y_i^2, Z_i^2)\geq G_j^2( Y_i^2, Z_i^2)$.
			\end{lemm}
			\begin{proof}
				See Appendix \ref{secauxres1}
			\end{proof}
			{Below, we derive another important property satisfied by the centered BTZ scheme \eqref{linBTZ1} under our framework. This property involves the uniform boundedness of the component $Y_i$	of the scheme, which plays a key role in ensuring stability and error bounds in our setting.}
			\begin{lemm}\label{BTZbound}
				Let Assumption \ref{assum52} be in force. Then 
				\begin{align*}
					|Y_i| \leq (\|\xi\|_{\infty} + \Lambda_0 T)\exp\left(1+\Lambda_yT \right).
				\end{align*}
			\end{lemm}
			\begin{proof}
				See Appendix \ref{secauxres1}
			\end{proof}
			\begin{remark}
				We emphasize that the bound obtained above does not depend on $\alpha \in (0,1)$, despite the driver being stochastic Lipschitz in $y$. Moreover, this bound is similar to that of the component solution $Y$ to \eqref{2.2} (see, for instance, \cite{ImkRhOliv24}).
			\end{remark}
			{We conclude this subsection with another important property: we prove that the BMO bound of the control process $Z_i$ in the BTZ scheme \eqref{linBTZ1} is preserved under our assumptions.}
			\begin{lemm}\label{BTZbmo}	Let Assumption \ref{assum53} be in force and let $( Y_i^{\pi}, Z_i^{\pi})$ be the solution associated to $\mathcal{E}[g_i,\xi, H_i]$. Set $ Z_i^{\pi}(s):= Z_i^{\pi}$ for all $s\in [t_i,t_{i+1})$, for all $0\leq i \leq N.$ Then, there exists a constant $C>0$ not depending on $\pi$ such that $$\mathbb{E}\Big[\sum_{j=i}^{N-1}{\delta t_j}| Z_{j}^{\pi}|^2\Big/\mathfrak{F}_{t_i}\Big] \leq C.$$
			\end{lemm}
			\begin{proof}
				See Appendix \ref{secauxres1}
			\end{proof}

			{ \bf Stability of the explicit BTZ scheme:} Here we study the difference between the scheme given by \eqref{linBTZ1} and its perturbed version given below.
			\begin{align}\label{linBTZ3}
				\begin{cases}
					\tilde Y_i = \mathbb{E}\left[ 	\tilde Y_{i+1} + G_i(\mathbb{E}[	\tilde Y_{i+1}/\mathfrak{F}_{t_{i}}],\tilde Z_{i})\delta t_i\big/ \mathfrak{F}_{t_{i}}  \right] +\zeta_i^{Y},\\
					\tilde Z_{i}= \mathbb{E}[\tilde Y_{i+1}H_i\big/\mathfrak{F}_{t_{i}}],\\
					\tilde Y_N= \xi, \,\,\, \tilde Z_N= 0,
				\end{cases}
			\end{align}
			where $(\zeta_i^{Y})_i$ is a family of $\mathfrak{F}_{t_i}\text{-measurable}$ and square integrable random variables. furthermore, we assume
			\begin{align}\label{611}
				\sup_{0\leq i<N}\mathbb{E}\Big[\sum_{j=i}^{N-1}|\tilde{Z}_j|^2\delta t_j\big/\mathfrak{F}_{t_i}\Big] < c.
			\end{align}
			In the sequel, we will adopt the following notations for all $0\leq i \leq N$
			\begin{align*}
				\tilde \delta Y_i^{\pi}:= \tilde Y_i -Y_i^{\pi},\quad \tilde \delta Z_i^{\pi}:= \tilde Z_i -Z_i^{\pi},
			\end{align*}
			and 
			\begin{align*}
				\tilde \delta Y_i^{\pi}:=\tilde{\delta} Y^{\pi}(t), \forall t\in [t_i,t_{i+1}]\quad \tilde \delta Z_i^{\pi}:=\tilde \delta Z^{\pi}(t) , \forall t\in [t_i,t_{i+1}].
			\end{align*}
			Additionally, we introduce the following processes for $0\leq i\leq n$,
			\begin{align*}
				\tilde E_i^{\pi}:= \prod_{j=i}^{N-1}(1 +\delta t_j\tilde \Pi_j^{\pi} \cdot {H}_j),
			\end{align*}
			and as in Lemma \ref{Lemm6.4}, we define		\begin{align*}
				\tilde \Pi_j^{\pi}:&= \frac{G_j(\mathbb{E}[\tilde{Y}_{j+1}/\mathfrak{F}_{t_j}],\tilde{Z}_j)-G_j(\mathbb{E}[\tilde{Y}_{j+1}/\mathfrak{F}_{t_j}],{Z}_j^{\pi})}{|\tilde\delta{Z}_{j}^{\pi}|^2}\tilde\delta{Z}_{j}^{\pi}{\bf 1}_{\{\tilde\delta{Z}_{j}^{\pi}\neq 0\}}.
			\end{align*}
			The processes $\tilde{\Gamma}_j^{\pi}$ and $\tilde{\Upsilon}_j^{\pi}$ are also defined in a similar way. Under Assumption \ref{assum56}, for all $1\leq i\leq N-1$ we have that  $\tilde E_i^{\pi}>0$ and $\mathbb{E}[\tilde E_i^{\pi}/\mathfrak{F}_{t_i}] =1$, which implies that the measure $\mathbb{Q}^{\pi}$ given by 
			\begin{align}\label{meas1}
				\frac{\mathrm{d}\mathbb{Q}^{\pi}}{\mathrm{d}\mathbb{P}}:= \tilde E_i^{\pi}(t),\quad \tilde E_i^{\pi}(t):= \prod_{0\leq j\leq i-1} \left(1 +\delta t_j\tilde \Pi_j^{\pi} \cdot {H}_j\right),
			\end{align}
			defines a probability measure. The equivalence between the newly introduced probability measure $\mathbb{Q}^{\pi}$ and the underlying one $\mathbb{P}$ for all $\pi$, is a consequence of the following proposition. This can be derived in our general setting (without any major difficulty, provided that \eqref{611} holds) in a manner similar to the approach in \cite[Proposition 2.11]{ChaRichou16} (see also \cite[Proposition 3.3.3]{Zhou}).
			\begin{prop}\label{propB7}
				Under Assumption \ref{assum53}, the process $M_t:= \sum_{t_{i+1}\leq t} \delta t_i\tilde \Pi_i^{\pi} \cdot {H}_i$, is a $\bmo$ martingale with respect to the filtration $\mathfrak{F}_{t_i}$. Moreover, the Dol\'eans-Dade exponential $\tilde \Upsilon_i^{\pi}$ satisfies the reverse H\"older inequality.			
			\end{prop}
			
			As a by-product of the aforementioned proposition, along with equation \eqref{eq6.3}, we can derive a more manageable estimate of the difference between the backward components of the scheme \eqref{linBTZ1} and its perturbed version \eqref{linBTZ3}. This estimate will be described in detail below. A key step in obtaining this representation involves identifying a process $J_j^{\epsilon}$ such that 
			\begin{align*}
				\frac{1+\delta t_j \tilde\Gamma_j^{\pi} + \delta t_j \tilde\Pi_j^{\pi}\cdot H_j}{1+ \delta t_j \tilde\Pi_j^{\pi}\cdot H_j}= 1+ \frac{\delta t_j \tilde\Gamma_j^{\pi}}{1+ \delta t_j \tilde\Pi_j^{\pi}\cdot H_j} < J_j^{\epsilon}.
			\end{align*}
			The superscript $\epsilon$ in $J_j^{\epsilon}$ is kept here, to emphasize the dependence of the process on the constant $\epsilon$  which appears in Assumption \ref{assum53}. 
			An interesting discussion on how to construct the desired  process $J_j^{\epsilon}$ can be found in \cite[Page 67]{Zhou}, and it is readily seen that the same reasoning applies in our setting. Therefore, we choose the process $J_j^{\epsilon}$ such that $0<J^{\epsilon}_j< (1+\Lambda_y \delta t_j)^{\frac{1}{\epsilon}}$. Hence 
			\begin{align*}
				|\tilde \delta Y_i^{\pi}| &\leq \mathbb{E}\Big[\Upsilon_i|\tilde\delta Y_{N}| + \sum_{k=i}^{N-1}|\zeta_k^{Y}| \frac{\Upsilon_i}{\Upsilon_k}\Big/\mathfrak{F}_{t_i}\Big]\\
				&\leq \mathbb{E}\Big[\tilde E_i^{\pi}\prod_{j=i}^{N-1}J_j^{\epsilon}|\tilde\delta Y_{N}| + \sum_{k=i}^{N-1}|\zeta_k^{Y}| \frac{\tilde E_i^{\pi}}{\tilde E_k^{\pi}}\prod_{j=i}^{k-1}J_j^{\epsilon}\Big/\mathfrak{F}_{t_i}\Big]\\
				&\leq \exp\big(\frac{\Lambda_y h}{\epsilon}\big) \mathbb{E}\Big[\tilde E_i^{\pi}|\tilde\delta Y_{N}| + \sum_{k=i}^{N-1}|\zeta_k^{Y}| \frac{\tilde E_i^{\pi}}{\tilde E_k^{\pi}}\Big/\mathfrak{F}_{t_i}\Big].
			\end{align*}
			Using the fact that $\mathbb{E}[\tilde E_k^{\pi}/\mathfrak{F}_{t_k}] =1$ for all $i\leq k$, we finally obtain that
			\begin{align*}
				|\tilde \delta Y_i^{\pi}| 	&\leq \exp\big(\frac{\Lambda_y h}{\epsilon}\big) \mathbb{E}\Big[\tilde E_i^{\pi}\Big(|\tilde\delta Y_{N}| + \sum_{k=i}^{N-1}|\zeta_k^{Y}|\Big) \Big/\mathfrak{F}_{t_i}\Big].
			\end{align*}
			Thanks to \eqref{meas1} and applying the reverse H\"older's inequality, we deduce the following 
			\begin{lemm}\label{lemm68}
				Under Assumption \ref{assum56}, there exist constants $C_{\epsilon}>0$ and $q^{*}$ independent of $N$ such that for all 
				\begin{align}
					|\tilde \delta Y_i^{\pi}| \leq C_{\epsilon} \Big(\mathbb{E}\Big[|\tilde \delta Y_{N}^{\pi}|^{q^{*}}/\mathfrak{F}_{t_i}\Big]^{\frac{1}{q^{*}}} + \mathbb{E}\Big[\Big(\sum_{j=i}^{N-1}|\zeta_j^{Y}|\Big)^{q^{*}}\Big/\mathfrak{F}_{t_i}\Big]^{\frac{1}{q^{*}}}\Big),
				\end{align}
				where $q^{*}$ is the conjugate arising from the reverse H\"older inequality of Proposition \ref{propB7} and the constant $\epsilon > 0$.
			\end{lemm}
			We conclude this section with another important result regarding the stability of the component $Z$. We clarify the impact of the perturbation term $\zeta_i^{Y}$ on the difference between the control process components of the schemes \eqref{linBTZ1} and \eqref{linBTZ3}, respectively. 
			\begin{lemm}\label{lemm69}
				Let Assumption \ref{assum53} be in force. Then we obtain that
				\begin{align}
					\mathbb{E}\Big[\sum_{i=0}^{N-1}\delta t_i|\tilde \delta Z_i^{\pi}|^2\Big] \leq C\Big( \mathbb{E}|\tilde\delta Y_N^{\pi}|^2 + \mathbb{E} \big[\sup_{0\leq i\leq N-1}|\tilde \delta Y_i^{\pi}|^4\big]^{\frac{1}{2}} + \mathbb{E}\Big[\sum_{i=0}^{N}\tilde \delta Y_i^{\pi}(\zeta_i^{Y})\Big]\Big).
				\end{align}
			\end{lemm}
			The proof of the above lemma essentially follows the arguments in \cite[Proposition 2.12]{ChaRichou16} (see also \cite[Proposition 3.3.4]{Zhou}). For the sake of completeness, we will briefly reproduce it in the Appendix \ref{secauxres1} section below.

			\subsubsection{Time-discretization of FBSDEs with singular drifts coefficient} 
			In this subsection, we treat the numerical aspect of the FBSDE \eqref{2.1}-\eqref{2.2}.
			
			The following result pertains to the error between the BSDE \eqref{2.2} and the approximated solution \eqref{TruncY}. It can be viewed as an extension of \cite[Theorem 6]{ImkRhOliv24} to the case where the drift of the forward equation is merely bounded and Dini-continuous, with a path-dependent terminal value. We will not reproduce its proof here, as it can be carried out in a manner similar to that in \cite[Theorem 6]{ImkRhOliv24}, without any significant difficulties. In fact, one of the pivotal estimate for the proof of the theorem is given by \eqref{boundZ0}.
			\begin{thm}\label{thmconve1}
				Let  Assumption \ref{assum5.1} be in force. Let $(X,Y,Z)$ be the solution to \eqref{2.1}-\eqref{2.2} and $(X,Y^n,Z^n)$ be the solution of \eqref{2.1}-\eqref{TruncY}, $n\ge 1$. Then, for any $p > 1$ and $\kappa \geq 1$, there exists a positive finite constant $C(p,\kappa)$ such that for any $\nu\geq 1$ and any $n\in \mathbb{N}$, we have
				\begin{align*}
					\mathbb{E}\Big[ \sup_{t\in [0,T]} |Y_t^n -Y_t|^{2p} + \Big( \int_{0}^{T} |Z^n_s -Z_s|^2 \mathrm{d}s\Big)^p  \Big] \leq  {\frac{C(p,\kappa)}{n^q}}\wedge \exp(-C n^2).
			\end{align*}
		\end{thm}
		
		{\bf The perturbed explicit BTZ scheme:}
		Let us observe that the truncated solution $(Y^n,Z^n)$ to the BSDE \eqref{TruncY} can be expressed as a perturbed BTZ scheme \eqref{linBTZ3} as follows: set $\tilde{Y}_i:= Y_{t_i}^n$
		\begin{align}\label{linBTZ2}
			\begin{cases}
				\tilde Y_i^n = \mathbb{E}\Big[ 	\tilde Y_{i+1}^n + g_n(t_i,\mathcal{X}_i^{\pi},\mathbb{E}[	\tilde Y_{i+1}^n/\mathfrak{F}_{t_{i}}],\tilde Z_{i}^n)\delta t_i\big/ \mathfrak{F}_{t_{i}}  \Big] +\zeta_i^{Y},\\
				\tilde Z_{i}^n= \mathbb{E}[\tilde Y_{i+1}^nH_i^R\big/\mathfrak{F}_{t_{i}}],\\
				\tilde Y_N^n= \xi^{\pi}, \,\,\, \tilde Z_N^n= 0,
			\end{cases}
		\end{align}
		where 
		\begin{align}\label{eq712}  \zeta_i^{Y} = \mathbb{E}\Big[\int_{t_i}^{t_{i+1}} g_n(s,X_s,Y_s^n,Z_s^n)- g_n(t_i,\mathcal{X}_i^{\pi},\mathbb{E}[	\tilde Y_{i+1}^n/\mathfrak{F}_{t_{i}}],\tilde Z_{i}^n)\mathrm{d}s\big/ \mathfrak{F}_{t_{i}}  \Big].
		\end{align}
	The following result plays a key role in the proof of the main error estimates of the BTZ scheme, as it allows the application of the result derived earlier in their proof 
	\begin{lemm}\label{lemm74}
		The perturbed scheme $(\tilde{Y}_i,\tilde{Z}_i)$ solution to \eqref{linBTZ2} satisfies for all $0\leq i \leq N-1$
		\begin{align*}
			\mathbb{E}\Big[\sum_{j=i}^{N-1}|\tilde{Z}_i^n|^2\delta t_j\big/\mathfrak{F}_{t_{i}}\Big] \leq c.
		\end{align*}
	\end{lemm}
	\begin{proof}
		See Appendix \ref{secauxres1}.
	\end{proof}
	
	{\begin{remark}\label{rem75}
			We observe that the processes $\hat{Z}^n$ and $\bar{\bar{Z}}^n$, given respectively by \eqref{eq715} and \eqref{eq715}, are also $\bmo\text{-martingales}$ under the underlying probability $\mathbb{P}$. In addition, their $\bmo\text{-norms}$ depend on $\sup_{n\in\mathbb{N}}\|Z^n*W\|_{\bmo}$.
	\end{remark}}
	
	\begin{remark}\label{rem76} We carefully underline that, under Assumption \ref{assum5.1}, the scheme \eqref{linBTZ2} satisfies Assumption \ref{assum56}. Moreover, there exists a constant $\epsilon \in (0,1)$ that does not depend on $N,n$ and $R$, such that 
		\begin{align}
			h\Lambda_y + \sup_{0\leq i\leq N-1}\delta t_i| H_i^R|\ell(0)\Lambda_z \leq 1-\epsilon.
		\end{align}
	\end{remark}
	Thanks to Lemmas \ref{BTZbound} and \ref{BTZbmo}, we deduce that 
	\begin{cor} Under Assumption \ref{assum5.1} and for $N$ large enough, the following holds
		\begin{align*}
			\sup_{0\leq i\leq N}\Big(|\mathcal{Y}^{\pi}_i| + \mathbb{E}\Big[\sum_{k=i}^{N-1}|\mathcal{Z}_k^{\pi}|^2\delta t_k/\mathfrak{F}_{t_i}\Big] \Big) \leq C.
		\end{align*}
	\end{cor}
	We are now ready to prove Theorem \ref{Main63} which follows from the following propositions 
	\begin{prop}\label{Prop64} Let assumptions of Theorem \ref{Main63} be in force. Then, there exists $q^{*}>1$ such that for all $p\geq 1$
		\begin{align*}
			\mathbb{E}\Big[\sup_{0\leq i\leq N}|Y_{t_i}-\mathcal{Y}_{i}^{\pi}|^{2p}\Big]\leq \mathbb{E}\Big[|\Phi(X)-\Phi(\mathcal{X}^{\pi})|^{2pq^{*}}\Big]^{\frac{1}{q^{*}}}+ C_p\sup_{0\leq i\leq N-1}\mathbb{E}[|X_{t_i}-\mathcal{X}_i^{\pi}|^{4pq^{*}}]^{\frac{1}{2q^{*}}}+ C h^p.
		\end{align*}
	\end{prop}
	
	\begin{prop}\label{Prop77}
		Under the conditions of Theorem \ref{Main63} there exists $q^{*}>1$ such that 
		\begin{align*}
			\mathbb{E}\Big[\int_{0}^{T}|Z_{s}-\mathcal{Z}^{\pi}(s)|^2\mathrm{d}s\Big]
			\leq C\Big( \mathbb{E}\Big[|\Phi(X)-\Phi(\mathcal{X}^{\pi})|^{8q^{*}}\Big]^{\frac{1}{8q^{*}}}h^{\frac{1}{2}}+  \sup_{0\leq j\leq N-1}\mathbb{E}\Big[|X_{t_j}-\mathcal{X}_j^{\pi}|^{16q^{*}}\Big]^{\frac{1}{16q^{*}}}h^{\frac{1}{2}} +  h \Big).
		\end{align*}
	\end{prop}
	
	\subsubsection{Proof of Proposition \ref{Prop64} and Proposition \ref{Prop77}}\label{profbtz}
	\begin{proof}[Proof of Proposition \ref{Prop64}]
		We first notice that 
		\begin{align*}
			\mathbb{E}\Big[\sup_{0\leq i\leq N}|Y_{t_i}-\mathcal{Y}_{i}^{\pi}|^{2p}\Big]\leq C(p)\Big( \mathbb{E}\Big[\sup_{0\leq i\leq N}|Y_{t_i}-{Y}_{t_i}^{n}|^{2p}\Big] + \mathbb{E}\Big[\sup_{0\leq i\leq N}|{Y}_{t_i}^{n}-\mathcal{Y}_{i}^{\pi}|^{2p}\Big]\Big).
		\end{align*}
		Thanks to Theorem \ref{thmconve1} and keep in mind that $\alpha\in(0,1/2)$ and $h = n^{-\alpha}$, with $h$ small enough, we  deduce that 
		\[ \mathbb{E}\Big[\sup_{0\leq i\leq N}|Y_{t_i}-{Y}_{t_i}^{n}|^{2p}\Big] \leq C_{\alpha,p}h^p.\]
		Let us now  turn to the control between the truncated solution $Y^n$ and $\mathcal{Y}^{\pi}$. We first notice that the perturbed term $\zeta_i^{Y}$ given in \eqref{eq712}, can be rewritten as:
		\begin{align}\label{eq716}
			\zeta_i^Y: = \mathbb{E}\Big[\zeta_i^{Y,t}+\zeta_i^{Y,x}+\zeta_i^{Y,y}+\zeta_i^{Y,y,\mathbb{E}}+\zeta_i^{Y,\hat{z}}+\zeta_i^{Y,\bar{\bar{z}}}+\zeta_i^{Y,\tilde{z}}/\mathfrak{F}_{t_i}\Big],
		\end{align}
		where 
		\begin{align*}
			\zeta_i^{Y,t}:=& \int_{t_i}^{t_{i+1}}g_n(s,X_s,Y_s^n,Z_s^n)-g_n(t_i,X_s,Y_s^n,Z_s^n)\mathrm{d}s,\\
			\zeta_i^{Y,x}:=& \int_{t_i}^{t_{i+1}}g_n(t_i,X_s,Y_s^n,Z_s^n)-g_n(t_i,\mathcal{X}_i^{\pi},Y_s^n,Z_s^n)\mathrm{d}s,\\
			\zeta_i^{Y,y}:=& \int_{t_i}^{t_{i+1}}g_n(t_i,\mathcal{X}_i^{\pi},Y_s^n,Z_s^n)-g_n(t_i,\mathcal{X}_i^{\pi},Y_{t_i}^n,Z_s^n)\mathrm{d}s,\\
			\zeta_i^{Y,y,\mathbb{E}}:=& \int_{t_i}^{t_{i+1}}g_n(t_i,\mathcal{X}_i^{\pi},Y_{t_i}^n,Z_s^n)-g_n(t_i,\mathcal{X}_i^{\pi},\mathbb{E}[\tilde Y_{i+1}/\mathfrak{F}_{t_i}],Z_s^n)\mathrm{d}s,\\
			\zeta_i^{Y,\hat{z}}:=& \int_{t_i}^{t_{i+1}}g_n(t_i,\mathcal{X}_i^{\pi},\mathbb{E}[\tilde Y_{i+1}/\mathfrak{F}_{t_i}],Z_s^n)-g_n(t_i,\mathcal{X}_i^{\pi},\mathbb{E}[\tilde Y_{i+1}/\mathfrak{F}_{t_i}],\hat Z_i^n)\mathrm{d}s,\\
			\zeta_i^{Y,\bar{\bar{z}}}:=& \int_{t_i}^{t_{i+1}}g_n(t_i,\mathcal{X}_i^{\pi},\mathbb{E}[\tilde Y_{i+1}/\mathfrak{F}_{t_i}],\hat Z_i^n)-g_n(t_i,\mathcal{X}_i^{\pi},\mathbb{E}[\tilde Y_{i+1}/\mathfrak{F}_{t_i}],\bar{\bar{Z}}_i^n)\mathrm{d}s,\\
			\zeta_i^{Y,\tilde{z}}:=& \int_{t_i}^{t_{i+1}}g_n(t_i,\mathcal{X}_i^{\pi},\mathbb{E}[\tilde Y_{i+1}/\mathfrak{F}_{t_i}],\bar{\bar{Z}}_i^n)-g_n(t_i,\mathcal{X}_i^{\pi},\mathbb{E}[\tilde Y_{i+1}/\mathfrak{F}_{t_i}],\tilde{Z}_i^n)\mathrm{d}s.
		\end{align*}
		Applying Lemma \ref{lemm68} and \cite[Remark 2.5]{ChaRichou16} to the terms $\zeta_i^{Y,x},\zeta_i^{Y,y},\zeta_i^{Y,\bar{\bar{z}}}$ and $\zeta_i^{Y,\tilde{z}}$, we deduce that
		\begin{align*}
			&|Y_{t_i}^n-\mathcal{Y}_i^{\pi}|\\
			&\leq C \mathbb{E}\left[|Y_{t_N}^n-\mathcal{Y}_N^{\pi}|^{q^{*}}/\mathfrak{F}_{t_i}\right]^{\frac{1}{q^{*}}} + C\mathbb{E}\Big[\Big(\sum_{j=i}^{N-1}\zeta_j^{Y,x}\Big)^{q^{*}}/\mathfrak{F}_{t_i}\Big]^{\frac{1}{q^{*}}} + C\mathbb{E}\Big[\Big(\sum_{j=i}^{N-1}\zeta_j^{Y,y}\Big)^{q^{*}}/\mathfrak{F}_{t_i}\Big]^{\frac{1}{q^{*}}}\\
			& \quad +C\mathbb{E}\Big[\Big(\sum_{j=i}^{N-1}\zeta_j^{Y,\bar{\bar{z}}}\Big)^{q^{*}}/\mathfrak{F}_{t_i}\Big]^{\frac{1}{q^{*}}}+ C\mathbb{E}\Big[\Big(\sum_{j=i}^{N-1}\zeta_j^{Y,\tilde{z}}\Big)^{q^{*}}/\mathfrak{F}_{t_i}\Big]^{\frac{1}{q^{*}}} + C\mathbb{E}^{\mathbb{Q}^{\pi}}\Big[\sum_{j=i}^{N-1}\big|\zeta_j^{Y,t}\big|/\mathfrak{F}_{t_i}\Big]\\
			&\quad + C\mathbb{E}^{\mathbb{Q}^{\pi}}\Big[\sum_{j=i}^{N-1}\big|\zeta_j^{Y,\hat{z}}\big|/\mathfrak{F}_{t_i}\Big] + C\mathbb{E}^{\mathbb{Q}^{\pi}}\Big[\sum_{j=i}^{N-1}\big|\zeta_j^{Y,y,\mathbb{E}}\big|/\mathfrak{F}_{t_i}\Big].
		\end{align*}
		Taking the $2p$-th power, the supremum on both sides, and applying the Jensen and Doob's maximal inequalities, we obtain for all $p\geq 2$
		\begin{align*}
			&\mathbb{E}\Big[\sup_{0\leq i\leq N}|Y_{t_i}^n-\mathcal{Y}_i^{\pi}|^{2p}\Big]\\
			&\leq C_p \mathbb{E}\left[|Y_{t_N}^n-\mathcal{Y}_N^{\pi}|^{2pq^{*}}\right]^{\frac{1}{q^{*}}} + C_p\mathbb{E}\Big[\Big(\sum_{j=0}^{N-1}\big|\zeta_j^{Y,x}\big|\Big)^{2pq^{*}}\Big]^{\frac{1}{q^{*}}} + C_p\mathbb{E}\Big[\Big(\sum_{j=0}^{N-1}\big|\zeta_j^{Y,y}\big|\Big)^{2pq^{*}}\Big]^{\frac{1}{q^{*}}}\\
			& \quad +C_p\mathbb{E}\Big[\Big(\sum_{j=0}^{N-1}\big|\zeta_j^{Y,\bar{\bar{z}}}\big|\Big)^{2pq^{*}}\Big]^{\frac{1}{q^{*}}}+ C_p\mathbb{E}\Big[\Big(\sum_{j=0}^{N-1}\big|\zeta_j^{Y,\tilde{z}}\big|\Big)^{2pq^{*}}\Big]^{\frac{1}{q^{*}}} + C_p\mathbb{E}\Big[\sup_{0\leq i\leq N}\mathbb{E}^{\mathbb{Q}^{\pi}}\Big[\sum_{j=i}^{N-1}\big|\zeta_j^{Y,t}\big|/\mathfrak{F}_{t_i}\Big]^{2p}\Big]\\
			&\quad + C_p\mathbb{E}\Big[\sup_{0\leq i\leq N}\mathbb{E}^{\mathbb{Q}^{\pi}}\Big[\sum_{j=i}^{N-1}\big|\zeta_j^{Y,\hat{z}}\big|/\mathfrak{F}_{t_i}\Big]^{2p}\Big] + C_p\mathbb{E}\Big[\sup_{0\leq i\leq N}\mathbb{E}^{\mathbb{Q}^{\pi}}\Big[\sum_{j=i}^{N-1}\big|\zeta_j^{Y,y,\mathbb{E}}\big|/\mathfrak{F}_{t_i}\Big]^{2p}\Big].
		\end{align*}
		Let us analyse separately the different terms that appear on the right side of the above inequality. By definition, the first term corresponds to the error analysis of the terminal value. To estimate the second term, we primarily use the fact that the generator is stochastic Lipschitz continuous in $x$ and apply H\"older's inequality to deduce that
		\begin{align*}
			\mathbb{E}\Big(\sum_{j=0}^{N-1}\big|\zeta_j^{Y,x}\big|\Big)^{2pq^{*}}&\leq C_p\mathbb{E}\Big[\Big(\sum_{j=0}^{N-1}\int_{t_j}^{t_{j+1}}(1+\ell(|Y_s^n|)|Z_s^n|^{\alpha})|X_s-\mathcal{X}_j^{\pi}|\mathrm{d}s\Big)^{2pq^{*}}\Big]\\
			&\leq C_p \mathbb{E}\Big[\Big(\sum_{j=0}^{N-1}\sup_{t_j\leq s\leq t_{j+1}}(1+\ell(|Y_s^n|)|Z_s^n|^{\alpha})|X_s-\mathcal{X}_j^{\pi}|\delta t_j\Big)^{2pq^{*}}\Big]\\
			&\leq C_p \mathbb{E}\Big[\sum_{j=0}^{N-1}\sup_{t_j\leq s\leq t_{j+1}}(1+\ell(|Y_s^n|)|Z_s^n|^{\alpha})^{2pq^{*}}|X_s-\mathcal{X}_j^{\pi}|^{2pq^{*}}\delta t_j\Big]\Big(\sum_{j=0}^{N-1}\delta t_j\Big)^{2pq^{*}-1}\\
			&\leq C_p T^{2pq^{*}} \sup_{0\leq j \leq N-1}\mathbb{E}\Big[\sup_{t_j\leq s\leq t_{j+1}}(1+\ell(|Y_s^n|)|Z_s^n|^{\alpha})^{2pq^{*}}|X_s-\mathcal{X}_j^{\pi}|^{2pq^{*}}\Big]\\
			&\leq C_p T^{2pq^{*}} \sup_{0\leq j \leq N-1}\mathbb{E}\Big[\sup_{t_j\leq s\leq t_{j+1}}(1+\ell(|Y_s^n|)|Z_s^n|^{\alpha})^{4pq^{*}}\Big]^{\frac{1}{2}}\mathbb{E}\Big[|X_s-\mathcal{X}_j^{\pi}|^{4pq^{*}}\Big]^{\frac{1}{2}}.
		\end{align*}
		From \eqref{boundZ0} and the uniform bound of $Y^n$ we deduce that
		\begin{align*}
			\mathbb{E}\Big(\sum_{j=0}^{N-1}\big|\zeta_j^{Y,x}\big|\Big)^{2pq^{*}}&\leq C(p,q^{*})\Big(\sup_{0\leq j \leq N-1}\mathbb{E}\Big[\sup_{t_j\leq s\leq t_{j+1}}|X_s-{X}_{t_j}|^{4pq^{*}}\Big]^{\frac{1}{2}} + \sup_{0\leq j \leq N-1}\mathbb{E}\Big[|X_{t_j}-\mathcal{X}_j^{\pi}|^{4pq^{*}}\Big]^{\frac{1}{2}}\Big).
		\end{align*}
		Under our assumptions, we observe that
		\begin{align*}
			&\sup_{0\leq j \leq N-1}\mathbb{E}\Big[\sup_{t_j\leq s\leq t_{j+1}}|X_s-{X}_{t_j}|^{2pq^{*}}\Big]\\ &\leq C_{pq^{*}} \sup_{0\leq j \leq N-1}\Big( \mathbb{E} \sup_{t_j\leq s\leq t_{j+1}} \Big|\int_{t_j}^{s}b(v,X_v)\mathrm{d}v\Big|^{4pq^{*}} + \mathbb{E} \sup_{t_j\leq s\leq t_{j+1}} |W_s-W_{t_j}|^{4pq^{*}} \Big)\\
			&\leq C_{pq^{*}} h^{2pq^{*}}.
		\end{align*}
		Therefore, there exists a constant $C_{pq^{*}}$ not depending on $h$, such that
		\begin{align*}
			\mathbb{E}\Big[\Big(\sum_{j=0}^{N-1}\big|\zeta_j^{Y,x}\big|\Big)^{2pq^{*}}\Big]^{\frac{1}{q^{*}}} \leq C_{pq^{*}} h^p + \sup_{0\leq j \leq N-1}\mathbb{E}\Big[|X_{t_j}-\mathcal{X}_j^{\pi}|^{4pq^{*}}\Big]^{\frac{1}{2q^{*}}}.
		\end{align*}
		Similarly, from the growth of $g_n$ and applying H\"older's inequality, we deduce that
		\begin{align*}
			\mathbb{E}\Big[\Big(\sum_{j=0}^{N-1}\big|\zeta_j^{Y,y}\big|\Big)^{2pq^{*}}\Big]^{\frac{1}{q^{*}}} &\leq C_p\mathbb{E}\Big[\Big(\sum_{j=0}^{N-1}\int_{t_j}^{t_{j+1}}|Y_s^n-{Y}_{t_j}^{n}|\mathrm{d}s\Big)^{2pq^{*}}\Big]^{\frac{1}{q^{*}}} \\
			&\leq C_p\mathbb{E}\Big[\sum_{j=0}^{N-1}\sup_{t_j\leq s\leq t_{j+1}}|Y_s^n-{Y}_{t_j}^{n}|^{2pq^{*}}\delta t_j\Big]^{\frac{1}{q^{*}}}\Big(\sum_{j=0}^{N-1}\delta t_j\Big)^{\frac{2pq^{*}-1}{q^{*}}}\\
			&\leq C_{pq^{*}} T^{2p}\Big(\sup_{0\leq j \leq N-1}\mathbb{E}\Big[\sup_{t_j\leq s\leq t_{j+1}}|Y_s^n-{Y}_{t_j}^{n}|^{2pq^{*}}\Big] \Big)^{\frac{1}{q^{*}}} \\
			&\leq C_{pq^{*}} h^{p},
		\end{align*}
		where, we used Theorem \ref{thm1.1} to obtain the last inequality.
		From our assumptions on the driver, we obtain that 
		\begin{align}
			&\mathbb{E}\Big[\Big(\sum_{j=0}^{N-1}\big|\zeta_j^{Y,\bar{\bar{z}}}\big|\Big)^{2pq^{*}}\Big]^{\frac{1}{q^{*}}}\notag\\ 
			\leq & C\mathbb{E}\Big[\Big(\sum_{j=0}^{N-1}\int_{t_j}^{t_{j+1}}\big(1+\ell(0)(|\hat{Z}^n(s)|+|\bar{\bar{Z}}^n(s)|)\big)|\hat{Z}^n(s)-\bar{\bar{Z}}^n(s)|\mathrm{d}s\Big)^{2pq^{*}}\Big]^{\frac{1}{q^{*}}}\notag	\\
			\leq & C\mathbb{E}\Big[\Big(\sum_{j=0}^{N-1}\int_{t_j}^{t_{j+1}}\big(1+\ell(0)(|\hat{Z}^n(s)|+|\bar{\bar{Z}}^n(s)|)\big)^2\mathrm{d}s\Big)^{2pq^{*}}\Big]^{\frac{1}{2q^{*}}}\notag\\ &\times\mathbb{E}\Big[\Big(\sum_{j=0}^{N-1}\int_{t_j}^{t_{j+1}}|\hat{Z}^n(s)-\bar{\bar{Z}}^n(s)|^2\mathrm{d}s\Big)^{2pq^{*}}\Big]^{\frac{1}{2q^{*}}},\label{eq717}
		\end{align}
		where we used the H\"older's inequality and recall that $\hat{Z}^n(s)$ and $\bar{\bar{Z}}^n(s)$ are respectively given by \eqref{eq714} and \eqref{eq715}. 
		We first observe that, by performing some elementary computations and then applying the H\"older's inequality, we obtain 
		\begin{align}\label{721}
			|\bar{\bar{Z}}^n(s) -\hat{Z}^n(s)| \leq & \mathbb{E}\Big[\Big| \int_{t_j}^{t_{j+1}}g_n(s,X_s,Y_s^n,Z_s^n)\mathrm{d}s\frac{W_{t_{j+1}}-W_{t_j}}{\delta t_j}\Big|\big/\mathfrak{F}_{t_j}\Big]\notag\\
			\leq & \mathbb{E}\Big[ \Big|\int_{t_j}^{t_{j+1}}g_n(s,X_s,Y_s^n,Z_s^n)\mathrm{d}s\Big|^2/\mathfrak{F}_{t_j}\Big]^{\frac{1}{2}}\mathbb{E}\Big[ \Big|\frac{W_{t_{j+1}}-W_{t_j}}{\delta t_j}\Big|^2\big/\mathfrak{F}_{t_j}\Big]^{\frac{1}{2}}\notag\\  
			\leq & \frac{C}{\sqrt{\delta t_j}}\mathbb{E}\Big[ \Big(\int_{t_j}^{t_{j+1}}1+|Z_s^n|^2\mathrm{d}s\Big)^2/\mathfrak{F}_{t_j}\Big]^{\frac{1}{2}}\notag\\
			\leq & C\sqrt{\delta t_j}+ \frac{C}{\sqrt{\delta t_j}}\mathbb{E}\Big[\Big(\int_{t_j}^{t_{j+1}}|Z_s^n|^2\mathrm{d}s\Big)^2\big/\mathfrak{F}_{t_j}\Big]^{\frac{1}{2}},
		\end{align}
		where, we used the growth of $g_n$, the bound of $Y^n$ and applied the conditional BDG inequality. Hence, squaring the above inequality, we deduce that
		\begin{align*}
			|\bar{\bar{Z}}^n(s) -\hat{Z}^n(s)|^2\leq C{\delta t_j} +\frac{C}{{\delta t_j}}\mathbb{E}\Big[\Big(\int_{t_j}^{t_{j+1}}|Z_s^n|^2\mathrm{d}s\Big)^2\big/\mathfrak{F}_{t_j}\Big].
		\end{align*}
		Hence, by applying successively the H\"older and Jensen inequalities, we obtain 
		\begin{align*}
			&\mathbb{E}\Big(\sum_{j=0}^{N-1}\int_{t_j}^{t_{j+1}}|\hat{Z}^n(s)-\bar{\bar{Z}}^n(s)|^2\mathrm{d}s\Big)^{2pq^{*}}\\
			&\leq C(p,q^{*},T)h^{2pq^{*}}+ C(p,q^{*})\mathbb{E}\Big[\sum_{j=0}^{N-1}\frac{1}{\delta t_j^{2pq^{*}}}\Big(\mathbb{E}\Big[ \Big(\int_{t_j}^{t_{j+1}}|Z_s^n|^2\mathrm{d}s\Big)^{4pq^{*}}\big/\mathfrak{F}_{t_j}\Big]\Big)\delta t_j\Big]\Big(\sum_{j=0}^{N-1}\delta t_j\Big)^{2pq^{*}-1}.
		\end{align*}
		This leads to
		\begin{align*}
			&\mathbb{E}\Big[\sum_{j=0}^{N-1}\int_{t_j}^{t_{j+1}}|\hat{Z}^n(s)-\bar{\bar{Z}}^n(s)|^2\mathrm{d}s\Big]^{2pq^{*}}\\
			&\leq C(p,q^{*},T)h^{2pq^{*}}+ C(p,q^{*})T^{2pq^{*}-1}\sum_{j=0}^{N-1}\frac{\delta t_j}{\delta t_j^{2pq^{*}}}\mathbb{E}\Big[\Big(\int_{t_j}^{t_{j+1}}|Z_s^n|^2\mathrm{d}s\Big)^{4pq^{*}}\Big] \\
			&\leq C(p,q^{*},T)h^{2pq^{*}}
			+ C(p,q^{*})T^{2pq^{*}-1}\sum_{j=0}^{N-1}\frac{\delta t_j }{\delta t_j^{2pq^{*}}}\mathbb{E}\Big[\sup_{t_j\leq s\leq t_{j+1}}|Z_s^n|^{4pq^{*}}(\delta t_j)^{4pq^{*}}\Big].
		\end{align*}
		From \eqref{boundZ0}, there exists a constant $C(p,q^{*},T)>0$ that does not depend on $N$, such that
		\begin{align}\label{722}
			\mathbb{E}\Big[\Big(\sum_{j=0}^{N-1}\int_{t_j}^{t_{j+1}}|\hat{Z}^n(s)-\bar{\bar{Z}}^n(s)|^2\mathrm{d}s\Big)^{2pq^{*}}\Big]^{\frac{1}{2q^{*}}}\leq C(p,q^{*},T)h^{p}.
		\end{align}
		On the other hand, we have seen that the processes $\bar{\bar{Z}}^n_j$ and $\hat{Z}^n_j$ belong to $\mathcal{H}_{BMO}$ (see Remark \ref{rem75}). Therefore, uniformly in $n$, we obtain
		\begin{align*}
			&\mathbb{E}\Big[\Big(\sum_{j=0}^{N-1}\int_{t_j}^{t_{j+1}}\big(1+\ell(0)(|\hat{Z}^n(s)|+|\bar{\bar{Z}}^n(s)|)\big)^2\mathrm{d}s\Big)^{2pq^{*}}\Big]\\
			&\leq C(p,q^{*},T) \|(1+\ell(0)(|\hat{Z}^n(s)|+|\bar{\bar{Z}}^n(s)|)\big)\|^{4pq^{*}}_{\mathcal{H}_{BMO}}.
		\end{align*}
		
		Similarly, one can also obtain that 
		\begin{align}
			\mathbb{E}\Big[\Big(\sum_{j=0}^{N-1}\big|\zeta_j^{Y,{\tilde{z}}}\big|\Big)^{2pq^{*}}\Big]^{\frac{1}{q^{*}}} 
			\leq C(p,q^{*},T)h^p.
		\end{align}
		In fact, from the growth of $g_n$ and by applying H\"older's inequality, we deduce that
		\begin{align*}
			&\mathbb{E}\Big[\Big(\sum_{j=0}^{N-1}\big|\zeta_j^{Y,\tilde{z}}\big|\Big)^{2pq^{*}}\Big]^{\frac{1}{q^{*}}}\notag\\ &\leq C\mathbb{E}\Big[\Big(\sum_{j=0}^{N-1}\int_{t_j}^{t_{j+1}}\big(1+\ell(0)(|\tilde{Z}^n(s)|+|\bar{\bar{Z}}^n(s)|)\big)|\tilde{Z}^n(s)-\bar{\bar{Z}}^n(s)|\mathrm{d}s\Big)^{2pq^{*}}\Big]^{\frac{1}{q^{*}}}\notag	\\
			&\leq C(p,q^{*},T) \|(1+\ell(0)(|\hat{Z}^n(s)|+|\bar{\bar{Z}}^n(s)|)\big)\|^{4pq^{*}}_{\mathcal{H}_{BMO}} \mathbb{E}\Big[\Big(\sum_{j=0}^{N-1}\int_{t_j}^{t_{j+1}}|\tilde{Z}^n(s)-\bar{\bar{Z}}^n(s)|^2\mathrm{d}s\Big)^{2pq^{*}}\Big]^{\frac{1}{2q^{*}}},
		\end{align*}
		where, $\tilde{Z}^n(s):=\tilde{Z}^n_j$ for all $s\in [t_j,t_{j+1}]$, $\tilde{Z}_j^n$ is given by $\eqref{linBTZ2}$ and belongs to $\mathcal{H}_{BMO}$, thanks to Lemma \ref{lemm74}. 
		On the other hand, from the definitions of the processes $\tilde{Z}^n$ and $\bar{\bar{Z}}^n$, respectively, we obtain that 
		\begin{align*}
			&|\tilde{Z}_j^n(s)-\bar{\bar{Z}}_j^n(s)|^2\\ =& \mathbb{E}\Big[(Y^n_{t_{j+1}}-Y^n_{t_j})\Big(H_i^R-\frac{W_{t_{j+1}}-W_{t_j}}{\delta t_j}\Big)\big/\mathfrak{F}_{t_j}\Big]^2\\
			\leq & \mathbb{E}\Big[|Y^n_{t_{j+1}}-Y^n_{t_j}|^2/\mathfrak{F}_{t_j}\Big]\mathbb{E}\Big[\Big(H_i^R-\frac{W_{t_{j+1}}-W_{t_j}}{\delta t_j}\Big)^2\big/\mathfrak{F}_{t_j}\Big]\\
			\leq & C \mathbb{E}\Big[ \Big( \int_{t_j}^{t_{j+1}}|g_n(s,X_s,Y_s^n,Z_s^n)|\mathrm{d}s\Big)^2+\int_{t_j}^{t_{j+1}}|Z_s^n|^2\mathrm{d}s \big/\mathfrak{F}_{t_j}\Big] \mathbb{E}\Big[\Big(H_i^R-\frac{W_{t_{j+1}}-W_{t_j}}{\delta t_j}\Big)^2\big/\mathfrak{F}_{t_j}\Big] \\
			\leq & C \mathbb{E}\Big[ (\delta t_j)^2+\Big( \int_{t_j}^{t_{j+1}}|Z_s^n|^2\mathrm{d}s\Big)^2+\int_{t_j}^{t_{j+1}}|Z_s^n|^2\mathrm{d}s \big/\mathfrak{F}_{t_j}\Big] \mathbb{E}\Big[\Big(H_i^R-\frac{W_{t_{j+1}}-W_{t_j}}{\delta t_j}\Big)^2\big/\mathfrak{F}_{t_j}\Big],
		\end{align*}
		where, we applied the H\"older's inequality,  and used the growth of $g_n$ and the uniform boundedness of $Y^n$ to derive the above inequalities.
		Using the definition of  the coefficients $H^R_i$ given by \eqref{eq79}, we observe that
		\begin{align*}
			\mathbb{E}\Big[\Big(H_j^R-\frac{W_{t_{j+1}}-W_{t_j}}{\delta t_j}\Big)^2\big/\mathfrak{F}_{t_j}\Big] \leq \frac{R^2}{\delta t_j}.
		\end{align*}
		Therefore, there exists a constant $C>0$ independent of $n$ such that
		\begin{align*}
			|\tilde{Z}_j^n(s)-\bar{\bar{Z}}_j^n(s)|^2 &\leq C \mathbb{E}\Big[ \delta t_j+\frac{1}{\delta t_j}\Big( \int_{t_j}^{t_{j+1}}|Z_s^n|^2\mathrm{d}s\Big)^2+ \frac{1}{\delta t_j}\int_{t_j}^{t_{j+1}}|Z_s^n|^2\mathrm{d}s \big/\mathfrak{F}_{t_j}\Big].
		\end{align*}
		Hence, applying the Jensen and H\"older's inequalities, we obtain
		\begin{align*}
			&\mathbb{E}\Big[\Big(\sum_{j=0}^{N-1}\int_{t_j}^{t_{j+1}}|\tilde{Z}^n(s)-\bar{\bar{Z}}^n(s)|^2\mathrm{d}s\Big)^{2pq^{*}}\Big]\\
			\leq &C  \sum_{j=0}^{N-1}\delta t_j\Big[ (\delta t_j)^{2pq^{*}}+\frac{1}{(\delta t_j)^{2pq^{*}}}\mathbb{E}\Big( \int_{t_j}^{t_{j+1}}|Z_s^n|^2\mathrm{d}s\Big)^{4pq^{*}}+ \frac{1}{(\delta t_j)^{2pq^{*}}}\mathbb{E}\Big(\int_{t_j}^{t_{j+1}}|Z_s^n|^2\mathrm{d}s\Big)^{2pq^{*}}\Big]\Big(\sum_{j=0}^{N-1}\delta t_j\Big)^{2pq^{*}-1}.
		\end{align*}
		Thanks to \eqref{boundZ0}, we deduce that 
		\begin{align}\label{724}
			\mathbb{E}\Big[\Big(\sum_{j=0}^{N-1}\int_{t_j}^{t_{j+1}}|\tilde{Z}^n(s)-\bar{\bar{Z}}^n(s)|^2\mathrm{d}s\Big)^{2pq^{*}}\Big]
			\leq C T^{2pq^{*}}\big(h+ h^{2pq^{*}}\big).
		\end{align}	
		The following bound is a consequence the regularity of the generator on its time variable
		\begin{align}
			\mathbb{E}\Big[\sup_{0\leq i\leq N}\mathbb{E}^{\mathbb{Q}^{\pi}}\Big[\sum_{j=i}^{N-1}\big|\zeta_j^{Y,t}\big|/\mathfrak{F}_{t_i}\Big]^{2p}\Big]\leq  \sup_{0\leq i\leq N}\Big(\sum_{j=i}^{N-1}\Lambda_t(\delta t_j)^{3/2}\Big)^{2p}\leq C(p,T)h^{p}.
		\end{align}
		Let us now turn to establishing the bounds for the remaining two terms. Using the local Lipschitz assumption of the driver $g$, we first observe that
		\begin{align*}
			&\mathbb{E}^{\mathbb{Q}^{\pi}}\Big[\sum_{j=i}^{N-1}\big|\zeta_j^{Y,\hat{z}}\big|/\mathfrak{F}_{t_i}\Big] \\
			&\leq 	\mathbb{E}^{\mathbb{Q}^{\pi}}\Big[\sum_{j=i}^{N-1} \int_{t_j}^{t_{j+1}}(1+\ell(0)(|\hat{Z}^n(s)|+|Z_s^n|))|\hat{Z}^n(s)-{Z}^n_s|\mathrm{d}s\big/\mathfrak{F}_{t_i}\Big]\\
			&\leq C\|(1+\ell(0)(|\hat{Z}^n(s)|+|Z_s^n|))\|_{\mathcal{H}_{BMO}}\mathbb{E}^{\mathbb{Q}^{\pi}}\Big[\sum_{j=i}^{N-1} \int_{t_j}^{t_{j+1}}|\hat{Z}^n(s)-{Z}^n_s|^2\mathrm{d}s\big/\mathfrak{F}_{t_i}\Big]^{\frac{1}{2}}.
		\end{align*}
		By applying Corollary \ref{cor 1.3}, we deduce that 
		\begin{align*}
			\mathbb{E}\Big[\sup_{0\leq i\leq N}\mathbb{E}^{\mathbb{Q}^{\pi}}\Big[\sum_{j=i}^{N-1}\big|\zeta_j^{Y,\hat{z}}\big|/\mathfrak{F}_{t_i}\Big]^{2p}\Big] &\leq C \mathbb{E}\Big[\sup_{0\leq i\leq N}\mathbb{E}^{\mathbb{Q}^{\pi}}\Big[\sum_{j=i}^{N-1} \int_{t_j}^{t_{j+1}}|\hat{Z}^n(s)-{Z}^n_s|^2\mathrm{d}s\big/\mathfrak{F}_{t_i}\Big]^{p}\Big]\\
			&\leq C h^p.	
		\end{align*}
		Finally, using the growth of $g_n$ once more, we obtain that
		\begin{align*}
			\mathbb{E}^{\mathbb{Q}^{\pi}}\Big[\sum_{j=i}^{N-1}\big|\zeta_j^{Y,y,\mathbb{E}}\big|\big/\mathfrak{F}_{t_i}\Big] &\leq \Lambda_y 	\mathbb{E}^{\mathbb{Q}^{\pi}} \Big[\sum_{j=i}^{N-1} \int_{t_j}^{t_{j+1}} |Y_{t_j}- \mathbb{E}[\tilde{Y}_{j+1}/\mathfrak{F}_{t_j}]| \mathrm{d}s\big/\mathfrak{F}_{t_i}\Big]\\
			&\leq \Lambda_y \mathbb{E}^{\mathbb{Q}^{\pi}} \Big[\sum_{j=i}^{N-1}\delta t_j \mathbb{E}\Big[\int_{t_j}^{t_{j+1}} |g_n(s,X_s,Y_s^n,Z_s^n)|\mathrm{d}s/\mathfrak{F}_{t_j}\Big]	 \big/\mathfrak{F}_{t_i}\Big]\\
			&\leq C \mathbb{E}^{\mathbb{Q}^{\pi}} \Big[ \int_{t_i}^{T} (1+\ell(|Y_s^n|)|Z_s^n|^2)\mathrm{d}s \big/\mathfrak{F}_{t_i}\Big]h\\
			&\leq C(\epsilon,\|Z^n*W\|_{\bmo}) h.
		\end{align*}
		Hence
		\begin{align*}
			\mathbb{E}\Big[\sup_{0\leq i\leq N}\mathbb{E}^{\mathbb{Q}^{\pi}}\Big[\sum_{j=i}^{N-1}\big|\zeta_j^{Y,y,\mathbb{E}}\big|\big/\mathfrak{F}_{t_i}\Big]^{2p}\Big] \leq C(\epsilon,p,\|Z^n*W\|_{\bmo}) h^{2p}.
		\end{align*}

	\end{proof}
	
	\begin{proof}[Proof of Proposition \ref{Prop77}]
		We first observe that 
		\begin{align*}
			|Z_{s}-\mathcal{Z}^{\pi}(s)| \leq |Z_{s}-Z_s^n|+|Z_s^n-\hat{Z}_s^n|+ |\hat{Z}_s^n-\bar{\bar{Z}}^n(s)|+ |\bar{\bar{Z}}^n(s)-\tilde{Z}^n(s)| + |\tilde{Z}^n(s)-\mathcal{Z}^{\pi}(s)|.
		\end{align*}
		Hence, after performing the necessary computations, we obtain that
		\begin{align*}
			\mathbb{E}\Big[\sum_{j=0}^{N-1}\int_{t_j}^{t_{j+1}}|Z_{s}-\mathcal{Z}^{\pi}(s)|^2\mathrm{d}s\Big] \leq C \left( II_1 + II_2 + II_3 +  II_4 + II_5\right),
		\end{align*}
		where, each term $(II_l)_{1\leq l\leq  5}$ can be deduced easily. It follows from Theorem \ref{thmconve1} that $II_1\leq C(T){h}$. Observe that
		\begin{align*}
			|Z_s^n-\hat{Z}_s^n|^2 
			&\leq 2  |Z_s^n-{Z}_{t_i}^n|^2+ 2 \Big|{Z}_{t_i}^n-\frac{1}{\delta t_i}\mathbb{E}\Big[\int_{t_i}^{t_{i+1}}{Z}_v^n\mathrm{d}v/\mathfrak{F}_{t_i}\Big]\Big|^2\\
			&\leq  2  |Z_s^n-{Z}_{t_i}^n|^2 + \frac{2}{\delta t_i}\mathbb{E}\Big[\int_{t_i}^{t_{i+1}}|{Z}_v^n-{Z}_{t_i}^n|^2\mathrm{d}v/\mathfrak{F}_{t_i}\Big].
		\end{align*}
		Thanks to  Theorem \ref{thm1.1}, we deduce that
		$$II_2 \leq 4 \mathbb{E}\Big[\sum_{j=0}^{N-1}\int_{t_j}^{t_{j+1}}|Z_s^n-\hat{Z}_s^n|^2\mathrm{d}s\Big] \leq C h.$$
		Similar to \eqref{722}, there exists a constant $C>0$, such that $II_3 \leq C h$.	Moreover, from \eqref{724}, we also obtain that $II_4 \leq C(1+h)$.
		Let us turn now to the bound for $II_5$, whose proof mainly relies on the application of	Lemma \ref{lemm69}. Indeed, recall that the expression for $\zeta_j^{Y}$ is given by \eqref{eq716}, and applying the result from Proposition \ref{Prop64}, we obtain
		\begin{align*}
			&II_5 = \mathbb{E}\Big[\sum_{j=0}^{N-1}\delta t_j|\tilde{Z}^n(s)-\mathcal{Z}^{\pi}(s)|^2\Big]\\
			&\leq C \mathbb{E}|\Phi(X)-\Phi(\mathcal{X}^{\pi})|^2 + \mathbb{E}\Big[\sup_{0\leq j\leq N-1}|\tilde{Y}_{t_j}^n-\mathcal{Y}_j^{\pi}|^4\Big]^{\frac{1}{2}}+ \mathbb{E}\Big[\sum_{j=0}^{N-1}(\tilde{Y}_{t_j}^n-\mathcal{Y}_j^{\pi})\zeta_j^{Y}\Big]\\
			&\leq \mathbb{E}\Big[|\Phi(X)-\Phi(\mathcal{X}^{\pi})|^{4q^{*}}\Big]^{\frac{1}{2q^{*}}}+ C_p\sup_{0\leq j\leq N-1}\mathbb{E}[|X_{t_j}-\mathcal{X}_j^{\pi}|^{8q^{*}}]^{\frac{1}{4q^{*}}}+ C h+ \mathbb{E}\Big[\sum_{j=0}^{N-1}(\tilde{Y}_{t_j}^n-\mathcal{Y}_j^{\pi})\zeta_j^{Y}\Big],
		\end{align*}
		The last term on right side of the above inequality can be estimated as follows 
		\begin{align*}
			&\mathbb{E}\Big[\sum_{j=0}^{N-1}(\tilde{Y}_{t_j}^n-\mathcal{Y}_j^{\pi})\zeta_j^{Y}\Big] \\&\leq 
			\mathbb{E}\Big[\Big|\sum_{j=0}^{N-1}|(\tilde{Y}_{t_j}^n-\mathcal{Y}_j^{\pi})\mathbb{E}\Big[\zeta_j^{Y,t}+\zeta_j^{Y,x}+\zeta_j^{Y,y}+\zeta_j^{Y,y,\mathbb{E}}+\zeta_i^{Y,\hat{z}}+\zeta_i^{Y,\bar{\bar{z}}}+\zeta_i^{Y,\tilde{z}}/\mathfrak{F}_{t_i}\Big]\Big|\Big]\\
			&\leq C \mathbb{E}\Big[\sup_{j} |\tilde{Y}_{t_j}^n-\mathcal{Y}_j^{\pi}|^2\Big]^{\frac{1}{2}}\mathbb{E}\Big[\Big(\sum_{j=0}^{N-1}\big|\zeta_j^{Y,t}+\zeta_j^{Y,x}+\zeta_j^{Y,y}+\zeta_j^{Y,y,\mathbb{E}}+\zeta_j^{Y,\tilde{z}}+\zeta_j^{Y,\bar{\bar{z}}}\big|\Big)^2\Big]^{\frac{1}{2}}\\
			&\quad+ \mathbb{E}\Big| \sum_{j=0}^{N-1}(\tilde{Y}_{t_j}^n-\mathcal{Y}_j^{\pi})\zeta_j^{Y,\hat{z}}\Big|,
		\end{align*}
		where we used the H\"older's inequality to obtain the last inequality.  We notice that, the second term on the right-hand side of the latter inequality can be controlled in the same manner as done, for instance, in the proof of the previous Proposition. It can be shown that all such terms are bounded by $Ch^{\frac{1}{2}}$, where the constant $C>0$ does not depend on $n$. We now turn to deriving the estimate of the third one. Indeed, using the growth of the driver $g$ and the local boundedness of the function $\ell$, we obtain the existence of a constant $C>0$ such that
		\begin{align*}
			&\mathbb{E}\Big[\Big| \sum_{j=0}^{N-1}(\tilde{Y}_{t_j}^n-\mathcal{Y}_j^{\pi})\zeta_j^{Y,\hat{z}}\Big|\Big]\\
			& \leq C \mathbb{E}\Big[ \sum_{j=0}^{N-1}\int_{t_j}^{t_{j+1}}\big|\tilde{Y}^n(s)-\mathcal{Y}^{\pi}(s)\big|\big(1+\ell(|Y_s^n|)(|{Z}_s^n|+|\hat{Z}_s^n|)\big)|{Z}_s^n-\hat{Z}_s^n|\mathrm{d}s\Big]\\
			&\leq C \mathbb{E}\Big[\sup_{0\leq i\leq N-1}\mathbb{E}\Big[\int_{t_i}^T|\tilde{Y}^n(s)-\mathcal{Y}^{\pi}(s)|\big(1+\ell(|Y_s^n|)(|{Z}_s^n|+|\hat{Z}_s^n|)\big)|{Z}_s^n-\hat{Z}_s^n|\mathrm{d}s\big/\mathfrak{F}_{t_i}\Big]\Big].
		\end{align*}
		We recall that the processes ${Z}_s^n$ and $\hat{Z}_s^n$ are both elements of $\mathcal{H}_{BMO}$ and applying the H\"older's inequality, we deduce that 
		\begin{align*}
			&\mathbb{E}\Big[\Big| \sum_{j=0}^{N-1}(\tilde{Y}_{t_j}^n-\mathcal{Y}_j^{\pi})\zeta_j^{Y,\hat{z}}\Big|\Big]\\
			&\leq C \mathbb{E}\Big[\sup_{0\leq i\leq N-1}\mathbb{E}\big[\sup_{0\leq j\leq N-1}|\tilde{Y}_j^n-\mathcal{Y}_j^{\pi}|^4/\mathfrak{F}_{t_i}\big]^{\frac{1}{4}}\mathbb{E}\left[\int_{t_i}^T\big|{Z}_s^n-\hat{Z}_s^n|^2\mathrm{d}s\big/\mathfrak{F}_{t_i}\right]^{\frac{1}{2}}\Big]\\
			&\leq C \mathbb{E}\Big[\sup_{0\leq i\leq N-1}\mathbb{E}\big[\sup_{0\leq j\leq N-1}|\tilde{Y}_j^n-\mathcal{Y}_j^{\pi}|^4/\mathfrak{F}_{t_i}\big]^{\frac{1}{2}}\Big]^{\frac{1}{2}}\mathbb{E}\Big[\sup_{0\leq i\leq N-1}\mathbb{E}\Big[\int_{t_i}^T\big|{Z}_s^n-\hat{Z}_s^n|^2\mathrm{d}s\big/\mathfrak{F}_{t_i}\Big]\Big]^{\frac{1}{2}}\\
			&\leq C \mathbb{E}\Big[\sup_{0\leq i\leq N-1}\mathbb{E}\big[\sup_{0\leq j\leq N-1}|\tilde{Y}_j^n-\mathcal{Y}_j^{\pi}|^4/\mathfrak{F}_{t_i}\big]^{2}\Big]^{\frac{1}{8}} h^{\frac{1}{2}},
		\end{align*}
		where we used Corollary \ref{cor 1.3} to deduce the last inequality. Therefore, applying the Doob's maximal inequality, we obtain  
		\begin{align*}
			\mathbb{E}\Big[\Big| \sum_{j=0}^{N-1}(\tilde{Y}_{t_j}^n-\mathcal{Y}_j^{\pi})\zeta_j^{Y,\hat{z}}\Big|\Big] \leq 4C \mathbb{E}\left[\sup_{0\leq j\leq N-1}|\tilde{Y}_j^n-\mathcal{Y}_j^{\pi}|^8\right]^{\frac{1}{8}} h^{\frac{1}{2}}.
		\end{align*}
	\end{proof}
	
	\section{Conclusion and some perspectives} \label{secconcl}
	\subsubsection*{Summary of findings:} In this paper, we establish for the first time the rate of convergence for the time discretization of BSDEs coupled with non smooth-type forward SDEs. By incorporating more complex behaviors and relaxing traditional regularity assumptions, our work significantly enhances the understanding of forward-backward SDEs (FBSDEs) in cases where the coefficients of the forward process demonstrates irregularities. This advancement fills a crucial gap in the literature, providing a more rigorous framework for analyzing and approximating FBSDEs under challenging conditions. These findings hold particular importance for potential applications that demand precise numerical approximations, such as those in finance, physics, and stochastic control, where the interplay between backward and forward processes is complex and existing methods often prove inadequate. Our approach not only broadens the scope of FBSDE models but also lays the groundwork for future research into more efficient and accurate computational techniques in systems driven by non-smooth dynamics.
	
	\subsubsection*{Numerical scheme-- An open problem:}
	However, the BTZ scheme exposed in this work remains very theoretical since it assumes an accurate computations of the conditional expectations involved in the model. For this reason it would have been interesting to develop a fully implementable numerical scheme by following the works \cite{DelarueMenozzi06,ChaRichou16,Zhou} in our context. The main ingredient consists to use the quantization method which provides an algorithm to compute expectation appearing in \eqref{linBTZ1} by establishing a grid associated with probability based on the distribution of the random variables $(\Delta W_i)$. Precisely, these random variables are approximated by a sequence of centered random variables $(\Delta \hat W_i =\sqrt{\delta t_i} G_M(\frac{\Delta W_i}{\sqrt{\delta t_i}}))$ with discrete support, and $G_M$ stands for the projection operator on the optimal quantization grid for the standard Gaussian distribution with $M$ points in the support and the following bound holds for all $p\geq 1$
	\begin{align}
		\mathbb{E}\left[ |\Delta W_{t_i} - \Delta \hat W_{t_i}|^p  \right]^{{1}/{p}} \leq C_p \sqrt{h_i}M^{-1/d}.
	\end{align}
	Henceforth, this necessitates the development of a numerical scheme for the forward equation \eqref{2.1} using the Markovian quantization method, particularly when the drift term lies outside the standard Lipschitz continuous framework. To the best of our knowledge, no such result currently exists, as it requires a more advanced stability analysis result of the Euler-Maruyama scheme for SDEs with non-smooth coefficients than the one established for instance in the aforementioned Corollary \ref{Cor210} (or equivalently in \cite{GaLing23}, for integrable drift coefficients).
	
	In fact, let us consider the following tho version of the Euler scheme approximating the equation \eqref{2.1} under the assumption (A1)
	\begin{align*}
		\bold{X}_{i+1} &= \bold{X}_{i} + b(t_i,\bold{X}_{i})h_i +\Delta W_{t_i}\\
		\bold{Y}_{i+1} &= \bold{Y}_{i} + b(t_i,\bold{Y}_{i})h_i +\Delta W_{t_i} +\xi_i,
	\end{align*}
	where $\xi_i$ stands for a family of random variables such that for all $i$ we have $\xi_i \in L^{2p}(\Omega)$. The second equation can be viewed as a perturbation of the first one. Standard computations lead to the following for all $p\geq 2$
	\begin{align*}
		\mathbb{E}[\sup_{0\leq k\leq i}|\bold{X}_{k}-\bold{Y}_{k}|^{p}] \leq C(p)|\bold{X}_{0}-\bold{Y}_{0}| + C(p)\mathbb{E}\Bigg[ \sup_{0\leq k\leq i}\Big|\sum_{j=0}^{i-1}\Big(b(t_j,\bold{X}_{j})-b(t_j,\bold{Y}_{j})\Big)h_j\Big|^p\Bigg] + C(p) \mathbb{E} \Big|\sum_{j=0}^{i-1} \xi_j\Big|^p.
	\end{align*} 
	Since the drift $b$ here is not Lipschitz continuous, this prevents us to apply the a Gronwall-type inequality as it was done for instance in \cite{ChaRichou16}, \cite{Zhou} and \cite{Zhang041}. Therefore, by just using the boundedness of $b$ we deduce the following bound 
	\begin{align*}
		\mathbb{E}[\sup_{0\leq k\leq i}|\bold{X}_{k}-\bold{Y}_{k}|^{p}]\leq C(p,T)\Bigg(1+ \mathbb{E} \Big|\sum_{j=0}^{i-1} \xi_j\Big|^p\Bigg)
	\end{align*}

	\begin{lemm}
		Let $\Phi$ be constructible functional with construction $\varphi$ and let $\pi: 0=t_0<t_t\cdots < t_N =T$ be a partition with $\delta t_i =t_{i+1}-t_i$ and $h=\sup_i|\delta t_i| \leq KN^{-1}$ for some $K>0$. We define the following 
	\end{lemm}
	
	\appendix
	\section{Proof of auxiliary results }\label{secauxres1}
	
	\begin{proof}[Proof of Lemma \ref{Lemm6.4}]
		The first assertion follows by induction by noticing that $\delta Z_i =\mathbb{E}[\delta Y_{i+1} H_i/\mathfrak{F}_{t_i}]$ and by the following standard linearization 
		\begin{align*}
			\delta Y_i &= \mathbb{E}\Big[ \delta Y_{i+1} + \delta t_i \Gamma_i \mathbb{E}[\delta Y_{i+1}/\mathfrak{F}_{t_i}]  + \delta t_i \Pi_i \cdot \delta Z_i + \delta t_i \delta G_i / \mathfrak{F}_{t_i} \Big]\\
			&= \mathbb{E}\left[(1+\delta t_i\Gamma_i + \delta t_i \Pi_i\cdot{H}_i)\delta Y_{i+1}/\mathfrak{F}_{t_i}\right] + \mathbb{E}\left[\delta t_i \delta G_i / \mathfrak{F}_{t_i}\right]. 
		\end{align*}
		On the other hand, the previous representation leads to the desired comparison principle. Specifically, observing that $|\Gamma_i| \leq \Lambda_y$ and $|\Pi_i| \leq 2\Lambda_z(n)\ell(0)$, where $\ell$ is a strictly positive function, we deduce from the assumptions of the lemma that, for all $0\leq i \leq N-1$ the following holds
		\begin{align*}
			\delta Y_i = \mathbb{E}\Big[\Upsilon_i\delta Y_{t_N} + \sum_{k=i}^{N-1}\delta t_k\delta G_k \frac{\Upsilon_i}{\Upsilon_k}\Big/\mathfrak{F}_{t_i}\Big]\geq 0.
		\end{align*}
		This concludes the proof.
	\end{proof}
	\begin{proof}[Proof of Lemma \ref{BTZbound}]
		We will compare the solution to the BTZ scheme with the solution $\hat Y_i$, which has terminal value $\|\xi\|_{\infty}$ and drivers $F_i$ given by 
		$F_i(y,z) = \Lambda_0 + \Lambda_y|y|$. SBy induction, we deduce that $\hat Y_i$ is deterministic and the control process given by $\hat Z_i = \mathbb{E}[\hat Y_{i+1} H_i/\mathfrak{F}_{t_i}]= 0$. From the assumption, we observe that 
		\[ G_i(\hat Y_i,\hat Z_i) \leq \Lambda_0 + \Lambda_y|\hat Y_i|= F_i(\hat Y_i,0)= F_i(\hat Y_i,\hat Z_i).\]
		This implies that $Y_i \leq \hat Y_i$ based on the comparison principle mentioned above. On the other hand, the following holds 
		\begin{align*}
			\hat Y_i 
			&\leq  \Big(1+ \Lambda_y\delta t_i \Big)\hat Y_{i+1} + \Lambda_0 \delta t_i\\
			&\leq \prod_{i=0}^{N-1}  \Big(1+ \Lambda_y\delta t_i \Big)\|\xi\|_{\infty} + \sum_{i=0}^{N-1}\Big\{\prod_{j=0}^{i-1}\Big(1+ \Lambda_y\delta t_j \Big)\Big\}\Lambda_0 \delta t_i\\
			&\leq (\|\xi\|_{\infty} + \Lambda_0 T)\exp\left(1+\Lambda_yT\right).
		\end{align*}
		On the other hand, by using a similar argument as above one can obtain that 
		\[\hat Y_i \geq -(\|\xi\|_{\infty} + \Lambda_0 T)\exp\left(1+\Lambda_yT\right).\]
	\end{proof}
	\begin{proof}[Proof of Lemma \ref{BTZbmo}]
		
		We will use a technique similar to that employed in \cite{ChaRichou16} and \cite{Zhou}. Let $\mathcal{M}$ denote the bound of the BTZ scheme $Y_i^{\pi}$ obtained in Lemma \ref{BTZbound}. Define $(Y_i^{(l),\pi}, Z_i^{(l),\pi})_{1 \leq l \leq \kappa}$ as the BTZ scheme associated with the terminal value $\xi^{(l)} = \frac{\xi}{\kappa}$ and the following generator.
		
		\begin{align}
			\hat G_i^{(l)}(y,z)&= \frac{G_i(0,0)}{\kappa} + G_i\Big(y+\sum_{m=1}^{l-1}\mathbb{E}[ Y_{i+1}^{(m),\pi}/\mathfrak{F}_{t_i}],z+\sum_{m=1}^{l-1} Z_{i}^{(m),\pi}\Big)\notag\\
			&\quad-G_i\Big(\sum_{m=1}^{l-1}\mathbb{E}[ Y_{i+1}^{(m),\pi}/\mathfrak{F}_{t_i}],\sum_{m=1}^{l-1} Z_{i}^{(m),\pi}\Big),
		\end{align}
		with the convention that $\sum_{i=1}^0=0$ and $\kappa\in \mathbb{N}$ not depending on $N$. From uniqueness and linearity of the BTZ scheme, we obtain that
		\begin{align*}
			Y_i^{\pi} = \sum_{l=1}^{\kappa}Y_i^{(l),\pi} \quad \text{ and }\quad Z_i^{\pi} = \sum_{l=1}^{\kappa} Z_i^{(l),\pi}.
		\end{align*}
		Moreover, for all $y,y',z,z'$ we have that
		\begin{align*}
			&|\hat G_i^{(l)}(y,z)-\hat G_i^{(l)}(y',z')|\\ &= \Big|G_i\Big(y+\sum_{m=1}^{l-1}\mathbb{E}[ Y_{i+1}^{(m),\pi}/\mathfrak{F}_{t_i}],z+\sum_{m=1}^{l-1} Z_{i}^{(m),\pi}\Big)-G_i\Big(y'+\sum_{m=1}^{l-1}\mathbb{E}[ Y_{i+1}^{(m),\pi}/\mathfrak{F}_{t_i}],z'+\sum_{m=1}^{l-1} Z_{i}^{(m),\pi}\Big)\Big|\\
			&\leq \Lambda_y(1+|z-z'|^{\alpha_0})|y-y'|+ \Lambda_z(n)\ell(|y-y'|)|z-z'|.	
		\end{align*}
		This immediately implies that for all $l\in\{1,\cdots,\kappa\}$ the family of drivers $(\hat G_i^{(l)})_i$ satisfies Assumption \ref{assum56}.  Therefore, applying Lemma \ref{BTZbound} yields that
		\begin{align}\label{bound66}
			\sup_{0\leq i\leq N}|Y_i^{(l),\pi}| \leq\Big(\|\xi^{(l)}\|_{\infty} + \frac{|G_i(0,0)|}{\kappa}\Big)\exp\left(1+\Lambda_yT\right) = \frac{\mathcal{M}}{\kappa}.
		\end{align}
		It is then sufficient to establish that 
		\begin{align}\label{519}
			\mathbb{E}\Big[\sum_{j=i}^{N-1}{\delta t_j}| Z_{j}^{(l),\pi}|^2\Big/\mathfrak{F}_{t_i}\Big] \leq C
		\end{align}
		for all $l\in \{1,\cdots,\kappa\}$ with $C$ a constant independent of $N$. We observe that, the BTZ solution $(Y_i^{(l),\pi}, Z_i^{(l),\pi})_{1\leq l\leq \kappa}$ can be rewritten as
		\begin{align}\label{eqbtz1}
			Y_{i+1}^{(l),\pi}=  Y_{i}^{(l),\pi}- G_i^{(l),\pi}\big(\mathbb{E}[ Y_{i+1}^{(l),\pi}/\mathfrak{F}_{t_i}], Z_{i}^{(l),\pi}\big)\delta t_i + \nu_i,
		\end{align}
		where
		\begin{align*}
			\nu_i = \mathit{B}_i -\frac{\delta t_i}{c_i} Z_{i}^{(l),\pi}\cdot ({H}_i)^{T}, \quad
			\mathit{B}_i =  Y_{i+1}^{(l),\pi} -\mathbb{E}[ Y_{i+1}^{(l),\pi}/\mathfrak{F}_{t_i}].
		\end{align*}
		Thus squaring both sides of \eqref{eqbtz1} and using the definition of $\nu_i$, we obtain that 	
		\begin{align}\label{eq514}
			\frac{1}{2}\frac{\delta t_i}{C_2}|Z_{i}^{(l),\pi}|^2 &\leq \mathbb{E}[|Y_{i+1}^{(l),\pi}|^2/\mathfrak{F}_{t_i}] - | Y_{i}^{(l),\pi}|^2 + 2 Y_{i}^{(l),\pi}\hat G_i^{(l)}\left(\mathbb{E}[ Y_{i+1}^{(l),\pi}/\mathfrak{F}_{t_i}], Z_{i}^{(l),\pi}\right)\delta t_i.
		\end{align}
		Using primarily the growth properties of $G_i$, we obtain for all $l\in\{1,\cdots,\kappa\}$ 
		\begin{align*}
			&|\hat G_i^{(l)}(y,z)| \notag\\
			&\leq \frac{\Lambda_0}{\kappa} +\Lambda_z+  \Lambda_y(1+|z|^{\alpha_0})|y|+ \Lambda_z(1+\ell(|y|))|z|^2+ \Lambda_z\ell(|y|)(l-1)\sum_{m=1}^{l-1}|Z_{i}^{(m),\pi}|^2.
		\end{align*}
		In particular, using the bound \eqref{bound66} and the local boundedness of the function $\ell$, we deduce the existence of a constant $C(\|\ell\circ|Y^{(l),\pi}|\|_{\infty})$ that does not depend on $N$, such that
		\begin{align}\label{eq515}
			&|\hat G_i^{(l)}\left(\mathbb{E}[ Y_{i+1}^{(l),\pi}/\mathfrak{F}_{t_i}], Z_{i}^{(l),\pi}\right)| \notag\\
			&\leq \frac{\Lambda_0 +\Lambda_y\mathcal{M}}{\kappa} +\Lambda_z+  \frac{\Lambda_y\mathcal{M}}{\kappa}|Z_{i}^{(l),\pi}|^{\alpha_0}+ \Lambda_z(1+\|\ell\circ|Y^{(l),\pi}|\|_{\infty})|Z_{i}^{(l),\pi}|^2\notag\\
			&\quad+ C(\|\ell\circ|Y^{(l),\pi}|\|_{\infty})(l-1)\sum_{m=1}^{l-1}|Z_{i}^{(m),\pi}|^2\notag\\
			&\leq \frac{\Lambda_0 +\Lambda_y\mathcal{M}}{\kappa} +\Lambda_z+  C_{\alpha_0}\Big(\frac{\Lambda_y\mathcal{M}}{\kappa}\Big)^{\frac{2}{2-\alpha_0}} + \Lambda_z(1+\frac{\alpha_0}{2}+\|\ell\circ|Y^{(l),\pi}|\|_{\infty})|Z_{i}^{(l),\pi}|^2\notag\\
			&\quad+ C(\|\ell\circ|Y^{(l),\pi}|\|_{\infty})(l-1)\sum_{m=1}^{l-1}|Z_{i}^{(m),\pi}|^2.
		\end{align}
		Plugging \eqref{eq515} into \eqref{eq514}, using an induction argument on $i$ and taking the conditional expectation with respect to $\mathfrak{F}_{t_i}$, we deduce that
		\begin{align*}
			&\Big(	\frac{1}{2C_2}- 2\frac{\mathcal{M}}{\kappa}\Lambda_z(1+\frac{\alpha_0}{2}+\|\ell\circ|Y^{(l),\pi}|\|_{\infty})\Big)	\mathbb{E}\Big[\sum_{j=i}^{N-1}\delta t_j| Z_{j}^{(l),\pi}|^2/\mathfrak{F}_{t_i}\Big]\\
			&\leq  2\frac{\mathcal{M}}{\kappa}\Big(\Big(\frac{\Lambda_0}{\kappa} + \Lambda_z\Big)T 
			+C(\|\ell\circ|Y^{(l),\pi}|\|_{\infty})(l-1)\sum_{j=i}^{N-1}\mathbb{E}\Big[\sum_{m=1}^{l-1}\delta t_j|Z_{j}^{(m),\pi}|^2/\mathfrak{F}_{t_i}\Big]\Big) \\
			& \quad+2\frac{\mathcal{M}^2}{\kappa^2}\Big(1+2\Lambda_y T\Big) + 2\frac{\mathcal{M}}{\kappa} C_{\alpha_0}\Big(\frac{\Lambda_y\mathcal{M}}{\kappa}\Big)^{\frac{2}{2-\alpha_0}}T.
		\end{align*}
		Hence, an appropriate choice of the parameter $\kappa$ ensures the existence of a constant $C$ independent of $N$ such that	
		\begin{align*}
			\mathbb{E}\Big[\sum_{j=i}^{N-1}\delta t_j| Z_{j}^{(l),\pi}|^2/\mathfrak{F}_{t_i}\Big]
			\leq C \Big(1+\sum_{j=i}^{N-1}\mathbb{E}\Big[\sum_{m=1}^{l-1}\delta t_j|Z_{j}^{(m),\pi}|^2/\mathfrak{F}_{t_i}\Big]\Big) .
		\end{align*}
		Therefore \eqref{519} follows from an induction principle on $l$. This concludes the proof.	
	\end{proof}
	\begin{proof}[Proof of Lemma \ref{lemm69}]
		We first observe that the perturbed scheme \eqref{linBTZ3} can be rewritten as follows
		\begin{align}\label{eq 6.16}
			\tilde{Y}_i = \tilde{Y}_{i+1} +\delta t_iG_i(\mathbb{E}[\tilde{Y}_{i+1}/\mathfrak{F}_{t_i}],\tilde{Z}_i) +\zeta_i^{Y} - \frac{\delta t_i}{c_i}(\tilde{Z}_i)^{\bold T}\cdot H_i - \tilde\nu_i,
		\end{align}
		where $\tilde\nu_i$ is an $\mathfrak{F}_{t_{i+1}}\text{-measurable}$ random variable, such that $\mathbb{E}[\tilde{\nu}_i/\mathfrak{F}_{t_i}]=\mathbb{E}[\tilde{\nu}_iH_i/\mathfrak{F}_{t_i}]=0$ and $\mathbb{E}[(\tilde{\nu}_i)^2/\mathfrak{F}_{t_i}]<\infty.$
		Hence, together with \eqref{eq 6.2}, we obtain that
		\begin{align*}
			\tilde \delta Y_i^{\pi} &= \tilde \delta Y_{i+1}^{\pi} + \left(G_i(\mathbb{E}[\tilde{Y}_{i+1}/\mathfrak{F}_{t_i}],\tilde{Z}_i) - G_i(\mathbb{E}[{Y}^{\pi}_{i+1}/\mathfrak{F}_{t_i}],{Z}_i^{\pi})\right) + \zeta_i^{Y} - \frac{\delta t_i}{c_i}(\tilde \delta{Z}_i^{\pi})\cdot H_i - \tilde \delta \nu_i\\
			&=\tilde \delta Y_{i+1}^{\pi} + \delta t_i\tilde{\Gamma}_i\mathbb{E}[\tilde\delta {Y}_{i+1}^{\pi}/\mathfrak{F}_{t_i}] + \delta t_i \tilde{\Pi}_i \tilde \delta {Z}_i^{\pi} + \zeta_i^{Y} - \frac{\delta t_i}{c_i}(\tilde \delta{Z}_i^{\pi})\cdot H_i - \tilde \delta \nu_i\\
			&= \tilde \delta Y_{i+1}^{\pi} + \bold{A}_i - \frac{\delta t_i}{c_i}(\tilde \delta{Z}_i^{\pi})\cdot H_i - \tilde \delta \nu_i,
		\end{align*}
		where $\tilde \delta \nu_i = \nu_i - \tilde \nu_i$. Using the fact that 
		\begin{align*}
			(\tilde \delta Y_{i+1}^{\pi})^2 &= (\tilde \delta Y_{i}^{\pi})^2 + 2\tilde \delta Y_{i}^{\pi}(\tilde \delta Y_{i+1}^{\pi}-\tilde \delta Y_{i}^{\pi})+ (\tilde \delta Y_{i+1}^{\pi}-\tilde \delta Y_{i}^{\pi})^2\\
			&= (\tilde \delta Y_{i}^{\pi})^2 + 2\tilde \delta Y_{i}^{\pi}\Big(\frac{\delta t_i}{c_i}(\tilde \delta{Z}_i^{\pi})\cdot H_i + \tilde \delta \nu_i-\bold{A}_i\Big)+ \Big(\frac{\delta t_i}{c_i}(\tilde \delta{Z}_i^{\pi})\cdot H_i + \tilde \delta \nu_i-\bold{A}_i\Big)^2.
		\end{align*}
		Then by taking the conditional expectation, and from the assumptions of the Lemma, we obtain that
		\begin{align*}
			\mathbb{E}\left[(\tilde \delta Y_{i+1}^{\pi})^2/\mathfrak{F}_{t_i}\right]\geq (\tilde \delta Y_{i}^{\pi})^2 -2\delta t_i \tilde \delta Y_{i}^{\pi} \left(\tilde{\Gamma}_i\mathbb{E}[\tilde\delta {Y}_{i+1}^{\pi}/\mathfrak{F}_{t_i}]+ \tilde{\Pi}_i \tilde \delta {Z}_i^{\pi} \right) -2 \tilde \delta Y_{i}^{\pi} (\zeta_i^{Y}) + \frac{\delta t_i}{c_i}|\tilde \delta{Z}_i^{\pi}|^2.
		\end{align*}
		Hence, by rearranging the terms above, applying Young's inequality, summing over $i$  and taking the expectation, we deduce the existence of a constant $C$ that does not depend on $N$ such that
		\begin{align*}
			\mathbb{E}\Big[\sum_{i=0}^{N-1}(\tilde \delta Y_{i}^{\pi})^2 \delta t_i\Big] &\leq C \Big(\mathbb{E}[(\tilde \delta Y_{N}^{\pi})^2-(\tilde \delta Y_{0}^{\pi})^2] +\mathbb{E} [\sup_{0\leq i\leq N-1}|\tilde \delta Y_{i}^{\pi}|^2] + \mathbb{E} \sum_{i=0}^{N-1}\tilde \delta Y_{i}^{\pi} (\zeta_i^{Y})\Big)\\
			&\quad +  \mathbb{E}\Big[\sum_{i=0}^{N-1}\delta t_i(\tilde \delta Y_{i}^{\pi})^2 |\tilde{\Pi}_i|^2\Big]
		\end{align*} 
		To conclude, we first recall that $|\tilde\Pi_i| \leq \Lambda_z(1+\ell(0)(|\tilde{Z}_i|+|Z_i^{\pi}|)) \in \mathcal{H}_{BMO}$. Hence from the BDG inequality for discrete martingale, we deduce that 
		\begin{align*}
			\mathbb{E}\Big[\sum_{i=0}^{N-1}\delta t_i(\tilde \delta Y_{i}^{\pi})^2 |\tilde{\Pi}_i|^2\big] &\leq \mathbb{E}\Big[\Big(\sum_{i=0}^{N-1}\delta t_i |\tilde{\Pi}_i|^2\Big)^2\Big]^{\frac{1}{2}}\mathbb{E}[\sup_{0\leq i\leq N-1}|\tilde \delta Y_{i}^{\pi}|^4]^{\frac{1}{2}}\\
			&\leq C \mathbb{E}[\sup_{0\leq i\leq N-1}|\tilde \delta Y_{i}^{\pi}|^4]^{\frac{1}{2}}.
		\end{align*}
		The proof is completed.
	\end{proof}

	\begin{proof}[Proof of Lemma \ref{lemm74}]
		Let us notice that 
		\begin{align*}
			\mathbb{E}\Big[\sum_{j=i}^{N-1}|\tilde{Z}_i^n|^2\delta t_j\big/\mathfrak{F}_{t_{i}}\Big] &\leq 2 \mathbb{E}\Big[\sum_{j=i}^{N-1}|\tilde{Z}_j^n-\bar{\bar{Z}}_j^n|^2\delta t_j\big/\mathfrak{F}_{t_{i}}\Big] +2 \mathbb{E}\Big[\sum_{j=i}^{N-1}|\bar{\bar{Z}}_j^n|^2\delta t_j\big/\mathfrak{F}_{t_{i}}\Big]\\
			&= 2(J_1 + J_2),
		\end{align*}
		with 
		\begin{align}\label{eq714}
			\bar{\bar{Z}}_j^n:= \bar{\bar{Z}}^n(t) \quad \forall t\in[t_j,t_{j+1}],\quad	 \bar{\bar{Z}}_j^n:= \mathbb{E}\Big[Y_{t_{j+1}}^n\frac{W_{t_{j+1}}-W_{t_j}}{\delta t_j}/\mathfrak{F}_{t_{j}}\Big].
		\end{align}
		We start with the computation of $J_1$. We observe that 
		\begin{align*}
			|\tilde{Z}_j-\bar{\bar{Z}}_j^n|^2\delta t_j = \mathbb{E}\Big[(Y^n_{t_{j+1}}-Y^n_{t_j})\Big(H_i^R-\frac{W_{t_{j+1}}-W_{t_j}}{\delta t_j}\Big)/\mathfrak{F}_{t_j}\Big]^2\delta t_j.
		\end{align*}
		Therefore, using, for instance, the definition of the coefficients $H^R_i$ given in \eqref{eq79}, we obtain the existence of a constant $C>0$ such that
		\begin{align*}
			\mathbb{E}\Big[\sum_{j=i}^{N-1}|\tilde{Z}_j-\bar{\bar{Z}}_j^n|^2\delta t_j/\mathfrak{F}_{t_{i}}\Big] &\leq C \mathbb{E}\Big[\sum_{j=i}^{N-1} |Y^n_{t_{j+1}}-Y^n_{t_j}|^2/\mathfrak{F}_{t_i}\Big]\\
			&\leq C \mathbb{E}\Big[\sum_{j=i}^{N-1} \Big| \int_{t_j}^{t_{j+1}}g_n(s,X_s,Y_s^n,Z_s^n)\mathrm{d}s-\int_{t_j}^{t_{j+1}}Z_s^n\mathrm{d}W_s\Big|^2 \Big/\mathfrak{F}_{t_i}  \Big]\\
			&\leq C \mathbb{E}\Big[ \Big( \int_{t_i}^{T}|g_n(s,X_s,Y_s^n,Z_s^n)|\mathrm{d}s\Big)^2+\int_{t_i}^{T}|Z_s^n|^2\mathrm{d}s \Big/\mathfrak{F}_{t_i}  \Big].
		\end{align*}
		Here, we used the conditional BDG inequality to obtain the last inequality. The second term on the right-hand side can easily be controlled by the $\mathcal{H}_{BMO}\text{-norm}$ of the process $Z^n$, uniformly in $n$.
		Meanwhile, a careful inspection of the first term is crucial, as we do not have an almost sure bound for the process $Z^n_t$ of the form $\sup_n|Z_t^n| \leq C(1+\sup_t|X_t|)$ as seen in \cite{ChaRichou16} or \cite{Zhou}, for instance.
		Using the growth of $g_n$, the uniform bound on the process $Y^n$ and the local boundedness of the function $\ell$, we observe that 
		\begin{align*}
			\mathbb{E}\Big[ \Big( \int_{t_i}^{T}|g_n(s,X_s,Y_s^n,Z_s^n)|\mathrm{d}s\Big)^2 \Big/\mathfrak{F}_{t_i} \Big]\leq C T^2 +C\mathbb{E}\Big[ \Big( \int_{t_i}^{T}|Z_s^n|^2\mathrm{d}s\Big)^2 \Big/\mathfrak{F}_{t_i} \Big].
		\end{align*} 
		Let $h$ be the Lyapunov function given in Lemma \ref{lyapunov}. Applying the It\^o's formula, we obtain 
		\begin{align}
			\mathrm{d}h(Y_s^n)= \Big(\frac{1}{2}h''(Y_s^n)|Z_s^n|^2-h'(Y_s^n)g_n(s,X_s,Y_s^n,Z_s^n) \Big)\mathrm{d}s +\mathrm{d}M_s^n,
		\end{align}
		where $\mathrm{d}M_s^n$ is a stochastic integral.
		Using the property of the Lyapunov function, we deduce that
		\begin{align*}
			\int_{t_i}^{T} |Z_s^n|^2\mathrm{d}s &\leq h(Y_{T}^n)- h(Y_{t_i}^n)+ kT + \int_{t_i}^{T} \mathrm{d}M_s^n.
		\end{align*}
		Squaring and taking the conditional expectation in the above inequality, we deduce that
		\begin{align*}
			\sup_{n\in\mathbb{N}} \mathbb{E}\Big[\Bigg(\int_{t_i}^{T} |Z_s^n|^2\mathrm{d}s\Bigg)^2/\mathfrak{F}_{t_i}\Big] &\leq C \mathbb{E}\Big[|h(Y_{T}^n)- h(Y_{t_i}^n)|^2+ k^2T^2 + \int_{t_i}^{T}|Z_s^n|^2\mathrm{d}s\big/\mathfrak{F}_{t_i}\Big]\\
			&\leq 2C\|Dh\|_{L^{\infty}}^2\mathbb{E}\Big[|Y_T^n-Y_{t_i}^n|^2/\mathfrak{F}_{t_i}\Big] + Ck^2T^2 + C\|Z^n*W\|_{\bmo}\\
			&\leq C(T)(1+\|Z^n*W\|_{\bmo}),
		\end{align*}
		where, the last inequality follows from the uniform boundedness on the process $Y^n$(see \eqref{77}). We have in particular 
		\begin{align}\label{716}
			\mathbb{E}\Big[ \Big( \int_{t_i}^{T}|g_n(s,X_s,Y_s^n,Z_s^n)|\mathrm{d}s\Big)^2 \Big/\mathfrak{F}_{t_i} \Big]\leq C(T)(1+\|Z^n*W\|_{\bmo}).
		\end{align} 
		Furthermore, there exists a constant $C$ that does not depend on $n$, such that 
		\begin{align*}
			\mathbb{E}\Big[\sum_{j=i}^{N-1}|\tilde{Z}_j-\bar{\bar{Z}}_j^n|^2\delta t_j/\mathfrak{F}_{t_{i}}\Big] \leq  C(1 + \|Z*W\|_{\bmo}).
		\end{align*}
		We now turn to the control of $J_2$. We observe that 
		\begin{align*}
			\mathbb{E}\Big[\sum_{j=i}^{N-1}|\bar{\bar{Z}}_j^n|^2\delta t_j/\mathfrak{F}_{t_{i}}\Big] &\leq 2\mathbb{E}\Big[\sum_{j=i}^{N-1}|\hat{Z}_j^n-\bar{\bar{Z}}_j^n|^2\delta t_j/\mathfrak{F}_{t_{i}}\Big] +2 \mathbb{E}\Big[\sum_{j=i}^{N-1}|\hat{Z}_j|^2\delta t_j/\mathfrak{F}_{t_{i}}\Big],
		\end{align*}
		where 
		\begin{align}\label{eq715}
			\hat{Z}_j^n:= \hat{Z}^n(t) \quad \forall t\in[t_j,t_{j+1}],\quad	\hat{Z}_j^n = \frac{1}{\delta t_j}\mathbb{E}\Big[\int_{t_j}^{t_{j+1}}Z_s^n\mathrm{d}s\Big/\mathfrak{F}_{t_j}\Big].
		\end{align}
		Similarly to \cite[Lemma 2.1]{ChaRichou16}, by using Jensen's inequality and the tower property, we obtain that
		\begin{align}
			\mathbb{E}\Big[\sum_{j=i}^{N-1}|\hat{Z}_j|^2\delta t_j/\mathfrak{F}_{t_{i}}\Big]\leq  C(1+\sup_n\|Z^n*W\|_{\bmo}).
		\end{align}
		On the other hand, from the definitions of the processes $\hat{Z}^n$ and $\bar{\bar{Z}}^n$, respectively, we obtain that 
		\begin{align*}
			\mathbb{E}\Big[\sum_{j=i}^{N-1}|\hat{Z}_j^n-\bar{\bar{Z}}_j^n|^2\delta t_j/\mathfrak{F}_{t_{i}}\Big]&= \mathbb{E}\Big[\sum_{j=i}^{N-1}\Big|\mathbb{E}\Big[\int_{t_j}^{t{_{j+1}}}g_n(s,X_s.Y_s^n,Z_s^n)\mathrm{d}s\frac{W_{t_{j+1}}-W_{t_j}}{\delta t_j}/\mathfrak{F}_{t_j}\Big]\Big|^2\delta t_j\Big/\mathfrak{F}_{t_{i}}\Big]\\
			&\leq \mathbb{E}\Big[\Big(\int_{t_i}^{T}|g_n(s,X_s.Y_s^n,Z_s^n)|\mathrm{d}s\Big)^2\Big/\mathfrak{F}_{t_{i}}\Big],
		\end{align*}
		where, the inequality follows from the the Cauchy inequality and the tower property. Therefore, from \eqref{716}, there exists a constant independent of $n$ such that
		\begin{align*}
			\mathbb{E}\Big[\sum_{j=i}^{N-1}|\hat{Z}_j^n-\bar{\bar{Z}}_j^n|^2\delta t_j/\mathfrak{F}_{t_{i}}\Big]\leq C(1+\sup_n\|Z^n*W\|_{\bmo}).
		\end{align*} 
		This concludes the proof.
	\end{proof}
	
	\bibliographystyle{plain}
	\bibliography{Biblio1,bibliography,Biblio}
\end{document}